\documentclass[12pt]{amsart}

\usepackage{ucs}

\usepackage{amssymb}
\usepackage{amsthm}
\usepackage{amsmath}
\usepackage{latexsym}
\usepackage[cp1251]{inputenc}
\usepackage{graphicx}
\usepackage{wrapfig}
\usepackage{caption}
\usepackage{subcaption}
\usepackage{indentfirst}
\usepackage[left=2.5cm,right=2.5cm,top=2.5cm,bottom=2.5cm,bindingoffset=0cm]{geometry}
\usepackage{enumerate}
\usepackage{makecell}

\DeclareMathOperator{\aut}{Aut}

\DeclareMathOperator{\cay}{Cay}
\DeclareMathOperator{\cyc}{Cyc}

\DeclareMathOperator{\id}{id}

\DeclareMathOperator{\orb}{Orb}

\DeclareMathOperator{\rk}{rk}

\DeclareMathOperator{\Span}{Span}

\DeclareMathOperator{\sym}{Sym}
\DeclareMathOperator{\rad}{rad}

\DeclareMathOperator{\Hol}{Hol}

\DeclareMathOperator{\alt}{Alt}
\DeclareMathOperator{\divv}{Div}

\def\tm#1{\item[{\rm (#1)}]}

\makeatletter 
\def\@seccntformat#1{\csname the#1\endcsname. } 
\def\@biblabel#1{#1.} 

\newcommand{\overbar}[1]{\mkern 1.5mu\overline{\mkern-1.5mu#1\mkern-1.5mu}\mkern 1.5mu}

\title[Classification of abelian Schur groups I]{Classification of abelian Schur groups I}

\author{Grigory Ryabov}

\address{Sobolev Institute of Mathematics, Novosibirsk, Russia}
\email{gric2ryabov@gmail.com}

\thanks{The author was supported by the state contract of the Sobolev Institute of Mathematics (project number FWNF-2026-0011).}

\date{}

\newtheorem{prop}{Proposition}[section]

\newtheorem*{claim1}{Claim~1}
\newtheorem*{claim2}{Claim~2}
\newtheorem*{claim3}{Claim~3}
\newtheorem*{claim4}{Claim~4}

\newtheorem{lemm}[prop]{Lemma}
\newtheorem{theo}[prop]{Theorem}

\newtheorem*{prob}{Problem (P\"{o}schel, 1974)}
\newtheorem{corl}[prop]{Corollary}

\theoremstyle{definition}

\newtheorem{rem}[prop]{Remark}

\begin{document}

\begin{abstract}
A finite group $G$ is called a \emph{Schur} group if every Schur ring over $G$ is \emph{schurian}, i.e. associated in a natural way with a subgroup of the symmetric group~$\sym(G)$ that contains all right translations of~$G$. The list of all possible abelian Schur groups was obtained by Evdokimov, Kov\'acs, and Ponomarenko in 2016.  In two papers, we complete a classification of abelian Schur groups. In the present paper, we study schurity of several groups from the list. First, we prove that a direct product of the elementary abelian group of order~4 and a cyclic group, whose order is an odd prime power or a product of two distinct odd primes, is a Schur group. Second, we establish nonschurity of some other groups from the list.

\end{abstract}

\maketitle
\section{Introduction}

Let $G$ be a finite group. A subring of the group ring $\mathbb{Z}G$ is called an \emph{$S$-ring} (a \emph{Schur ring}) over~$G$ if it is a free $\mathbb{Z}$-module spanned by a partition of $G$, closed under taking inverse and containing the identity element $e$ of $G$ as a class. The notion of $S$-ring was introduced by Schur~\cite{Schur}. Using the $S$-ring approach, Schur proved that every primitive permutation group containing a regular cyclic subgroup of composite order is $2$-transitive which generalizes the Burnside theorem~\cite[p.~339]{Burn}. The general theory of $S$-rings was developed by Wielandt (see~\cite[Chapter~IV]{Wi}). To date, $S$-rings are realized as a powerful tool for studying permutation groups containing regular subgroups, combinatorial Cayley objects, especially, isomorphisms of them, and representations of finite groups. For more details on $S$-rings and their applications, we refer the reader to~\cite[Section~2.4]{CP} and~\cite{MP0}.

To study permutation groups having regular subgroups, Schur used $S$-rings whose partition of the underlying group consists of the orbits of a one-point stabilizer of an appropriate permutation group (see Section~2.3 for the exact definition). As was shown by Wielandt~\cite{Wi}, not all $S$-rings can be constructed in this way. An $S$-ring is said to be \emph{schurian} if it can~\cite{Po}. The schurity (the property to be schurian) is one of the crucial properties of $S$-rings, closely related to the isomorphism problem for Cayley graphs (see~\cite[Section~5.3]{MP0} for details).

A finite group $G$ is defined to be a \emph{Schur} group if every $S$-ring over $G$ is schurian~\cite{Po}. The main problem considered in the present paper is the following one.

\begin{prob}
Determine all Schur groups.
\end{prob}

In general, the above problem seems to be hard, in particular, because the number of $S$-rings over a given group can be exponential in the order of the group. There are several results on schurity of nonabelian groups~\cite{MP,PV,Ry1,Ry3,Ry4}. However, most of the results on the schurity problem are concerned with abelian groups. The first result on schurity of abelian groups was obtained in paper~\cite{Po}. Namely, it was proved that cyclic $p$-groups of odd order are Schur. This result was extended to the case $p=2$ in paper~\cite{GNP}. In paper~\cite{KP}, the schurity of cyclic groups of order~$pq$, where $p$ and $q$ are distinct primes, was proved. The above results were used later for solving the isomorphism problem for Cayley graphs over the corresponding groups~\cite{KP0,KP}.

A complete classification of cyclic Schur groups was obtained in paper~\cite{EKP1}. Namely, it was proved that a cyclic group of order~$n$ is Schur if and only if $n$ belongs to one of the following families of integers:
$$p^k,~pq^k,~2pq^k,~pqr,~2pqr,$$
where $p$, $q$, and $r$ are distinct primes and $k\geq 0$ is an integer.

Investigations of schurity of abelian groups was continued in paper~\cite{EKP2}. At first, necessary and sufficient condition of schurity for elementary abelian groups was obtained there. Namely, an elementary abelian noncyclic group of order~$n$ is Schur if and only if
$$n\in\{4,8,9,16,27,32\}.$$
Further, it was shown that an abelian Schur group which is neither cyclic nor elementary abelian belongs to one of the explicitly given infinite families of groups. Given a positive integer~$n$, the cyclic and elementary abelian groups of order~$n$ are denoted by $C_n$ and $E_n$, respectively.

\begin{theo}\cite[Theorem~1.3]{EKP2}\label{main0}
An abelian Schur group, which is neither cyclic nor elementary abelian, is isomorphic to a group from one of the following nine families:
\begin{enumerate}
\tm{1} $C_2 \times C_{2^k}$, $C_{2p}\times C_{2^k}$, $E_4 \times C_{p^k}$, $E_4 \times C_{pq}$, $E_{16}\times C_p$,

\tm{2} $C_3 \times C_{3^k}$, $C_{6}\times C_{3^k}$, $E_9\times C_{q}$, $E_9 \times C_{2q}$,
\end{enumerate}
where $p$ and $q$ are distinct primes, $p\neq 2$, and $k\geq 1$ is an integer.
\end{theo}

In the same paper, it was proved that the groups $E_4 \times C_p$, where $p$ is an odd prime, are Schur groups. A schurity property of the groups $C_2\times C_{2^k}$ and $C_3 \times C_{3^k}$ was established in papers~\cite{MP} and~\cite{Ry2}, respectively. In paper~\cite{PR}, it was checked that $E_9 \times C_{q}$ is a Schur group for every prime~$q$.

In two papers, we study the shurity problem for all groups from Theorem~\ref{main0} whose schurity is unknown. In the present paper, we deal with the most laborious case, namely, with the groups $E_4 \times C_{p^k}$ and $E_4 \times C_{pq}$. We prove the following theorem.

\begin{theo}\label{main1}
The groups $E_4\times C_{p^k}$ and $E_4 \times C_{pq}$, where $p$ and $q$ are odd primes and $k\geq 1$, are Schur groups.
\end{theo}

To prove Theorem~\ref{main1}, we characterize all $S$-rings over the groups from this theorem. 

\begin{theo}\label{main15}
Every nontrivial $S$-ring over $E_4\times C_{p^k}$ or $E_4 \times C_{pq}$, where $p$ and $q$ are odd primes and $k\geq 1$, is cyclotomic or a nontrivial tensor or generalized wreath product. 
\end{theo}

A similar statement was proved for an arbitrary cyclic group in papers~\cite{LM1,LM2} and for some other abelian groups close to cyclic groups in papers~\cite{EKP2,MP0,PR,Ry2}. On the other hand, the above statement does not hold, e.g., for elementary abelian groups (see~\cite{Wi}). It seems interesting to characterize all groups whose all nontrivial $S$-rings are cyclotomic or nontrivial tensor or generalized wreath products (see Subsection~2.3 and Section~3 for the definitions).

Perhaps, the most surprising fact concerning the groups from Theorem~\ref{main0} is that not all of them are Schur groups. This is demonstrated by two theorems below.

\begin{theo}\label{main2}
The group $C_{2p}\times C_{2^k}$, where $p$ is an odd prime and $k\geq 3$, is not a Schur group.
\end{theo}

Since the class of Schur groups is closed under taking subgroups, it suffices to prove Theorem~\ref{main2} only for $k=3$. It should be noted that if $k=1,2$, then the group from Theorem~\ref{main2} is a Schur group. Indeed, the schurity of the group $E_4\times C_p$ (case $k=1$) follows from~\cite[Theorem~1.4]{EKP2} as was mentioned before, whereas the schurity of the group $C_{2p}\times C_4$ (case $k=2$) will be verified in the subsequent paper together with the schurity of other abelian groups of order $8p$, where $p$ is an odd prime.

One more family of non-Schur groups from Theorem~\ref{main0} is presented by the theorem below.

\begin{theo}\label{main3}
Let $p$ be an odd prime. The group $E_{16} \times C_p$ is a Schur group if and only if $p=3$.
\end{theo}

In the proofs of Theorems~\ref{main2} and~\ref{main3}, we construct new nonschurian $S$-rings over the groups $C_{2p}\times C_{8}$ and $E_{16} \times C_p$, respectively. These nonschurian $S$-rings are generalized wreath products of $S$-rings over proper subgroups of the above groups (see Sections~9 and~10). The nonschurian $S$-ring over $C_{2p}\times C_{8}$ from Section~9 in case $p=3$ was found by computer calculations using the package~\cite{GAP} and appeared in paper~\cite{Ziv}.

We finish the introduction with a brief outline of the paper. Section~2 contains a necessary background of $S$-rings, especially, of the multiplier theorems and isomorphisms. In Section~3, we provide an information on three constructions of $S$-rings, namely, tensor, generalized wreath, and star products. Several required facts on dual $S$-rings over the group of all irreducible complex characters of an abelian group are given in Section~4. Recent results on $S$-rings over an abelian group having a Sylow subgroup of prime order are provided in Section~5. In Section~6, we give a background of $S$-rings over cyclic groups, especially, over cyclic $p$-groups. In Section~7, we obtain a description of $S$-rings over the groups $E_4\times C_{p^k}$ and $E_4 \times C_{pq}$ (Theorem~\ref{e4cn}) that implies Theorem~\ref{main15}. In Sections~8,~9, and~10, we prove Theorems~\ref{main1},~\ref{main2} and~\ref{main3}, respectively.

\vspace{2mm}

\noindent {\bf Notation.}

\vspace{2mm}

\noindent The identity element and the set of all nonidentity elements of a group $G$ are denoted by $e$ and $G^\#$, respectively.
 
\vspace{2mm}

\noindent The projections of $X\subseteq G_1\times G_2$ to $G_1$ and $G_2$ are denoted by $X_{G_1}$ and $X_{G_2}$, respectively.

\vspace{2mm}

\noindent If $X\subseteq G$, then the element $\sum_{x\in X} {x}$ of the group ring $\mathbb{Z}G$ is denoted by $\underline{X}$.

\vspace{2mm}

\noindent The set $\{x^{-1}:~x\in X\}$ is denoted by $X^{-1}$.

\vspace{2mm}

\noindent The subgroup of $G$ generated by $X$ is denoted by $\langle X\rangle$; we also set $\rad(X)=\{g\in G:\ gX=Xg=X\}$.

\vspace{2mm}

\noindent  The set $\{(g,xg):~x\in X,~g\in G\}\subseteq G^2$ is denoted by $r(X)$.

\vspace{2mm}

\noindent If $m\in \mathbb{Z}$, then the set $\{x^m:~x\in X\}$ is denoted by $X^{(m)}$.

\vspace{2mm}

\noindent If $G$ is abelian and $m\in \mathbb{Z}$ coprime to~$|G|$, then the automorphism of $G$ which maps every $g\in G$ to $g^m$ is denoted by~$\sigma_m$. 

\vspace{2mm}

\noindent  The group of all permutations of a set $\Omega$ is denoted by $\sym(\Omega)$.

\vspace{2mm}

\noindent  The subgroup of $\sym(G)$ induced by all right multiplications of $G$ is denoted by $G_r$.

\vspace{2mm}

\noindent  If $f\in \sym(\Omega)$ and $\Delta\subseteq \Omega$ is such that $\Delta^f=\Delta$, then the permutation of $\Delta$ induced by~$f$ is denoted $f^{\Delta}$. 

\vspace{2mm}

\noindent For a set $\Delta\subseteq \sym(G)$ and a section $S=U/L$ of $G$ we set 
$$\Delta^S=\{f^S:~f\in \Delta,~S^f=S\},$$
where $S^f=S$ means that $f$ permutes the $L$-cosets in $U$ and $f^S$ denotes the bijection of $S$ induced by $f$.

\vspace{2mm}

\noindent If $K\leq \sym(\Omega)$ and $\alpha\in \Omega$, then the set of all orbits of $K$ on $\Omega$ and the stabilizer of $\alpha$ in $K$ are denoted by $\orb(K,\Omega)$ and $K_{\alpha}$, respectively.

\vspace{2mm}



\noindent The set of all positive divisors of $n\in \mathbb{Z}$ is denoted by~$\divv(n)$. 

\vspace{2mm}

\noindent Given a cyclic group $D$ of order~$n$ and $l\in\divv(n)$, the subgroup of $D$ of order~$l$ is denoted by $D_l$ and the set of all generators of $D_l$ is denoted by $D_l^*$.

\vspace{2mm}

\noindent If $X\subseteq H\times D$, where $D$ is a cyclic group of order~$n$, $h\in H$, and $l\in\divv(n)$, then put $X_h=h^{-1}X\cap D$ and $X_{h,l}=h^{-1}X\cap D_l^*$.

\section{$S$-rings}

In this section, we provide a necessary background of $S$-rings. In general, we follow~\cite{EKP2,MP,Ry2}, where most of the material is contained.

\subsection{Definitions and basic facts}

A subring  $\mathcal{A}\subseteq \mathbb{Z} G$ is called an \emph{$S$-ring} (a \emph{Schur} ring) over $G$ if there exists a partition $\mathcal{S}=\mathcal{S}(\mathcal{A})$ of~$G$ such that:

$(1)$ $\{e\}\in\mathcal{S}$;

$(2)$  if $X\in\mathcal{S}$, then $X^{-1}\in\mathcal{S}$;

$(3)$ $\mathcal{A}=\Span_{\mathbb{Z}}\{\underline{X}:\ X\in\mathcal{S}\}$.

\noindent The elements of $\mathcal{S}$ are called the \emph{basic sets} of $\mathcal{A}$. The number of basic sets is called the \emph{rank} of $\mathcal{A}$ and denoted by~$\rk(\mathcal{A})$. Clearly, the group ring $\mathbb{Z}G$ is an $S$-ring. If $|G|\neq 1$, then the partition $\{\{e\},G^\#\}$ defines the $S$-ring $\mathcal{T}_G$ of rank~$2$ over~$G$. It is easy to check that if $X,Y\in \mathcal{S}(\mathcal{A})$, then $XY\in \mathcal{S}(\mathcal{A})$ whenever $|X|=1$ or $|Y|=1$. The $S$-ring $\mathcal{A}$ is called \emph{symmetric} if $X=X^{-1}$ for every $X\in \mathcal{S}(\mathcal{A})$ and \emph{antisymmetric} if for every $X\in \mathcal{S}(\mathcal{A})$ the equality $X=X^{-1}$ implies that $X=\{e\}$.

Let $X,Y\in\mathcal{S}$. If $Z\in \mathcal{S}$, then by Condition~$(3)$, the number of distinct representations of $z\in Z$ in the form $z=xy$ with $x\in X$ and $y\in Y$ does not depend on the choice of $z\in Z$. Denote this number by $c^Z_{XY}$. One can see that 
$$\underline{X}\cdot\underline{Y}=\sum_{Z\in \mathcal{S}(\mathcal{A})}c^Z_{XY}\underline{Z}.$$
Therefore the numbers  $c^Z_{XY}$ are the structure constants of $\mathcal{A}$ with respect to the basis $\{\underline{X}:\ X\in\mathcal{S}\}$.

A set $T \subseteq G$ is called an \emph{$\mathcal{A}$-set} if $\underline{T}\in \mathcal{A}$. If $T$ is an $\mathcal{A}$-set, then put 
$$\mathcal{S}(\mathcal{A})_T=\{X\in \mathcal{S}(\mathcal{A}):~X\subseteq T\}.$$ A subgroup $H \leq G$ is called an \emph{$\mathcal{A}$-subgroup} if $H$ is an $\mathcal{A}$-set. The $S$-ring $\mathcal{A}$ is called \emph{primitive} if there is no a nontrivial proper $\mathcal{A}$-subgroup of $G$ and \emph{imprimitive} otherwise. One can verify that for every $\mathcal{A}$-set $X$, the groups $\langle X \rangle$ and $\rad(X)$ are $\mathcal{A}$-subgroups. The set of all $\mathcal{A}$-subgroups of $G$ is denoted by $\mathcal{H}(\mathcal{A})$. It is easy to see that if $X\subseteq \rad(Y)$, then
\begin{equation}\label{inrad0}
\underline{X}\cdot \underline{Y}=\underline{Y}\cdot \underline{X}=|X|\underline{Y}.
\end{equation}

\begin{lemm}\cite[Lemma~2.1]{EKP2}\label{intersection}
Let $\mathcal{A}$ be an $S$-ring over a group $G$, $H$ an $\mathcal{A}$-subgroup of $G$, and $X \in \mathcal{S}(\mathcal{A})$. Then the number $|X\cap Hx|$ does not depend on $x\in X$.
\end{lemm}

\vspace{2mm}

\begin{lemm}~\cite[Theorem~2.6]{MP}\label{separ}
Let $X$ be a basic set of an $S$-ring $\mathcal{A}$ over a group $G$. Suppose that $H\leq \rad(X\setminus H)$ for some subgroup $H$ of $G$ such that $X \cap H\neq \varnothing$ and $X\setminus H \neq \varnothing$. Then $\rad(X)\leq H$ and $X=\langle X \rangle \setminus \rad(X)$.
\end{lemm}

Let $\{e\}\leq L \unlhd U\leq G$. A section $U/L$ is called an \emph{$\mathcal{A}$-section} if $U$ and $L$ are $\mathcal{A}$-subgroups. If $S=U/L$ is an $\mathcal{A}$-section, then the module
$$\mathcal{A}_S=\Span_{\mathbb{Z}}\left\{\underline{X}^{\pi}:~X\in\mathcal{S}(\mathcal{A}),~X\subseteq U\right\},$$
where $\pi:U\rightarrow U/L$ is the canonical epimorphism, is an $S$-ring over $S$.

\subsection{Multiplier theorems}

Let $G$ be an abelian group. If $X,Y\subseteq G$ are such that $Y^{(m)}=X$ for some $m\in \mathbb{Z}$ coprime to~$|G|$, then we say that $X$ and $Y$ are \emph{rationally conjugate}. In this case, if $X\in \orb(K,G)$ for some $K\leq \aut(G)$, then $Y\in \orb(K,G)$. The following two statements are known as the first and second Schur theorems on multipliers (see~\cite[Theorem~23.9, (a)-(b)]{Wi}).

\begin{lemm} \label{burn}
Let $\mathcal{A}$ be an $S$-ring over an abelian group $G$. Then $X^{(m)}\in \mathcal{S}(\mathcal{A})$ for every $X\in \mathcal{S}(\mathcal{A})$ and every $m\in \mathbb{Z}$ coprime to~$|G|$.
\end{lemm}

\begin{lemm} \label{sch}
Let $\mathcal{A}$ be an $S$-ring over an abelian group $G$, $p$ a prime divisor of $|G|$, and $H=\{g\in G:g^p=e\}$. Then for every $\mathcal{A}$-set $X$, the set 
$$X^{[p]}=\{x^p:x\in~X,~|X\cap Hx|\not\equiv 0\mod p\}$$ 
is an $\mathcal{A}$-set.
\end{lemm}

Below, we deduce from Lemma~\ref{burn} one more lemma which will be used further in the proof of Theorem~\ref{main1}.

\begin{lemm}\label{orbit}
Let $G=H\times D$, where $H$ is an abelian group and $D\cong C_n$ with $n$ coprime to $|H|$, $\mathcal{A}$ an $S$-ring over~$G$, and $X\in \mathcal{S}(\mathcal{A})$. Then there exists $K^D\leq \aut(D)$ satisfying the following: $X_{h,l}=h^{-1}X\cap D_l^*\in \orb(K^D,D)$ for all $h\in H$ and $l\in \divv(n)$ such that $X_{h,l}\neq \varnothing$.
\end{lemm}

\begin{proof}
Since $|H|$ and $n$ are coprime, $\aut(G)=\aut(H)\times \aut(D)$. Let $K$ be the setwise stabilizer of $X$ in the group $\{\id_H\}\times \aut(D)\leq \aut(G)$ and $h\in H$ and $l\in \divv(n)$ be such that $X_{h,l}\neq \varnothing$. By the definition of $K$, the set $X_{h,l}$ is $K$-invariant. On the other hand, for all $x,y\in X_{h,l}$, there is an integer $m$ coprime to~$n$ such that $x^{\sigma_m}=x^m=y$ and $m\equiv 1 \mod |H|$ and hence $(hx)^m=hy$. The latter implies that $X^{(m)}\cap X\neq \varnothing$. So $X^{(m)}=X$ by Lemma~\ref{burn}. Therefore $\sigma_m\in K$ and consequently $K$ is transitive on $X_{h,l}$. Thus, $X_{h,l}$ is an orbit of~$K$ and hence of the group $K^D$ induced by $K$ on $D$. 
\end{proof}

\subsection{Isomorphisms and schurity}  

Put $\mathcal{R}(\mathcal{A})=\{r(X):~X\in \mathcal{S}(\mathcal{A})\}$. Let $\mathcal{A}^\prime$ be an $S$-ring over a group $G^\prime$. A bijection $f$ from $G$ to $G^\prime$ is called a \emph{(combinatorial) isomorphism} from $\mathcal{A}$ to $\mathcal{A}^\prime$ if $\mathcal{R}(\mathcal{A})^f=\mathcal{R}(\mathcal{A}^\prime)$, where $\mathcal{R}(\mathcal{A})^f=\{r(X)^f:~X\in \mathcal{S}(\mathcal{A})\}$ and $r(X)^f=\{(g^f,h^f):~(g,h)\in r(X)\}$. If there exists an isomorphism from $\mathcal{A}$ to $\mathcal{A}^\prime$, then $\mathcal{A}$ and $\mathcal{A}^\prime$ are said to be \emph{isomorphic}.

A bijection $f\in\sym(G)$ is defined to be a \emph{(combinatorial) automorphism} of $\mathcal{A}$ if $r(X)^f=r(X)$ for every $X\in \mathcal{S}(\mathcal{A})$. The set of all automorphisms of $\mathcal{A}$ forms the group called the \emph{automorphism group} of $\mathcal{A}$ and denoted by $\aut(\mathcal{A})$. One can see that $\aut(\mathcal{A})\geq G_r$. If $f\in \aut(\mathcal{A})$, then 
\begin{equation}\label{aut}
(Xy)^f=Xy^f
\end{equation}
for every $X\in \mathcal{S}(\mathcal{A})$ and $y\in G$. If $H$ is an $\mathcal{A}$-subgroup of $G$, then the set of all right $H$-cosets is an imprimitivity system of~$\aut(\mathcal{A})$.

Let $K$ be a subgroup of $\sym(G)$ containing $G_{r}$. Schur proved in paper~\cite{Schur} that the $\mathbb{Z}$-submodule
$$V(K,G)=\Span_{\mathbb{Z}}\{\underline{X}:~X\in \orb(K_e,G)\}$$
is an $S$-ring over $G$. An $S$-ring $\mathcal{A}$ over $G$ is called \emph{schurian} if $\mathcal{A}=V(K,G)$ for some $K\leq \sym(G)$ with $K\geq G_{r}$. One can see that $\mathcal{T}_G=V(K,G)$ for every $2$-transitive group $K\leq \sym(G)$ containing $G_r$ and hence $\mathcal{T}_G$ is schurian. The group $G$ is called a \emph{Schur} group if every $S$-ring over $G$ is schurian.

Let $\mathcal{A}=V(K,G)$ for some $K\leq \sym(G)$ such that $K\geq G_r$. One can see that $\aut(\mathcal{A})\geq K$. If $S$ is an $\mathcal{A}$-section, then $\mathcal{A}_S=V(K^S,S)$. This yields that $\mathcal{A}_S$ is also schurian. Therefore a section (in particular, a subgroup) of a Schur group is also a Schur group.

One can verify that $\mathcal{A}$ is schurian if and only if 
$$\mathcal{A}=V(\aut(\mathcal{A}),G)$$
or, equivalently, $\mathcal{S}(\mathcal{A})=\orb(\aut(\mathcal{A})_e,G)$.

The $S$-ring $\mathcal{A}$ is said to be \emph{normal} if $G_r$ is normal in $\aut(\mathcal{A})$ or, equivalently, 
$$\aut(\mathcal{A})\leq \Hol(G)=G_r\rtimes \aut(G).$$ 

Let $K\leq \aut(G)$. Then $\orb(K,G)$ forms a partition of $G$ that defines an $S$-ring $\mathcal{A}$ over~$G$. In this case, $\mathcal{A}$ is called \emph{cyclotomic} and denoted by $\cyc(K,G)$. One can see that $\mathcal{A}=V(G_rK,G)$. So every cyclotomic $S$-ring is schurian. If $S$ is an $\mathcal{A}$-section, then $\mathcal{A}_S=\cyc(K^S,G)$, i.e. $\mathcal{A}_S$ is also cyclotomic. Note that a normal $S$-ring can be nonschurian and hence noncyclotomic as well as a cyclotomic $S$-ring can be nonnormal. However, if a normal $S$-ring is schurian, then it is also cyclotomic.

\subsection{Minimal $S$-rings}

Two permutation groups $K_1$ and $K_2$ on a set $\Omega$ are called \emph{$2$-equivalent} if $\orb(K_1,\Omega^2)=\orb(K_2,\Omega^2)$ (here we assume that $K_1$ and $K_2$ act on $\Omega^2$ componentwise). In this case, we write $K_1\approx_2 K_2$. The relation $\approx_2$ is an equivalence relation on the set of all subgroups of $\sym(\Omega)$. Every $\approx_2$-equivalence class has the unique maximal element and may have more than one minimal elements with respect to inclusion (for the latter, see Example~$1$). Given $K\leq \sym(\Omega)$, the unique maximal element from the class containing $K$ is called the \emph{$2$-closure} of $K$ and denoted by $K^{(2)}$. 

Let $K\leq \sym(G)$ with $K\geq G_r$ and $\mathcal{A}=V(K,G)$. Then $K^{(2)}=\aut(\mathcal{A})$ and $V(K^\prime,G)=\mathcal{A}$ if and only if $K^\prime \approx_2 K$ for every $K^\prime\leq \sym(G)$ with $K^\prime\geq G_r$. A schurian $S$-ring $\mathcal{A}$ over $G$ is said to be \emph{$2$-minimal} if
$$\{K\leq \sym(G):~K\geq G_r,~K\approx_2 \aut(\mathcal{A})\}=\{\aut(\mathcal{A})\}.$$
The set of all minimal with respect to inclusion permutation groups $K^\prime\leq \sym(G)$ such that $K^\prime\geq G_r$ and $V(K^\prime,G)=\mathcal{A}$ is denoted by $\mathcal{K}^{\min}(\mathcal{A})$. Note that $\mathcal{A}$ is $2$-minimal if and only if $\aut(\mathcal{A})$ is $2$-isolated in the sense of~\cite{KK}. 

\hspace{5mm}

\noindent \textbf{Example~1.} This example demonstrates that $\mathcal{K}^{\min}(\mathcal{A})$ may contain more than one element. Let $p\geq 5$ be an odd prime, $G\cong C_p$, and $\mathcal{A}=\mathcal{T}_G$. Let $K_1=\Hol(G)$. Clearly, $K_1\geq G_r$ and $\mathcal{A}=V(K_1,G)$. One can see that the one-point stabilizer $(K_1)_e=\aut(G)$ acts regularly on $G^\#$. So $K_1\in \mathcal{K}^{\min}(\mathcal{A})$.

Let $K_2=\alt(G)$. Since $p$ is odd, every element from $G_r$ is an even permutation of $G$ and hence $K_2\geq G_r$. The group $K_2$ is $(p-2)$-transitive. Together with $p\geq 5$, this implies that $K_2$ is $2$-transitive and consequently $\mathcal{A}=V(K_2,G)$. As $p$ is odd, a generator of $\aut(G)$ is an odd permutation of $G$. So $K_1\nleq K_2$. Therefore $K_2^\prime \neq K_1$, where $K_2^\prime \in \mathcal{K}^{\min}(\mathcal{A})$ with $K_2^\prime\leq K_2$. Thus, $|\mathcal{K}^{\min}(\mathcal{A})|\geq 2$. 

\hspace{5mm}

The next lemma and remark after it easily follow from computer calculations using the package~\cite{GAP}.

\begin{lemm}\label{2minsmall}
Let $\mathcal{A}$ be an $S$-ring over a group $G$. If $|G|\leq 3$ or $|G|=4$ and $\mathcal{A}\neq \mathcal{T}_G$, then $\mathcal{A}$ is $2$-minimal. 
\end{lemm}

\begin{rem}\label{2minsmallrem}
Let $|G|=4$. Then $\aut(\mathcal{T}_G)=\sym(G)\cong \sym(4)$ and $\mathcal{T}_G=V(\sym(G),G)=V(\alt(G),G)$. This implies that $\mathcal{T}_G$ is not $2$-minimal. The group $\alt(G)\cong \alt(4)$ of order~$12$ is a unique proper subgroup of $\aut(\mathcal{T}_G)=\sym(G)$ which is $2$-equivalent to $\aut(\mathcal{T}_G)$ and hence $\mathcal{K}^{\min}(\mathcal{T}_G)=\{\alt(G)\}$.
\end{rem}

Two groups $K_1,K_2\leq \aut(G)$ are said to be \emph{Cayley equivalent} if $\orb(K_1,G)=\orb(K_2,G)$. In this case, we write $K_1\approx_{\cay} K_2$. If $\mathcal{A}=\cyc(K,G)$ for some $K\leq \aut(G)$, then $\aut_G(\mathcal{A})=\aut(\mathcal{A})\cap \aut(G)$ is the largest group which is Cayley equivalent to $K$. A cyclotomic $S$-ring $\mathcal{A}$ over $G$ is called \emph{Cayley minimal} if
$$\{K\leq \aut(G):~K\approx_{\cay} \aut_G(\mathcal{A})\}=\{\aut_G(\mathcal{A})\}.$$
It is easy to see that $\mathbb{Z}G$ is $2$-minimal and Cayley minimal.

\section{Tensor, generalized wreath, and star products}

This section contains a necessary information on three main constructions for producing $S$-rings.

\subsection{Tensor product}

Let $\mathcal{A}$ be an $S$-ring over a group $G$. Suppose that $G_1$ and $G_2$ are $\mathcal{A}$-subgroups such that $G=G_1\times G_2$. The $S$-ring $\mathcal{A}$ is defined to be a \emph{tensor product} of $\mathcal{A}_1=\mathcal{A}_{G_1}$ and $\mathcal{A}_2=\mathcal{A}_{G_2}$ if 
$$\mathcal{S}(\mathcal{A})=\mathcal{S}(\mathcal{A}_{1})\otimes \mathcal{S}(\mathcal{A}_{2})=\{X_1\times X_2:~X_1\in\mathcal{S}(\mathcal{A}_{1}),~X_2\in \mathcal{S}(\mathcal{A}_{2})\}.$$
In this case, we write $\mathcal{A}=\mathcal{A}_{1}\otimes \mathcal{A}_{2}$. The tensor product is called \emph{nontrivial} if $\{e\}<G_1<G$ and $\{e\}<G_2<G$, and \emph{trivial} otherwise. On the other hand, if we are given $S$-rings $\mathcal{A}_1$ and $\mathcal{A}_2$ over $G_1$ and $G_2$, respectively, then the above partition defines the $S$-ring $\mathcal{A}=\mathcal{A}_{1}\otimes \mathcal{A}_{2}$ over $G=G_1\times G_2$.

Let $\mathcal{A}$ be an $S$-ring over a group $G$ and $L$ an $\mathcal{A}$-subgroup of $G$. We say that $\mathcal{A}_L$ is \emph{$\otimes$-complemented} in $\mathcal{A}$ if there exists an $\mathcal{A}$-subgroup $U$ such that $G=L\times U$ and $\mathcal{A}=\mathcal{A}_L\otimes \mathcal{A}_U$.

One can check that
\begin{equation}\label{auttens}
\aut(\mathcal{A}_{1}\otimes \mathcal{A}_{2})=\aut(\mathcal{A}_{1})\otimes \aut(\mathcal{A}_{2}).
\end{equation}

The lemma below immediately follows from Eq.~\eqref{auttens}.

\begin{lemm}\label{schurtens}
A tensor product of two $S$-rings is schurian if and only if each of them is schurian. 
\end{lemm}

If $\mathcal{A}_{G_1}=\cyc(K_1,G_1)$ and  $\mathcal{A}_{G_2}=\cyc(K_2,G_2)$, then
\begin{equation}\label{cycltens}
\mathcal{A}_{G_1}\otimes \mathcal{A}_{G_2}=\cyc(K_1\times K_2,G_1\times G_2).
\end{equation}

\begin{lemm}\cite[Lemma 2.3]{EKP2}\label{tenspr}
Let $\mathcal{A}$ be an $S$-ring over an abelian group $G=G_1\times G_2$. Suppose that $G_1$ and $G_2$ are $\mathcal{A}$-subgroups. Then 
\begin{enumerate}
\tm{1} $X_{G_i}\in \mathcal{S}(\mathcal{A})$ for all $X\in \mathcal{S}(\mathcal{A})$ and $i\in\{1,2\}$;

\tm{2} $\mathcal{A} \geq \mathcal{A}_{G_1}\otimes \mathcal{A}_{G_2}$, and the equality is attained whenever $\mathcal{A}_{G_i}=\mathbb{Z}G_i$ for some $i\in \{1,2\}$.
\end{enumerate}
\end{lemm}

\begin{lemm}\label{tensgroup}
Let $K\leq \sym(G)$ such that $K\geq G_r$ and $\mathcal{A}=V(K,G)$. Suppose that $\mathcal{A}=\mathcal{A}_{G_1}\otimes \mathcal{A}_{G_2}$ for some $\mathcal{A}$-subgroups $G_1$ and $G_2$ such that $G=G_1\times G_2$ and $K^{G_i}\in \mathcal{K}^{\min}(\mathcal{A}_{G_i})$ for some $i\in\{1,2\}$. Then $K=K^{G_1}\times K^{G_2}$.
\end{lemm}

\begin{proof}
One can see that 
$$K\leq K^{G_1}\times K^{G_2}\leq \aut(\mathcal{A}_1)\times \aut(\mathcal{A}_2)\leq \aut(\mathcal{A})$$ 
and hence $K$ is a subdirect product of $K^{G_1}$ and $K^{G_2}$. This implies that there exist groups $K_1 \trianglelefteq K^{G_1}$, $K_2\trianglelefteq K^{G_2}$, and an isomorphism $\psi$ from $K^{G_1}/K_1$ to $K^{G_2}/K_2$ such that
$$K=\{(f_1,f_2)\in K^{G_1} \times K^{G_2}:~(K_1f_1)^\psi=K_2f_2\}.$$
Since $V(K^{G_1},G_1)=\mathcal{A}_{G_1}$, $V(K^{G_2},G_2)=\mathcal{A}_{G_2}$, and $V(K,G)=\mathcal{A}_{G_1}\otimes \mathcal{A}_{G_2}$, we conclude that $\mathcal{A}_{G_1}=V(K_1,G_1)$ and $\mathcal{A}_{G_2}=V(K_2,G_2)$. By the condition of the lemma, $K^{G_i}\in \mathcal{K}^{\min}(\mathcal{A}_{G_i})$ which yields that $K_i=K^{G_i}$. Thus, $K=K^{G_1}\times K^{G_2}$ as required.
\end{proof}

\hspace{5mm}

\noindent \textbf{Example~2.} The following example demonstrates that the condition $K^{G_i}\in \mathcal{K}^{\min}(\mathcal{A}_{G_i})$ for some $i\in\{1,2\}$ in Lemma~\ref{tensgroup} is essential. Let $\mathcal{A}_{G_1}=\mathcal{T}_{G_1}$ and $\mathcal{A}_{G_2}=\mathcal{T}_{G_2}$, where $|G_i|=p$ is an odd prime for each $i\in\{1,2\}$, and 
$$K=\alt(G_1)\times \alt(G_2) \cup (\sym(G_1)\setminus \alt(G_1)) \times (\sym(G_2)\setminus \alt(G_2)).$$
Then $K\geq\alt(G_1)\times \alt(G_2)\geq (G_1)_r \times (G_2)_r=G_r$ because $p$ is an odd prime, $\mathcal{A}=V(K,G)$, $K^{G_1}=\sym(G_1)$, $K^{G_2}=\sym(G_2)$, and $K\neq K^{G_1}\times K^{G_2}$. In this case, $K^{G_i}\notin \mathcal{K}^{\min}(\mathcal{A}_{G_i})$ for each $i\in\{1,2\}$ because $\alt(G_i)<K^{G_i}$ and $\alt(G_i)\approx_2 K^{G_i}$.

\hspace{5mm}

The corollary below immediately follows from Lemma~\ref{tensgroup} and Eq.~\eqref{auttens}.

\begin{corl}\label{2mintens}
A tensor product of two $2$-minimal $S$-rings is $2$-minimal.
\end{corl}

\subsection{Generalized wreath product} 

Let $S=U/L$ be an $\mathcal{A}$-section of $G$. The $S$-ring~$\mathcal{A}$ is called the \emph{$S$-wreath product} or \emph{generalized wreath product} of $\mathcal{A}_U$ and $\mathcal{A}_{G/L}$ if $L\trianglelefteq G$ and every basic set $X$ of $\mathcal{A}$ outside~$U$ is a union of some $L$-cosets or, equivalently, $L\leq \rad(X)$ for every $X\in \mathcal{S}(\mathcal{A})_{G\setminus U}$. In this case, we write $\mathcal{A}=\mathcal{A}_U \wr_S \mathcal{A}_{G/L}$. The $S$-wreath product is said to be \emph{nontrivial} if $L\neq \{e\}$ and $U\neq G$, and \emph{trivial} otherwise. The construction of a generalized wreath product of $S$-rings was introduced in paper~\cite{EP0}. If $L=U$, then the $S$-wreath product coincides with the \emph{wreath product} $\mathcal{A}_L\wr \mathcal{A}_{G/L}$ of $\mathcal{A}_L$ and $\mathcal{A}_{G/L}$. The $S$-ring $\mathcal{A}$ is said to be \emph{decomposable} if $\mathcal{A}$ is a nontrivial $S$-wreath product for some $\mathcal{A}$-section $S$ and \emph{indecomposable} otherwise.

Given a section $S=U/L$ of a group $G$ such that $L\trianglelefteq G$ and $S$-rings $\mathcal{A}_1$ and $\mathcal{A}_2$ over $U$ and $G/L$, respectively, such that $S$ is both an $\mathcal{A}_1$- and an $\mathcal{A}_2$-section, and $(\mathcal{A}_1)_S=(\mathcal{A}_2)_S$, there is a unique $S$-ring $\mathcal{A}$ over $G$ that is the $S$-wreath product with $\mathcal{A}_U=\mathcal{A}_1$ and $\mathcal{A}_{G/L}=\mathcal{A}_2$ (see~\cite{EP0}).

\begin{lemm}\cite[Corollary~5.7]{EP2}\label{schurwr}
Let $\mathcal{A}$ be an $S$-ring over an abelian group $G$ and $S=U/L$ an $\mathcal{A}$-section. Suppose that $\mathcal{A}$ is the $S$-wreath product and $\mathcal{A}_U$ and $\mathcal{A}_{G/L}$ are schurian. Then $\mathcal{A}$ is schurian if and only if there exist two groups $K_1\leq \sym(U)$ and $K_0\leq \sym(G/L)$ such that 
$$K_1\geq U_r,~K_0\geq (G/L)_r,~K_1\approx_2 \aut(\mathcal{A}_U),~K_0\approx_2 \aut(\mathcal{A}_{G/L}),~\text{and}~K_0^S=K_1^S.$$
\end{lemm}

\begin{lemm}\cite[Corollary~10.3]{MP}\label{2min}
Under the hypothesis of Lemma~\ref{schurwr}, the $S$-ring $\mathcal{A}$ is schurian if $\mathcal{A}_S$ is $2$-minimal. In particular, $\mathcal{A}$ is schurian if $U=L$. 
\end{lemm}



The lemma below for cyclic groups is~\cite[Theorem~7.5]{EKP1}. In fact, the proof for abelian groups is the same. To make the text self-contained, we provide a proof here.

\begin{lemm}\label{otimescomplement}
Under the hypothesis of Lemma~\ref{schurwr}, the $S$-ring $\mathcal{A}$ is schurian whenever $\mathcal{A}_L$ is $\otimes$-complemented in $\mathcal{A}_U$ or $\mathcal{A}_S$ is $\otimes$-complemented in $\mathcal{A}_{G/L}$. 
\end{lemm}
\begin{proof}
Suppose that the first condition of the lemma holds, i.e. $\mathcal{A}_U=\mathcal{A}_L\otimes \mathcal{A}_H$ for some $\mathcal{A}_U$-subgroup $H$. Clearly, $\mathcal{A}_S\cong \mathcal{A}_H$. Eq.~\eqref{auttens} implies that $\aut(\mathcal{A}_U)=\aut(\mathcal{A}_{L})\times \aut(\mathcal{A}_{H})$.
Let $K_0=\aut(\mathcal{A}_{G/L})$. Observe that $K_0^S\approx_2 \aut(\mathcal{A}_S)$ because $\mathcal{A}_{G/L}$ is schurian. Since $\mathcal{A}_S\cong \mathcal{A}_H$, there is $K_2 \leq \sym(H)$ such that $K_2\geq H_r$ and $K_2^S=K_0^S$. Put $K_1=\aut(\mathcal{A}_L)\times K_2$. One can see that: 
\begin{enumerate}
\tm{1} $K_1\geq U_r$ because $K_2\geq H_r$; 

\tm{2} $K_1\approx_2 \aut(\mathcal{A}_U)$ because $K_2^S\approx_2 \aut(\mathcal{A}_S)$ and $\aut(\mathcal{A}_{U})=\aut(\mathcal{A}_L) \times \aut(\mathcal{A}_H)$; 

\tm{3} $K_1^S=K_2^S=K_0^S$.
\end{enumerate} 
Therefore all the conditions of Lemma~\ref{schurwr} hold for $K_1$ and $K_0$. Thus, $\mathcal{A}$ is schurian.

Now suppose that the second condition of the lemma holds, i.e. $\mathcal{A}_{G/L}=\mathcal{A}_S\otimes \mathcal{A}_H$ for some $\mathcal{A}_{G/L}$-subgroup $H$.  Eq.~\eqref{auttens} yields that $\aut(\mathcal{A}_{G/L})=\aut(\mathcal{A}_{S})\times \aut(\mathcal{A}_{H})$. Let $K_1=\aut(\mathcal{A}_U)$. Note that $K_1^S\approx_2 \aut(\mathcal{A}_S)$ because $\mathcal{A}_{U}$ is schurian and $K_1^S\geq S_r$ because $K_1\geq U_r$. Put $K_0=K_1^S\times \aut(\mathcal{A}_H)$. Then:
\begin{enumerate}
\tm{1} $K_0\geq (G/L)_r$ because $K_1^S\geq S_r$; 

\tm{2} $K_0\approx_2 \aut(\mathcal{A}_{G/L})$ because $K_1^S\approx_2 \aut(\mathcal{A}_S)$ and $\aut(\mathcal{A}_{G/L})=\aut(\mathcal{A}_S) \times \aut(\mathcal{A}_H)$;

\tm{3} $K_0^S=K_1^S$. 
\end{enumerate} 
Thus, all the conditions of Lemma~\ref{schurwr} hold for $K_1$ and $K_0$ and consequently $\mathcal{A}$ is schurian.
\end{proof}

\subsection{Star product}

Let $L$ and $U$ be normal $\mathcal{A}$-subgroups. The $S$-ring $\mathcal{A}$ is called the \emph{star product} of $\mathcal{A}_L$ and $\mathcal{A}_U$  if the following conditions hold:

$(1)$ each $X\in \mathcal{S}(\mathcal{A})$ with $X\subseteq (U\setminus L) $ is a union of some $(L\cap U)$-cosets;

$(2)$ for each $X\in \mathcal{S}(\mathcal{A})$ with $X\subseteq G\setminus (L\cup U)$ there exist $Y\in \mathcal{S}(\mathcal{A}_L)$ and $Z\in \mathcal{S}(\mathcal{A}_U)$ such that $X=YZ$.

\noindent In this case, we write $\mathcal{A}=\mathcal{A}_L \star \mathcal{A}_U$. The construction of a star product of $S$-rings was introduced  in paper~\cite{HM} and extended to association schemes under the name ``crested product'' in paper~\cite{BC}. A star product is called \emph{nontrivial} if $L\neq \{e\}$ and $U\neq G$ and \emph{trivial} otherwise.

Let $L$ and $U$ be proper nontrivial normal $\mathcal{A}$-subgroups of $G$ and $\mathcal{A}=\mathcal{A}_L\star \mathcal{A}_U$. It follows immediately from the definitions that: 
\begin{enumerate}

\tm{1} if $|L\cap U|=1$, then $\mathcal{A}=\mathcal{A}_L\otimes \mathcal{A}_U$;

\tm{2} if $|L\cap U|>1$, then $\mathcal{A}$ is the nontrivial $L/(L\cap U)$-wreath product.

\end{enumerate}

\section{Duality}

Let $G$ be an abelian group. Denote by $\widehat{G}$ the group dual to~$G$, i.e., the group of all irreducible complex characters of~$G$. It is well known that $\widehat{G}\cong G$ and there is a uniquely determined lattice antiisomorphism between the subgroups of~$G$ and~$\widehat{G}$. The image of a subgroup $H$ of $G$ with respect to this antiisomorphism is denoted by~$H^\bot$.

For any $S$-ring $\mathcal{A}$ over the group~$G$, one can define the dual $S$-ring $\widehat{\mathcal{A}}$ over~$\widehat{G}$ as follows: two irreducible characters of $G$ belong to the same basic set of $\widehat{\mathcal{A}}$ if they have the same value on each basic set of~$\mathcal{A}$ (for the exact definition, we refer the reader to~\cite{EP3,Tam}). One can verify that 
$$\rk(\widehat{\mathcal{A}})=\rk(\mathcal{A})$$
and the $S$-ring dual to $\widehat{\mathcal{A}}$ is equal to~$\mathcal{A}$. The following lemma collects some facts on the dual $S$-ring taken from~\cite[Section~2.3]{EP4}.

\begin{lemm}\label{dual}
Let $G$ be an abelian group and $\mathcal{A}$ an $S$-ring over $G$. Then
\begin{enumerate}

\tm{1} the mapping $\mathcal{H}(\mathcal{A})\to\mathcal{H}(\widehat{\mathcal{A}})$, $H\mapsto H^\bot$ is a lattice antiisomorphism;

\tm{2} $\widehat{\mathcal{A}_H}=\widehat{\mathcal{A}}_{\widehat{G}/H^\bot}$ and $\widehat{\mathcal{A}_{G/H}}=\widehat{\mathcal{A}}_{H^\bot}$ for every $H\in \mathcal{H}(\mathcal{A})$;

\tm{3} $\mathcal{A}=\mathcal{A}_1\otimes\mathcal{A}_2$ if and only if $\widehat{\mathcal{A}}=\widehat{\mathcal{A}_1}\otimes\widehat{\mathcal{A}_2}$;

\tm{4} $\mathcal{A}$ is the $U/L$-wreath product for some $\mathcal{A}$-section $U/L$ if and only if $\widehat{\mathcal{A}}$ is the $L^{\bot}/U^{\bot}$-wreath product.

\tm{5} $\mathcal{A}$ is cyclotomic if and only if so is $\widehat{\mathcal{A}}$.

\end{enumerate}
\end{lemm}

\section{$S$-rings over $H\times C_p$}

Throughout this section, $G=H\times P$, where $H$ is an abelian group and $P\cong C_p$ with prime $p$ coprime to $|H|$, and $\mathcal{A}$ is an $S$-ring over~$G$. Let $H_1$ be a maximal $\mathcal{A}$-subgroup contained in $H$ and $P_1$ the least $\mathcal{A}$-subgroup containing $P$. Note that $H_1P_1$ is an $\mathcal{A}$-subgroup.

\begin{lemm}\cite[Lemma~6.3]{KR}\label{nonpower2}
If $H_1 \neq (H_1P_1)_{p'}$, the Hall $p'$-subgroup of $H_1P_1$, then $\mathcal{A}_{H_1P_1}=\mathcal{A}_{H_1} \star \mathcal{A}_{P_1}$.
\end{lemm}

\begin{lemm}\cite[Proposition 15]{MS}\label{nonpower3}
If $\mathcal{A}_{(H_1P_1)/H_1}\cong \mathbb{Z}C_p$, then $\mathcal{A}_{H_1P_1}=\mathcal{A}_{H_1} \star \mathcal{A}_{P_1}$. 
\end{lemm}

\begin{lemm}\label{nonpowernew1}
Suppose that $H_1<H$. Then one of the following statements holds:
\begin{enumerate}
\tm{1} $H_1P_1=G$, $P_1\lneq G$, and $\mathcal{A}=\mathcal{A}_{H_1} \star \mathcal{A}_{P_1}$;

\tm{2} $\mathcal{A}=\mathcal{A}_{H_1}\wr \mathcal{A}_{G/H_1}$ with $\mathcal{A}_{G/H_1}=\mathcal{T}_{G/H_1}$;

\tm{3} $\mathcal{A}$ is the nontrivial $(H_1P_1)/P_1$-wreath product.
\end{enumerate}
\end{lemm}

\begin{proof}
From~\cite[Lemma~6.2]{EKP2} it follows that Statement~$(2)$ of the lemma holds or $\mathcal{A}$ is the $(H_1P_1)/P_1$-wreath product with $P_1<G$. If the above wreath product is nontrivial, then Statement~$(3)$ of the lemma holds, whereas if it is trivial, i.e. if $H_1P_1=G$, then $H_1 \neq (H_1P_1)_{p'}=G_{p'}=H$. So $\mathcal{A}=\mathcal{A}_{H_1P_1}=\mathcal{A}_{H_1} \star \mathcal{A}_{P_1}$ by Lemma~\ref{nonpower2} and Statement~$(1)$ of the lemma holds. 
\end{proof}

\begin{lemm}\label{nonpowernew2}
If $p=2$, then $\mathcal{A}_{H_1P_1}=\mathcal{A}_{H_1} \star \mathcal{A}_{P_1}$.
\end{lemm}

\begin{proof}
If $H_1 \neq (H_1P_1)_{p'}$, then the lemma follows from Lemma~\ref{nonpower2}. If $H_1=(H_1P_1)_{p'}$, then $(H_1P_1)/H_1\cong C_2$ and hence $\mathcal{A}_{(H_1P_1)/H_1}\cong \mathbb{Z}C_2$. So in this case, the lemma follows from Lemma~\ref{nonpower3}.
\end{proof}

\section{$S$-rings over cyclic groups}

In this section, we provide several results on $S$-rings over cyclic groups, required for the proofs of the main results. An $S$-ring over a cyclic group is said to be \emph{circulant}.

\subsection{General results}

\begin{lemm}\cite[Theorem~25.4]{Wi}\label{primitive}
If $G$ is an abelian group of composite order having a cyclic Sylow subgroup, then $\mathcal{T}_G$ is the only primitive $S$-ring over $G$. 
\end{lemm}

The description of $S$-rings over a cyclic group was obtained in papers~\cite{LM1,LM2}. Below, we provide this description in a convenient for us form taken from~\cite[Theorem~4.1]{EP2}. If $\mathcal{A}$ is an $S$-ring over a cyclic group $G$, then put $\rad(\mathcal{A})=\rad(X)$, where $X$ is a basic set of $\mathcal{A}$ containing a generator of $G$. Due to Lemma~\ref{burn}, all basic sets of $\mathcal{A}$, containing a generator of $G$, are rationally conjugate and hence have the same radical. Therefore $\rad(\mathcal{A})$ does not depend on the choice of $X$.

\begin{lemm}\label{leungman}
Let $\mathcal{A}$ be an $S$-ring over a cyclic group. Then the following statements hold:
\begin{enumerate}

\tm{1} $|\rad(\mathcal{A})|=1$ if and only if $\mathcal{A}$  is a tensor product (possibly, trivial) of a normal cyclotomic $S$-ring with trivial radical and trivial $S$-rings;

\tm{2} $|\rad(\mathcal{A})|>1$ if and only if $\mathcal{A}$  is a nontrivial generalized wreath product.
\end{enumerate} 
\end{lemm}

\begin{lemm}\cite[Theorem~4.2(1)]{EP2}\label{cyclnormschur}
Every normal $S$-ring over a cyclic group is cyclotomic and, in particular, schurian. 
\end{lemm}

\begin{lemm}\label{cyclcayleymin}
Every cyclotomic $S$-ring over a cyclic group is Cayley minimal.
\end{lemm}

\begin{proof}
Let $G$ be a cyclic group and $\mathcal{A}=\cyc(K,G)$ for some $K\leq \aut(G)$. Assume that there exists $K^{\prime}<K$ such that $\mathcal{A}=\cyc(K^{\prime},G)$. Let $X$ be a basic set of $\mathcal{A}$ containing a generator $x$ of $G$. Clearly, the one-point stabilizer of $x$ in $\aut(G)$ is trivial. Since $X\in \orb(K,G)\cap \orb(K^\prime,G)$, we conclude that $|K|=|K_x||X|=|X|=|K^{\prime}_x||X|=|K^\prime|$, a contradiction to $K^\prime<K$. 
\end{proof}

\begin{lemm}\label{2minnorm}
Every normal $S$-ring over a cyclic group is $2$-minimal.
\end{lemm}

\begin{proof}
Let $\mathcal{A}$ be a normal $S$-ring over a cyclic group $G$. Clearly, $\aut(\mathcal{A})\leq \Hol(G)$. As $\mathcal{A}$ is normal, $\mathcal{A}$ is schurian and hence $\mathcal{A}=\cyc(\aut(\mathcal{A})_e,G)$ by Lemma~\ref{cyclnormschur}. Assume that $\mathcal{A}$ is not $2$-minimal, i.e. there exists $K<\aut(\mathcal{A})$ such that $K\geq G_r$ and $K\approx_2 \aut(\mathcal{A})$. Since $K<\aut(\mathcal{A})\leq \Hol(G)$, we obtain $K_e\leq \aut(G)$. This implies that $\mathcal{A}=\cyc(K_e,G)$. From Lemma~\ref{cyclcayleymin} it follows that $K_e=\aut(\mathcal{A})_e$ and consequently $K=G_rK_e=G_r\aut(\mathcal{A})_e=\aut(\mathcal{A})$, a contradiction to the assumption.
\end{proof}

\begin{lemm}\label{prime}
Under the hypothesis of Lemma~\ref{schurwr}, suppose that $U$ and $G/L$ are cyclic. Then $\mathcal{A}$ is schurian whenever $|S|$ is prime.
\end{lemm}

\begin{proof}
If $\mathcal{A}_S\neq \mathcal{T}_S$, then $\mathcal{A}_S$ is normal by Lemma~\ref{leungman}. So $\mathcal{A}_S$ is $2$-minimal by Lemma~\ref{2minnorm} and $\mathcal{A}$ is schurian by Lemma~\ref{2min}.

Suppose that $\mathcal{A}_S=\mathcal{T}_S$. Then $\mathcal{A}_S=\cyc(\aut(S),S)=V(\Hol(S),S)$. Since $\aut(\mathcal{A}_U)$ and $\mathcal{A}_{G/L}$ are schurian, \cite[Theorem~8.1(2)]{EP2} implies that there exist groups $K_1\leq \aut(\mathcal{A}_U)$ and $K_0\leq \aut(\mathcal{A}_{G/L})$ such that 
$$K_1\geq U_r,~K_0\geq (G/L)_r,~K_1\approx_2 \aut(\mathcal{A}_U),~K_0\approx_2 \aut(\mathcal{A}_{G/L}),~K_1^S=K_0^S=\Hol(S).$$
Thus, all the conditions of Lemma~\ref{schurwr} hold for $K_1$ and $K_0$ and hence $\mathcal{A}$ is schurian.
\end{proof}

\subsection{$S$-rings over cyclic $p$-groups}

For the lemma below see, e.g.,~\cite[Lemma~2.4]{PR}.

\begin{lemm}\label{cyclprime}
Every $S$-ring over a cyclic group of prime order is cyclotomic.
\end{lemm}

The first statement of the lemma below immediately follows from~\cite[Lemma~5.1]{EP1}, whereas the second one follows from Lemma~\ref{leungman} and the first one.

\begin{lemm}\label{trivradorb}
Let $p$ be an odd prime, $G$ a cyclic $p$-group, $\mathcal{A}$ an $S$-ring over $G$, and $X\in \mathcal{S}(\mathcal{A})\setminus\{\{e\}\}$.

\begin{enumerate}

\tm{1} If $\mathcal{A}=\cyc(K,G)$ for some $K\leq \aut(G)$ and $|\rad(\mathcal{A})|=1$, then $|X|=|K|\leq p-1$ and all the elements of $X$ lie in pairwise distinct cosets by every proper subgroup of~$\langle X \rangle$.

\tm{2} If $\mathcal{A}$ is nontrivial with trivial radical, then $\mathcal{A}=\cyc(K,G)$ for some $K\leq \aut(G)$ such that $|K|$ divides~$p-1$ and $|X|=|K|$. 

\end{enumerate}

\end{lemm}

\begin{lemm}\label{cyclradp}
Let $p$ be an odd prime and $\mathcal{A}$ a cyclotomic $S$-ring over a cyclic $p$-group $G$ such that $|\rad(\mathcal{A})|\in\{1,p\}$. Then $|\rad(\mathcal{A}_S)|=1$ for every $\mathcal{A}$-section $S\neq G$. 
\end{lemm}

\begin{proof}
Let $\mathcal{A}=\cyc(K,G)$, where $K\leq \aut(G)$. Since $G$ is a cyclic $p$-group of odd order, $\aut(G)=A_1\times A_2$, where 
$$A_1\cong C_{p-1}~\text{and}~A_2=\{\sigma_m:~m\in \mathbb{Z},~m\equiv 1\mod~p\}\cong C_{|G|/p}.$$ 
So $K=K_1\times K_2$, where $K_1=K\cap A_1$ and $K_2=K\cap A_2$. One can see that 
$$|K_2|=|\rad(\mathcal{A})|\in\{1,p\}.$$ 
Therefore $|K_2|=1$ or $K_2=\{\sigma_m:~m\in \mathbb{Z},~m\equiv 1\mod~|G|/p\}$. This implies that $K^{G/L_0}$ and $K^{U_0}$, where $L_0$ and $U_0$ are the subgroups of $G$ of order and index~$p$, respectively, are $p^\prime$-groups of order at most~$p-1$. If $S=U/L$ is an $\mathcal{A}$-section and $S\neq G$, then $U\leq U_0$ or $L\geq L_0$. This yields that $\mathcal{A}_S=\cyc((K^{G/L_0})^S,S)$ or $\mathcal{A}_S=\cyc((K^{U_0})^S,S)$. As $|(K^{G/L_0})^S|\leq p-1$ and $|(K^{U_0})^S|\leq p-1$, we conclude that every basic set of $\mathcal{A}_S$ is of size at most~$p-1$ and hence has a trivial radical. Thus, $|\rad(\mathcal{A}_S)|=1$. 
\end{proof}

\begin{lemm}\label{cyclpwreath}
Let $\mathcal{A}$ be a nontrivial generalized wreath product over a cyclic $p$-group~$G$. Then $\mathcal{A}$ is the nontrivial $U/L$-wreath product for an $\mathcal{A}$-section $U/L$ of $G$ such that $L$ is the least nontrivial $\mathcal{A}$-subgroup and $|\rad(\mathcal{A}_{U})|=1$.
\end{lemm}

\begin{proof}
Since $G$ is a cyclic $p$-group, there exists the least nontrivial $\mathcal{A}$-subgroup $L$. Let $U$ be the $\mathcal{A}$-subgroup generated by all basic set of $\mathcal{A}$ with trivial radical. Note that $U<G$ because $\mathcal{A}$ is a nontrivial generalized wreath product. Every basic set $X$ of $\mathcal{A}$ outside $U$ has a nontrivial radical and hence $\rad(X)\geq L$. Thus, $\mathcal{A}$ is the nontrivial $U/L$-wreath product as desired.
\end{proof}

A subset of $G$ is said to be \emph{regular} if it consists of elements of the same order. The lemma below can be deduced from the description of $S$-rings over cyclic $p$-groups~\cite{Po}. To make the text self-contained, we provide the proof here. 

\begin{lemm}\label{pnonreg}
Let $\mathcal{A}$ be an $S$-ring over a cyclic $p$-group $G$ and $X\in \mathcal{S}(\mathcal{A})$. Suppose that $X$ is not regular. Then $X=U\setminus L$ for some $\mathcal{A}$-subgroups $U>L$ such that $|U:L|\geq p^2$ and $\mathcal{A}=\mathcal{A}_L\wr \mathcal{A}_{G/L}=\mathcal{A}_U\wr \mathcal{A}_{G/U}$.
\end{lemm}

\begin{proof}
Let $L=\rad(X)$. The image of $T\subseteq G$ under the canonical epimorphism from $G$ to $G/L$ is denoted by $\overbar{T}$. The $S$-ring $\mathcal{A}_{\langle \overbar{X} \rangle}$ has a trivial radical. So $\mathcal{A}_{\langle \overbar{X} \rangle}$ is trivial or cyclotomic by Lemma~\ref{leungman}. As $X$ is not regular, $\overbar{X}$ also is not regular. Therefore $\mathcal{A}_{\langle \overbar{X} \rangle}$ can not be cyclotomic and hence it is trivial. Thus, $\overbar{X}\cup \{L\}$ is an $\mathcal{A}_{G/L}$-subgroup. Since $L=\rad(X)$, we conclude that $X=U\setminus L$ for some $\mathcal{A}$-subgroup $U$ as required. One can see that $|U:L|\geq p^2$ because otherwise $X$ is regular which contradicts to the assumption of the lemma.

Now let us prove the second part of the lemma.  If $U=G$, then $\mathcal{A}$ is the $L/L$-wreath product (possibly, trivial) and the trivial $U/U$-wreath product. Further, we assume that $U<G$ and hence $\overbar{U}<\overbar{G}$. The latter implies that $\mathcal{A}_{\overbar{G}}$ is nontrivial. Note that $\overbar{X}=(\overbar{U})^\#$ is a nonregular basic set of $\mathcal{A}_{\overbar{G}}$ with trivial radical. So $\mathcal{A}_{\overbar{G}}$ can not be cyclotomic and hence $\mathcal{A}_{\overbar{G}}$ is a nontrivial generalized wreath product by Lemma~\ref{leungman}. In view of Lemma~\ref{cyclpwreath}, $\mathcal{A}_{\overbar{G}}$ is the nontrivial $W/V$-wreath product, where $\rad(\mathcal{A}_W)$ is trivial and $V$ is the least nontrivial $\mathcal{A}_{\overbar{G}}$-subgroup. Since $\rad(\overbar{X})$ is trivial, we conclude that $\overbar{X}\subseteq W$. Hence $\mathcal{A}_W$ is not cyclotomic. Due to Lemma~\ref{leungman}, the $S$-ring $\mathcal{A}_W$ is trivial. Therefore $W=V=\overbar{U}$. Thus,
$$\mathcal{A}_{\overbar{G}}=\mathcal{A}_{\overbar{U}}\wr \mathcal{A}_{\overbar{G}/\overbar{U}}.$$

In view of the above equality, to prove the lemma it suffices to show that $L\leq \rad(Y)$ for every $Y\in \mathcal{S}(\mathcal{A})_{G\setminus U}$. Assume the contrary that $N=\rad(Y)<L$ for some $Y\in\mathcal{S}(\mathcal{A})_{G\setminus U}$. The image of $T\subseteq G$ under the canonical epimorphism from $G$ to $G/N$ is denoted by $\widetilde{T}$. As $Y$ lies outside $U$, we have $X\subseteq U\leq \langle Y \rangle$. So $\widetilde{X}$ is a nonregular basic set of $\mathcal{A}_{\langle \widetilde{Y}\rangle}$. This implies that $\mathcal{A}_{\langle\widetilde{Y}\rangle}$ can not be cyclotomic. Therefore $\mathcal{A}_{\langle \widetilde{Y}\rangle}$ is trivial by Lemma~\ref{leungman}. On the other hand, $\widetilde{X}$ and $\widetilde{Y}$ are nontrivial basic sets of $\mathcal{A}_{\langle\widetilde{Y}\rangle}$, a contradiction.
\end{proof}

\subsection{$S$-rings over some other cyclic groups}

\begin{lemm}\label{pq}
Every nontrivial $S$-ring over a cyclic group of order a product of two distinct primes is cyclotomic or a nontrivial wreath product.
\end{lemm}

\begin{proof}
An $S$-ring over a group of prime order is cyclotomic by Lemma~\ref{cyclprime} and hence a tensor product of $S$-rings over groups of prime orders is cyclotomic by Eq.~\eqref{cycltens}. Therefore the required statement follows from Lemma~\ref{leungman}.
\end{proof}

\begin{lemm}\label{2primepower}
Let $\mathcal{A}$ be a nontrivial $S$-ring over a cyclic group $G$ of order~$2p^k$, where $p$ is an odd prime and $k\geq 1$. Suppose that $|\rad(\mathcal{A})|=1$. Then $\mathcal{A}$ is normal unless $\mathcal{A}=\mathbb{Z}L\otimes \mathcal{T}_U$, where $L,U\leq G$ are such that $|L|=2$ and $|U|=p^k$.
\end{lemm}

\begin{proof}
Assume that $\mathcal{A}$ is nonnormal. Then $\mathcal{A}=\mathbb{Z}_L\otimes \mathcal{A}_U$ by Lemma~\ref{leungman}. If $|\rad(\mathcal{A}_U)|>1$, then $|\rad(\mathcal{A})|>1$, a contradiction to the assumption of the lemma. So $|\rad(\mathcal{A}_U)|=1$. If $\mathcal{A}_U\neq \mathcal{T}_U$, then $\mathcal{A}_U$ is normal by Lemma~\ref{leungman} and consequently so is $\mathcal{A}$ by Eq.~\eqref{auttens}, a contradiction to the assumption of the lemma. Thus, $\mathcal{A}_U=\mathcal{T}_U$ and hence $\mathcal{A}=\mathbb{Z}_L\otimes \mathcal{T}_U$. 
\end{proof}

\section{$S$-rings over $E_4\times C_n$}

The main goal of this section is to give a description of $S$-rings over the group $E_4\times C_n$, where  $n\in\{p^k,pq\}$, $p$ and $q$ are odd primes, and $k\geq 1$. We say that an $S$-ring $\mathcal{A}$ over $E_4\times C_n$ is \emph{dense} if $E_4$ and $C_n$ are $\mathcal{A}$-subgroups and \emph{nondense} otherwise.

\begin{theo}\label{e4cn}
Let $n\in\{p^k,pq\}$, where $p$ and $q$ are odd primes and $k$ is a positive integer, and $\mathcal{A}$ a nontrivial $S$-ring over the group $G=H\times D$, where $H\cong E_4$ and $D\cong C_n$. Then $\mathcal{A}$ is cyclotomic, or a nontrivial tensor product, or a nontrivial $S$-wreath product for some $\mathcal{A}$-section $S=U/L$ and one of the following statements holds:

\begin{enumerate}

\tm{1} $|S|\leq 2$;

\tm{2} $\mathcal{A}_{L}$ is $\otimes$-complemented in $\mathcal{A}_{U}$ or $\mathcal{A}_{S}$ is $\otimes$-complemented in $\mathcal{A}_{G/L}$;

\tm{3} $\mathcal{A}$ is nondense and at least one of the $S$-rings $\mathcal{A}_U$, $\mathcal{A}_{G/L}$ is a circulant cyclotomic $S$-ring with trivial radical;

\tm{4} $\mathcal{A}$ is dense, $L\leq D$, $U\geq H$, $\mathcal{A}_U$ is cyclotomic, and $|\rad(\mathcal{A}_{U\cap D})|=1$ unless $n=3^k$ and $|\rad(\mathcal{A}_{U\cap D})|=3$.
\end{enumerate} 

\end{theo}

Observe that $U\cap D$ in Statement~$(4)$ of Theorem~\ref{e4cn} is cyclic and hence $\rad(\mathcal{A}_{U\cap D})$ is defined correctly. 

Clearly, Theorem~\ref{main15} immediately follows from Theorem~\ref{e4cn}.

We divide the proof of Theorem~\ref{e4cn} into two cases depending on whether $\mathcal{A}$ is nondense or dense. We prove that every nondense nontrivial $S$-ring over $G$ is a nontrivial tensor or generalized wreath product and every dense $S$-ring over $G$ is cyclotomic or a nontrivial generalized wreath product. The nondense case is divided into two subcases depending on whether $H$ or $D$ is not an $\mathcal{A}$-subgroup (Theorems~\ref{hnot} and~\ref{dnot}, respectively), whereas the dense case is divided into two subcases depending on whether $\rad(\mathcal{A}_D)$ is trivial or not (Theorems~\ref{hpasubgroups2} and~\ref{hpasubgroups3}, respectively).

In the first subsection, we study a structure of a basic set of $\mathcal{A}$. The second and third subsections deal with the above mentioned cases when $\mathcal{A}$ is nondense and when $\mathcal{A}$ is dense, respectively. In the fourth subsection, we give a proof of Theorem~\ref{e4cn}. The fifth subsection contains some auxiliary statements to be used in the proof of Theorem~\ref{main1}.

Throughout this section, $p$ and $q$ are odd primes, $k$ is a positive integer, $n\in\{p^k,pq\}$, $D\cong C_n$, $P=D_p=\{g\in D:~g^p=e\}$, $H\cong E_4$, $G=H\times D$, and $\mathcal{A}$ is an $S$-ring over $G$.

\subsection{Structure of a basic set}

\begin{lemm}\label{2set}
Let $n=p^k$ and $X\in \mathcal{S}(\mathcal{A})$. Suppose that $X\nsubseteq D$. Then one of the following statements holds:
\begin{enumerate}

\tm{1} $P\leq \rad(X)$;

\tm{2} $\langle (X^{[p]})_H \rangle$ is a nontrivial $\mathcal{A}$-subgroup of~$H$;

\tm{3} $X\cap D\neq \varnothing$ and $X\cup \{e\}$ is an $\mathcal{A}$-subgroup of $G$. 

\end{enumerate}
\end{lemm}

\begin{proof}
Assume that $P\nleq \rad(X)$, i.e. Statement~$(1)$ of the lemma does not hold. Then the union $Y$ of all $hX_{h,l}=X\cap hD_l^*$, $h\in H$, $l\in \divv(n)$, such that $X_{h,l}\neq \varnothing$ and $P\nleq \rad(X_{h,l})$ is nonempty. As $D\cong C_{p^k}$ and $X_{h,l}\subseteq D$, the condition $P\nleq \rad(X_{h,l})$ is equivalent to $|\rad(X_{h,l})|=1$. Lemma~\ref{orbit} implies that each $X_{h,l}$ is an orbit of some subgroup~$K$ of $\aut(D_l)$. From Lemma~\ref{trivradorb}(1) it follows that $|K|\leq p-1$ and hence
\begin{equation}\label{sizep1} 
|X_{h,l}|\leq p-1
\end{equation} 
for all $h\in H$ and $l\in \divv(n)$ such that $X_{h,l}\neq \varnothing$ and $P\nleq \rad(X_{h,l})$. 

Put $H_0=Y_H$, $Y^{[p]^0}=Y$, $Y^{[p]^1}=Y^{[p]}$, and $Y^{[p]^i}=(Y^{[p]^{i-1}})^{[p]}$ for $i\geq 2$. Due to the definition of $Y$ and Eq.~\eqref{sizep1}, we have $Y^{[p]}=X^{[p]}$ and $Y^{[p]^i}=\{y^p:~y\in Y^{[p]^{i-1}}\}$. Therefore $(Y^{[p]^i})_H=H_0$ for every $i\geq 1$, in particular, $(Y^{[p]})_H=(X^{[p]})_H=H_0$. By Lemma~\ref{sch}, the set $Y^{[p]^i}$ is an $\mathcal{A}$-set for every $i\geq 1$. So $Y^{[p]^k}=H_0$ is an $\mathcal{A}$-set. Thus, $\langle H_0 \rangle$ is an $\mathcal{A}$-subgroup. If $H_0\neq \{e\}$, then Statement~$(2)$ of the lemma holds.

In view of the previous paragraph, we may assume that $H_0=\{e\}$ and hence 
$$Y\subseteq X\cap D\neq \varnothing.$$
Let $l=p^j$ be the greatest integer such that $X_{e,l}\neq \varnothing$ and $P\nleq \rad(X_{e,l})$. By the supposition of the lemma, $X\setminus D\neq \varnothing$. So
\begin{equation}\label{separe4cpk}
X\setminus D_l\neq \varnothing~\text{and}~P\leq \rad(X\setminus D_l).
\end{equation}
If $j\neq 1$, then $Y^{[p]^{j-1}}$ is a nonemepty $\mathcal{A}$-set inside $P^\#$. Therefore $P=\langle Y^{[p]^{j-1}} \rangle$ is an $\mathcal{A}$-subgroup. Now Lemma~\ref{intersection} and Eq.~\eqref{separe4cpk} imply that $P\leq \rad(X)$ in contrast to the assumption $P\nleq \rad(X)$. If $j=1$, then $X_{e,l}\subseteq P^\#$. Eq.~\eqref{separe4cpk} and Lemma~\ref{separ} applied to $X$ and $P$ yield that $\rad(X)\leq P$ and $X=\langle X \rangle \setminus \rad(X)$. By the assumption, $P\nleq \rad(X)$ and consequently $\rad(X)$ is trivial. Thus, $X=\langle X \rangle \setminus \{e\}$ and Statement~$(3)$ of the lemma holds.
\end{proof}

\subsection{Nondense $S$-rings}

\begin{theo}\label{hnot}
Suppose that $H\cong E_4$ is not an $\mathcal{A}$-subgroup. Then $\mathcal{A}$ is trivial, or a nontrivial tensor product, or a nontrivial $U/L$-wreath product for some $\mathcal{A}$-section $U/L$ satisfying one of the following conditions:
\begin{enumerate}

\tm{1} $|U/L|\leq 2$;

\tm{2} $\mathcal{A}_{L}$ is $\otimes$-complemented in $\mathcal{A}_{U}$;

\tm{3} $4 \nmid |U|$ and $|\rad(\mathcal{A}_{U})|=1$.

\end{enumerate}
\end{theo}

We divide the proof of Theorem~\ref{hnot} into two propositions. 

\begin{prop}\label{hnot1}
If $H$ is not an $\mathcal{A}$-subgroup, then Theorem~\ref{hnot} holds unless $\mathcal{A}$ is a nontrivial $U/L$-wreath product for some $\mathcal{A}$-section $U/L$ such that $4 \nmid |U|$ and $|L|\neq 2$.
\end{prop}

\begin{proof}
Let $n=pq$. Observe that the subgroup of $G$ of order~$4p$ or the subgroup of $G$ of order~$4q$ is not an $\mathcal{A}$-subgroup because otherwise $H$ is an intersection of $\mathcal{A}$-subgroups which contradicts to the supposition of the proposition. Without loss of generality, we may assume that the subgroup of $G$ of order~$4q$ is not an $\mathcal{A}$-subgroup. Then
\begin{equation}\label{orderh1} 
|H_1|\in\{1,2,q,2q\},
\end{equation}
where $H_1$ is a maximal $\mathcal{A}$-subgroup whose order is not divisible by~$p$. 

Let $P_1$ be the least $\mathcal{A}$-subgroup containing $P\cong C_p$. One of the statements of Lemma~\ref{nonpowernew1} holds for $\mathcal{A}$, $H_1$, and $P_1$. If Statement~$(1)$ of Lemma~\ref{nonpowernew1} holds, then $H_1P_1=G$ and $\mathcal{A}=\mathcal{A}_{H_1} \star \mathcal{A}_{P_1}$. If $H_1\cap P_1$ is trivial, then $\mathcal{A}=\mathcal{A}_{H_1} \otimes \mathcal{A}_{P_1}$ is a nontrivial tensor product and we are done. Otherwise, $\mathcal{A}$ is the nontrivial $U/L=H_1/(H_1\cap P_1)$-wreath product. Together with Eq.~\eqref{orderh1}, this implies that $|U/L|\leq 2$ as desired or $|H_1|=2q$ and $|H_1\cap P_1|=2$. In the latter case, Lemma~\ref{pq} yields that $\mathcal{A}_{H_1}=\mathcal{A}_{H_1\cap P_1} \otimes \mathcal{A}_{Q}$, where $Q$ is an $\mathcal{A}$-subgroup of order~$q$, or $\mathcal{A}_{H_1}=\mathcal{A}_{H_1\cap P_1} \wr \mathcal{A}_{H_1/(H_1\cap P_1)}$. In the first case, $\mathcal{A}_L$ is $\otimes$-complemented in $\mathcal{A}_U$ and $U/L$ satisfies Condition~$(2)$ from Theorem~\ref{hnot} as required. In the second one, $\mathcal{A}$ is the $L/L=(H_1\cap P_1)/(H_1\cap P_1)$-wreath product and the section $L/L$ satisfies Condition~$(1)$ from Theorem~\ref{hnot} as required.

If Statement~$(2)$ of Lemma~\ref{nonpowernew1} holds, then $\mathcal{A}$ is trivial as desired whenever $H_1$ is trivial and the $U/L=H_1/H_1$-wreath product satisfying Condition~$(1)$ from Theorem~\ref{hnot} otherwise. If Statement~$(3)$ of Lemma~\ref{nonpowernew1} holds, then $\mathcal{A}$ is the nontrivial $U/L=(H_1P_1)/P_1$-wreath product. Due to Eq.~\eqref{orderh1}, the section $U/L$ satisfies Condition~$(1)$ from Theorem~\ref{hnot} or $4 \nmid |U|$ and $|L|\neq 2$ as required.

Now let $n=p^k$, $k\geq 1$. As in the previous paragraph, $P_1$ denotes the least $\mathcal{A}$-subgroup containing $P$. One of Statements~$(1)$-$(3)$ of Lemma~\ref{2set} holds for every basic set of $\mathcal{A}$ outside~$D$. We divide the rest of the proof into two cases depending on existence of a basic set outside $D$ satisfying Statement~$(3)$ of Lemma~\ref{2set}. 

\hspace{5mm}

\noindent \textbf{Case~1.} Suppose that there is a basic set of $\mathcal{A}$ outside $D$, satisfying Statement~$(3)$ of Lemma~\ref{2set}. Then $\mathcal{A}_{P_1}=\mathcal{T}_{P_1}$ and $P_1\nleq D$. If $P_1=G$, then $\mathcal{A}$ is trivial and we are done. Further, we assume that $G\setminus P_1\neq \varnothing$.

Observe that the condition of Lemma~\ref{2set} holds for every $X\in \mathcal{S}(\mathcal{A})_{G\setminus P_1}$, i.e. $X\nsubseteq D$. Indeed, otherwise $\langle X \rangle$ is an $\mathcal{A}$-subgroup inside $D$. So $\langle X \rangle\geq P_1$ and hence 
$$D\geq \langle X\rangle \geq P_1\nleq D,$$
a contradiction.  

The above paragraph implies that one of Statements~$(1)$-$(3)$ from Lemma~\ref{2set} holds for every $X\in \mathcal{S}(\mathcal{A})_{G\setminus P_1}$. Statement~$(3)$ can not hold for $X\in \mathcal{S}(\mathcal{A})_{G\setminus P_1}$ because $P_1$ is a unique minimal $\mathcal{A}$-subgroup containing $P$. If Statement~$(1)$ holds for every $X\in \mathcal{S}(\mathcal{A})_{G\setminus P_1}$, then $P_1\leq \rad(X)$ and hence $\mathcal{A}$ is the $U/L=P_1/P_1$-wreath product and $U/L$ satisfies Condition~$(1)$ from Theorem~\ref{hnot} as required.

In view of the previous paragraph, we may assume that there is $X\in \mathcal{S}(\mathcal{A})_{G\setminus P_1}$ for which Statement~$(2)$ of Lemma~\ref{2set} holds and $P_1\nleq \rad(X)$. Since $H$ is not an $\mathcal{A}$-subgroup, $A=\langle (X^{[p]})_H \rangle$ is an $\mathcal{A}$-subgroup of order~$2$. Clearly, $\mathcal{A}_A=\mathbb{Z}A$. Note that $A\cap P_1$ is trivial because $\mathcal{A}_{P_1}=\mathcal{T}_{P_1}$. Therefore
\begin{equation}\label{tensap1}
\mathcal{A}_{A\times P_1}=\mathbb{Z}A\otimes \mathcal{T}_{P_1}
\end{equation}
by Lemma~\ref{tenspr}(2). If $A\times P_1=G$, then $\mathcal{A}$ is a nontrivial tensor product by Eq.~\eqref{tensap1} and we are done. So we may assume further that $A\times P_1<G$. 

If $P\leq \rad(Y)$ for every $Y\in \mathcal{S}(\mathcal{A})_{G\setminus (A\times P_1)}$, then $\mathcal{A}$ is the nontrivial $U/L=(A\times P_1)/P_1$-wreath product and $U/L$ satisfies Condition~$(1)$ from Theorem~\ref{hnot} as desired.

If there is $Y\in \mathcal{S}(\mathcal{A})_{G\setminus (A\times P_1)}$ such that $P\nleq \rad(Y)$, then Statement~$(2)$ of Lemma~\ref{2set} holds for~$Y$, i.e $\langle (Y^{[p]})_H\rangle$ is a nontrivial $\mathcal{A}$-subgroup of~$H$. Eq.~\eqref{tensap1} and $P_1\nleq D$ imply that $A$ is a unique nontrivial $\mathcal{A}$-subgroup of $H$ and hence $\langle (Y^{[p]})_H \rangle=A$. The latter equality yields that $Y^{[p]}\subseteq A\times D$. The set $Y^{[p]}$ is an $\mathcal{A}$-set by Lemma~\ref{sch} and consequently $\langle Y^{[p]} \rangle$ is an $\mathcal{A}$-subgroup of $A\times D$. However, $A$ is a unique nontrivial $\mathcal{A}$-subgroup of $A\times D$ by Eq.~\eqref{tensap1} and $P_1\nleq D$. Therefore $\langle Y^{[p]} \rangle=A$. From the latter equality, it follows that $Y\cap (A\times P_1)\neq \varnothing$. As $A\times P_1$ is an $\mathcal{A}$-subgroup, we obtain $Y\subseteq A\times P_1$ in contrast to the definition of~$Y$.

\hspace{5mm}

\noindent \textbf{Case~2.} Suppose that Statement~$(1)$ or Statement~$(2)$ of Lemma~\ref{2set} holds for every basic set of $\mathcal{A}$ outside $D$. Clearly, $P\nleq \rad(X)$ for every $X\in \mathcal{S}(\mathcal{A})$ such that $X\cap P\neq \varnothing$. So $N\geq P_1$, where $N$ is the $\mathcal{A}$-subgroup generated by all basic sets $X$ of $\mathcal{A}$ such that $P\nleq \rad(X)$.

 We may assume that $|N|$ is divisible by~$4$: otherwise $\mathcal{A}$ is the nontrivial $U/L=N/P_1$-wreath product and $4 \nmid |U|$ and $|L|\neq 2$ as required. Assume that $|N|$ is divisible by~$4$ or, equivalently, $N\geq H$. Then there is $Y\in \mathcal{S}(\mathcal{A})$ outside $D$ such that $P\nleq \rad(Y)$ because otherwise $N\leq D$ and hence $N\ngeq H$. Lemma~\ref{2set} implies that $A=\langle (Y^{[p]})_H \rangle$ is a nontrivial $\mathcal{A}$-subgroup of $H$. As $H$ is not an $\mathcal{A}$-subgroup, $|A|=2$.

Since $N\geq H$, there is $Z\in\mathcal{S}(\mathcal{A})_N$ such that $P\nleq \rad(Z)$ and $Z_H\nsubseteq A$ because otherwise $N\leq A\times D$ and hence $N\ngeq H$. Note that $Z\nsubseteq H$ and hence
$$Z_D\setminus \{e\}\neq \varnothing$$
because otherwise $H=\langle A,Z\rangle$ is an $\mathcal{A}$-subgroup which contradicts to the assumption of the proposition. The group $\langle (Z^{[p]})_H \rangle$ is a nontrivial $\mathcal{A}$-subgroup of~$H$ by Lemma~\ref{2set}. If $\langle (Z^{[p]})_H \rangle\neq A$, then $H=\langle A,(Z^{[p]})_H \rangle$ is an $\mathcal{A}$-subgroup, a contradiction to the assumption of the proposition. Therefore $\langle (Z^{[p]})_H \rangle=A$. This implies that $Z\cap (A^\#\times D^\#)\neq \varnothing$,
\begin{equation}\label{radzp}
P\nleq \rad(Z\cap (A\times D)),
\end{equation} 
and
\begin{equation}\label{radzp2}
P\leq \rad(Z\setminus (A\times D)).
\end{equation}

Let $V=\langle Z \rangle$. The $S$-ring $\mathcal{A}_{V/A}$ is an $S$-ring over a cyclic group $V/A\cong C_{2p^i}$, $i\in\{1,\ldots,k\}$. Every basic set of a cyclotomic $S$-ring or a nontrivial tensor product over $V/A$ is regular. However, $Z$ and hence $Z^\pi$, where $\pi$ is the canonical epimorphism from $G$ to $G/A$, are nonregular because $Z_D\setminus \{e\}\neq \varnothing$ and $Z_H\nsubseteq A$. Therefore $\mathcal{A}_{V/A}$ can not be cyclotomic or a nontrivial tensor product. Lemma~\ref{leungman} implies that $\mathcal{A}$ is trivial or $|\rad(\mathcal{A}_{V/A})|>1$. Observe that $|\rad(Z^\pi)|=1$ by Eq.~\eqref{radzp} and the assumption that $H$ is not an $\mathcal{A}$-subgroup and hence the latter is impossible. Thus, $\mathcal{A}_{V/A}$ is trivial. 

Let us prove that 
$$A\leq \rad(Z).$$
Since $\mathcal{A}_{V/A}$ is trivial, $Z$ contains an element $h_0$ of order~$2$. Let $z\in Z_D\setminus \{e\}\neq \varnothing$ and $h\in H$ such that $hz\in Z$. For every $i\in \mathbb{Z}$ coprime to~$p$, there exists $m_i\in \mathbb{Z}$ such that $m_i\equiv 1 \mod~2$ and $m_i\equiv i \mod~|z|$. One can see that $h_0^{m_i}=h_0\in X^{(m_i)}\cap X$ and hence $X^{(m_i)}=X$ by Lemma~\ref{burn}. Therefore $(hz)^{m_i}=hz^i\in X^{(m_i)}=X$. This implies that
\begin{equation}\label{radzp3}
hD_{|z|}^*\subseteq Z
\end{equation}
for all $z\in Z_D\setminus \{e\}$ and $h\in H$ such that $hz\in Z$. 

Observe that $Z\cap (A\times P^\#)\neq \varnothing$ because $\mathcal{A}_{V/A}=\mathcal{T}_{V/A}$. Eq.~\eqref{radzp3} yields that $P^\#\subseteq Z$ or $aP^\#\subseteq Z$. If both of these two inclusions hold, then $A\leq \rad(Z)$ by Lemma~\ref{intersection}. Thus, we may assume that exactly one of the above inclusions holds. Then $P^\#\subseteq Z$ or $P^\#\subseteq aZ$. Let $Z^\prime\in \{Z,aZ\}$ be such that $P^\#\subseteq Z^\prime$. From Eqs.~\eqref{radzp2} and~\eqref{radzp3}, it follows that $Z^\prime\setminus P\neq \varnothing$ and $P\leq \rad(Z^\prime\setminus P)$. Therefore $|\rad(Z)|=1$ and $Z^\prime=\langle Z\rangle^\#$ by Lemma~\ref{separ}. We conclude that Statement~$(3)$ of Lemma~\ref{2set} holds for $Z^\prime$, a contradiction to the assumption of this case. Thus, $A\leq \rad(Z)$. Due to triviality of $\mathcal{A}_{V/A}$, we obtain
\begin{equation}\label{zset}
Z=V\setminus A~\text{and}~\mathcal{A}_V=\mathbb{Z}A\wr \mathcal{T}_{V/A}.
\end{equation}
The latter equality yields that $P_1=V$.

If $V=G$, then due to Eq.~\eqref{zset}, $\mathcal{A}$ is the $U/L=A/A$-wreath product and $U/L$ satisfies Condition~$(1)$ from Theorem~\ref{hnot}. Further, we assume that $V<G$. If there exists $T\in \mathcal{S}(\mathcal{A})_{G\setminus V}$ such that $T\subseteq D$, then $\langle T \rangle$ is a nontrivial $\mathcal{A}$-subgroup of~$D$. So $\langle T \rangle \cap V\neq \varnothing$ and hence $V\subseteq \langle T \rangle \leq D$, a contradiction to the definition of~$Z$. Therefore every basic set outside $V$ is not a subset of~$D$ and hence Statement~$(1)$ or~$(2)$ of Lemma~\ref{2set} holds for every such basic set. If Statement~$(1)$ of Lemma~\ref{2set} holds for every basic set~$X$ outside~$V$, i.e. $P\leq \rad(X)$, then $V=P_1\leq\rad(X)$. This implies that $\mathcal{A}$ is the $U/L=V/V$-wreath product and $U/L$ satisfies Condition~$(1)$ from Theorem~\ref{hnot}. If Statement~$(2)$ of Lemma~\ref{2set} holds for some $T\in \mathcal{S}(\mathcal{A})_{G\setminus V}$, then $\langle (T^{[p]})_H \rangle=A$ because $A$ is a unique nontrivial $\mathcal{A}$-subgroup of~$H$. Together with Lemma~\ref{sch}, this yields that $W=\langle T^{[p]} \rangle$ is a nontrivial $\mathcal{A}$-subgroup of $A\times D$. Since $T$ lies outside~$V$, we have $(T^{[p]})_D\neq \{e\}$ and hence $|W|$ is divisible by~$p$. Therefore $W\cap Z\neq \varnothing$. Thus,
$$Z\subseteq W\leq A\times D,$$
a contradiction to the definition of~$Z$.
\end{proof}

\begin{prop}\label{refine}
If $H$ is not an $\mathcal{A}$-subgroup and $\mathcal{A}$ is a nontrivial $U/L$-wreath product for an $\mathcal{A}$-section $U/L$ such that $4 \nmid |U|$ and $|L|\neq 2$, then Theorem~\ref{hnot} holds. 
\end{prop}

\begin{proof}
Note that $U$ is cyclic because $4 \nmid |U|$. If $|\rad(\mathcal{A}_U)|=1$, then $U/L$ satisfies Condition~$(3)$ from Theorem~\ref{hnot} and we are done. From now on throughout the proof of this proposition, we assume that
\begin{equation}\label{radautriv} 
|\rad(\mathcal{A}_U)|>1.
\end{equation}

Suppose that $|U|$ is odd. Then $|U|=pq$ or $|U|$ is a $p$-power. In the former case, Eq.~\eqref{radautriv} and Lemma~\ref{pq} imply that $\mathcal{A}_U=\mathcal{A}_L\wr \mathcal{A}_{U/L}$. So $\mathcal{A}=\mathcal{A}_L\wr \mathcal{A}_{G/L}$ and Theorem~\ref{hnot} holds as required. In the latter one, $\mathcal{A}_U$ is the nontrivial $U_1/L_1$-wreath product for an $\mathcal{A}_U$-section $U_1/L_1$ of $U$ such that $L_1$ is the least nontrivial $\mathcal{A}_U$-subgroup and $|\rad(\mathcal{A}_{U_1})|=1$ by Eq.~\eqref{radautriv} and Lemma~\ref{cyclpwreath}. By the definitions of $U_1$ and $L_1$, we have $U_1\leq U$ and $L_1\leq L$. This implies that $\mathcal{A}$ is the $U_1/L_1$-wreath product and $U_1/L_1$ satisfies Condition~$(3)$ from Theorem~\ref{hnot} as required.

Now suppose that $|U|$ is even. Then $|U|=2m$, where $m$ is odd, because $4 \nmid |U|$. Let $A$ and $W$ be the Sylow $2$-subgroup and the Hall $2^\prime$-subgroup of $U$, respectively. If $|W|=1$, then $|L|=|U|=2$, a contradiction to the assumption $|L|\neq 2$ of the lemma. In what follows, we assume that 
$$|W|>1.$$

The least $\mathcal{A}_U$-subgroup containing $A$ and a maximal $\mathcal{A}_U$-subgroup of $W$ are denoted by $A_1$ and $W_1$, respectively. Assume that $|W_1|=1$. Then $\mathcal{A}_U$ is trivial or $\mathcal{A}_U=\mathcal{A}_{A_1}\wr \mathcal{A}_{U/A_1}$ by Lemma~\ref{nonpowernew1} applied to $\mathcal{A}=\mathcal{A}_U$, $H_1=W_1$, and $P_1=A_1$. In the former case, we obtain a contradiction to Eq.~\eqref{radautriv}. In the latter one, $A_1$ is the least nontrivial $\mathcal{A}_U$-subgroup because otherwise $A_1$ has a nontrivial $\mathcal{A}_U$-subgroup of odd order and hence $|W_1|>1$, a contradiction to the assumption $|W_1|=1$. So $A_1\leq L$. Together with $\mathcal{A}=\mathcal{A}_U\wr_{U/L}\mathcal{A}_{G/L}$ and $\mathcal{A}_U=\mathcal{A}_{A_1}\wr \mathcal{A}_{U/A_1}$, this yields that $\mathcal{A}=\mathcal{A}_{A_1}\wr \mathcal{A}_{G/A_1}$ and Theorem~\ref{hnot} holds as desired. In what follows, we assume that 
$$|W_1|>1.$$

\begin{lemm}\label{mw1}
With the above notation, Theorem~\ref{hnot} holds unless $L\geq M$ for some minimal nontrivial $\mathcal{A}_{W_1}$-subgroup $M$.
\end{lemm}

\begin{proof}
If $|L\cap W_1|>1$, then $L\cap W_1$ being an $\mathcal{A}_{W_1}$-subgroup contains a minimal nontrivial $\mathcal{A}_{W_1}$-subgroup and we are done. So we may assume that $|L\cap W_1|=1$. Since $|L|\neq 2$ (the assumption of the proposition) and $|W_1|>1$, this can happen only if $n=pq$, $U=W_1\times L$, $|W_1|$ is prime, $|L|$ is twice prime, and $L$ does not have a nontrivial $\mathcal{A}_L$-subgroup of odd order. Lemma~\ref{pq} applied to $\mathcal{A}_L$ implies that 
$$\mathcal{A}_L=\mathcal{T}_L~\text{or}~\mathcal{A}_L=\mathcal{A}_{L_0}\wr \mathcal{A}_{L/L_0}$$
for an $\mathcal{A}_L$-subgroup $L_0$ of order~$2$. 

In the former case, Lemma~\ref{leungman} applied to $\mathcal{A}_U$ yields that $\mathcal{A}_U=\mathcal{A}_{W_1}\otimes \mathcal{T}_L$. Therefore $|\rad(\mathcal{A}_U)|=1$, a contradiction to Eq.~\eqref{radautriv}. In the latter one, Eq.~\eqref{radautriv} and Lemma~\ref{leungman} applied to $\mathcal{A}_U$ yield that $\mathcal{A}_U$ is the nontrivial $(W_1L_0)/L_0$-wreath product. As $L_0\leq L$ and $\mathcal{A}$ is the $U/L$-wreath product, the $S$-ring $\mathcal{A}$ is the nontrivial $(W_1L_0)/L_0$-wreath product. Since $|L_0|=2$, we have $\mathcal{A}_{L_0}=\mathbb{Z}L_0$. So $\mathcal{A}_{W_1L_0}=\mathcal{A}_{W_1}\otimes \mathbb{Z}L_0$ by Lemma~\ref{tenspr}(2). Therefore $\mathcal{A}_{L_0}$ is $\otimes$-complemented in $\mathcal{A}_{W_1L_0}$. Thus, the section $(W_1L_0)/L_0$ satisfies Condition~$(2)$ from Theorem~\ref{hnot} and we are done.
\end{proof}

By Lemma~\ref{mw1}, we may assume till the end of the proof that 
$$L\geq M$$ 
for some minimal nontrivial $\mathcal{A}_{W_1}$-subgroup $M$. One can see that $\mathcal{A}_{W_1}$ is the $V/M$-wreath product (possibly, trivial), where $V$ is an $\mathcal{A}_{W_1}$-subgroup with $|\rad(\mathcal{A}_{V})|=1$. Indeed, this follows from Lemma~\ref{leungman} and Lemma~\ref{pq} if $|W_1|=pq$ and from Lemma~\ref{leungman} and Lemma~\ref{cyclpwreath} if $|W_1|$ is a $p$-power. Lemma~\ref{nonpowernew2} implies that 
\begin{equation}\label{starwa}
\mathcal{A}_{W_1A_1}=\mathcal{A}_{W_1} \star \mathcal{A}_{A_1}=(\mathcal{A}_{V}\wr_{V/M} \mathcal{A}_{W_1/M})\star \mathcal{A}_{A_1}.
\end{equation}

Let $W_1A_1=U$. If $W_1\cap A_1$ is trivial, then $\mathcal{A}_{A_1}$ is primitive and hence $\mathcal{A}_{A_1}=\mathcal{T}_{A_1}$ by Lemma~\ref{primitive}. Eq.~\eqref{starwa} implies that $\mathcal{A}_{U}=\mathcal{A}_{W_1} \otimes \mathcal{A}_{A_1}$ and consequently $\mathcal{A}_U$ is the $U_1/L_1=(VA_1)/M$-wreath product (possibly, trivial). As $M\leq L$, we conclude that $\mathcal{A}$ is the $U_1/L_1$-wreath product. Since $|\rad(\mathcal{A}_{V})|=1$ and 
$$\mathcal{A}_{U_1}=\mathcal{A}_{V}\otimes \mathcal{A}_{A_1}\cong \mathcal{A}_{V}\otimes \mathcal{T}_{A_1},$$
we obtain $|\rad(\mathcal{A}_{U_1})|=1$. Thus, $U_1/L_1$ satisfies Condition~$(3)$ from Theorem~\ref{hnot} and we are done.

Now let $W_1A_1<U$. This implies that 
$$W_1<W.$$ 
In this case, $|W_1|$ is a $p$-power and hence $M$ is the unique least $\mathcal{A}_{W_1}$-subgroup. Lemma~\ref{nonpowernew1} applied to $\mathcal{A}=\mathcal{A}_U$, $H_1=W_1$, and $P_1=A_1$ yields that $\mathcal{A}_U=\mathcal{A}_{W_1} \wr \mathcal{A}_{U/W_1}$ or 
\begin{equation}\label{gwrw1a1}
\mathcal{A}_U=\mathcal{A}_{W_1A_1} \wr_{(W_1A_1)/A_1} \mathcal{A}_{U/A_1}.
\end{equation}
In the former case, $\mathcal{A}$ is the nontrivial $V/M$-wreath product because $\mathcal{A}_{W_1}=\mathcal{A}_{V}\wr_{V/M} \mathcal{A}_{W_1/M}$ and $M\leq L$. Since $|\rad(\mathcal{A}_{V})|=1$, the section $V/M$ satisfies Condition~$(3)$ from Theorem~\ref{hnot} and we are done.

Suppose that Eq.~\eqref{gwrw1a1} holds. If $\mathcal{A}_{A_1}\neq \mathcal{T}_{A_1}$, then $A_1$ has a nontrivial $\mathcal{A}_{A_1}$-subgroup by Lemma~\ref{primitive} applied to $A_1$. This $\mathcal{A}_{A_1}$-subgroup is of odd order because $A_1$ is the least $\mathcal{A}_U$-subgroup containing $A$ and hence $|A_1\cap W_1|>1$. In view of Eqs.~\eqref{starwa} and~\eqref{gwrw1a1}, we conclude that $\mathcal{A}_U$ is the nontrivial $W_1/(W_1\cap A_1)$-wreath product. Since $M$ is the unique least $\mathcal{A}_{W_1}$-subgroup, we have $M\leq A_1\cap W_1$ and consequently $\mathcal{A}_U$ is the $W_1/M$-wreath product. As $\mathcal{A}_{W_1}=\mathcal{A}_{V}\wr_{V/M} \mathcal{A}_{W_1/M}$, we conclude that $\mathcal{A}_U$ is the nontrivial $V/M$-wreath product. Observe that $\mathcal{A}$ is the $V/M$-wreath product because $M\leq L$. The section $V/M$ satisfies Condition~$(3)$ from Theorem~\ref{hnot} because $|\rad(\mathcal{A}_{V})|=1$ and we are done.

Now let $\mathcal{A}_{A_1}=\mathcal{T}_{A_1}$. Then $|W_1\cap A_1|=1$. Eq.~\eqref{starwa} implies that
\begin{equation}\label{tensaw1} 
\mathcal{A}_{W_1A_1}=\mathcal{A}_{W_1}\otimes \mathcal{T}_{A_1}.
\end{equation}
If $A_1\neq A$, then $|W_1A_1|=2pq$, $U=W_1A_1$, and $L\in\{W_1,A_1\}$. So $U/L$ satisfies Condition~$(2)$ from Theorem~\ref{hnot} by Eq.~\eqref{tensaw1} and we are done. Further, we assume that $A_1=A$.

\begin{lemm}\label{rada}
With the above notation, let $X\in \mathcal{S}(\mathcal{A})$ be such that $X_D\nsubseteq W_1$. Then $A\leq \rad(X)$. 
\end{lemm} 

\begin{proof}
If $X\subseteq U$, then $X\subseteq U\setminus (W_1A_1)$. So $\rad(X)\geq A_1\geq A$ by Eq.~\eqref{gwrw1a1}. In the sequel, we assume that $X\nsubseteq U$. This implies that $L\leq \rad(X)$ because $\mathcal{A}$ is the $U/L$-wreath product. Let $R=\langle X \rangle$. Since $X_D\nsubseteq W_1$, we obtain 
$$R\cap (W\setminus W_1)\neq \varnothing.$$
Together with Eq.~\eqref{gwrw1a1}, this implies that $R\geq A_1\geq A$. The image of $Y\subseteq G$ under the canonical epimorphism from $G$ to $G/\rad(X)$ is denoted by $\overbar{Y}$. Clearly, $|\rad(\overbar{X})|=1$ and hence $\mathcal{A}_{\overbar{R}}$ is indecomposable.

Assume the contrary that $A\nleq \rad(X)$. Then $\overbar{R}\geq \overbar{A}$. Due to Eq.~\eqref{gwrw1a1}, we have $\rad(X)\cap D\leq W_1$ and 
\begin{equation}\label{radya}
\overbar{A}\leq \rad(\overbar{Y})
\end{equation}
 for every $Y\in \mathcal{S}(\mathcal{A})_{R}$ inside $R\cap (U\setminus (W_1A))$. Observe that $\mathcal{A}_{\overbar{R}}$ is nontrivial because $\overbar{A}$ is a proper nontrivial $\mathcal{A}_{\overbar{R}}$-subgroup. 

Suppose that $\overbar{R}$ is cyclic. Then $\overbar{R}$ is isomorphic to $C_{p^i}$ or $C_{2p^i}$, $i\geq 1$. As $\mathcal{A}_{\overbar{R}}$ is indecomposable, we conclude that $\rad(\mathcal{A}_{\overbar{R}})$ is trivial. So $\mathcal{A}_{\overbar{R}}$ is cyclotomic or a nontrivial tensor product by Lemma~\ref{leungman}. However, every basic set of such $S$-ring over $\overbar{R}$ has trivial radical which contradicts to Eq.~\eqref{radya}. 

Now suppose that $\overbar{R}$ is noncyclic. Then $\overbar{R}\cong E_4\times C_{p^i}$, $i\geq 1$. In view of Eq.~\eqref{radya}, the Hall $2^\prime$-subgroup of $\overbar{R}$ is not an $\mathcal{A}_{\overbar{R}}$-subgroup. So the Sylow $2$-subgroup of $\widehat{G}$ is not an $\widehat{\mathcal{A}_{\overbar{R}}}$-subgroup by Lemma~\ref{dual}(1). Since $\mathcal{A}_{\overbar{R}}$ is indecomposable and nontrivial, so is $\widehat{\mathcal{A}_{\overbar{R}}}$. Therefore $\widehat{\mathcal{A}_{\overbar{R}}}$ is a nontrivial tensor product by Proposition~\ref{hnot1} applied to $\widehat{\mathcal{A}_{\overbar{R}}}$. Thus,
\begin{equation}\label{tensorr}  
\mathcal{A}_{\overbar{R}}=\mathcal{A}_{\overbar{R}_1}\otimes \mathcal{A}_{\overbar{R}_2}
\end{equation}
for some $\mathcal{A}_R$-subgroups $R_1$ and $R_2$ by Lemma~\ref{dual}(3). As $\rad(X)\geq L$ and $|L|\neq 2$, we may assume without loss of generality that $|\overbar{R}_1|$ is a power of~$2$. Due to Eqs.~\eqref{radya} and~\eqref{tensorr}, $\overbar{R}_1\ngeq \overbar{A}$. Therefore $\overbar{R}_2$ is a cyclic group containing $\overbar{A}$. Lemma~\ref{leungman} and Eq.~\eqref{radya} yield that $|\rad(\mathcal{A}_{\overbar{R}_2})|\neq 1$. Together with Eq.~\eqref{tensorr}, this implies that $\mathcal{A}_{\overbar{R}}$ is decomposable, a contradiction.
\end{proof}

From Lemma~\ref{rada} it follows that 
$$\mathcal{A}=\mathcal{A}_{W_1A}\wr_{(W_1A)/A} \mathcal{A}_{G/A}~\text{or}~\mathcal{A}=\mathcal{A}_{W_1H}\wr_{(W_1H)/A} \mathcal{A}_{G/A}.$$ 
In the first case, $\mathcal{A}_A$ is $\otimes$-complemented in $\mathcal{A}_{W_1}$ by Eq.~\eqref{tensaw1}. So the section $(W_1A)/A$ satisfies Condition~$(2)$ from Theorem~\ref{hnot} and we are done. 

In the second case, $W_1H$ is an $\mathcal{A}$-subgroup. Let $a$ be a nontrivial element of $A$, $b\in H\setminus A$, and $T$ a basic set of $\mathcal{A}$ containing~$b$. Clearly, $W_1H=W_1A\times \langle b \rangle$, $W_1A$ is an $\mathcal{A}_{W_1H}$-subgroup, and $W_1$ is the largest $\mathcal{A}_{W_1H}$-subgroup of odd order. Therefore the $S$-ring $\mathcal{A}_{W_1H}$, the basic set $T$, and the subgroup $W_1$ satisfy~\cite[Hypothesis~$5.1$]{HK}. The structure of a basic set containing an involution of an $S$-ring satisfying~\cite[Hypothesis~$5.1$]{HK} is described in~\cite[Lemma~$5.5$]{HK}. Applying this lemma to $T$ and using the inclusion $\{a\}=A^\#\in \mathcal{S}(\mathcal{A})$ one can deduce that 
$$T=W_0Ab~\text{or}~T=W_0b$$
for some $\mathcal{A}_{W_1}$-subgroup~$W_0$. To complete the proof of the lemma, let us consider the above two cases separately.

Let $T=W_0Ab$. As $H=\langle Ab \rangle$ is not an $\mathcal{A}$-subgroup, $T\neq Ab$ and hence $|W_0|>1$. So $W_0$ satisfies the condition of~\cite[Lemma~$5.4$]{HK}. Applying this lemma to $\mathcal{A}_{W_1H}$, we conclude that $\mathcal{A}_{W_1H}$ is the $(W_1A)/(W_0A)$-wreath product. Therefore $\mathcal{A}$ is the $(W_1A)/A$-wreath product. Again, the section $(W_1A)/A$ satisfies Condition~$(2)$ from Theorem~\ref{hnot} by Eq.~\eqref{tensaw1} and we are done.

Now let $T=W_0b$, then $W_2=\langle T,W_1\rangle$ is an $\mathcal{A}_{W_1H}$-subgroup such that $|W_2\cap A|=1$ and $W_1H=W_2\times A$. Lemma~\ref{tenspr}(2) implies that $\mathcal{A}_{W_1H}=\mathcal{A}_{W_2}\otimes \mathcal{A}_{A}$. Therefore $(W_1H)/A$ satisfies Condition~$(2)$ from Theorem~\ref{hnot} as desired.  
\end{proof}

Theorem~\ref{hnot} immediately follows from Propositions~\ref{hnot1} and~\ref{refine}. 

\begin{theo}\label{dnot}
Suppose that $D$ is not an $\mathcal{A}$-subgroup. Then $\mathcal{A}$ is trivial, or a nontrivial tensor product, or a nontrivial $U/L$-wreath product for some $\mathcal{A}$-section $U/L$ satisfying one of the following conditions:
\begin{enumerate}

\tm{1} $|U/L|\leq 2$;

\tm{2} $\mathcal{A}_{U/L}$ is $\otimes$-complemented in $\mathcal{A}_{G/L}$;

\tm{3} $4 \nmid |G/L|$ and $|\rad(\mathcal{A}_{G/L})|=1$.

\end{enumerate}
\end{theo}

\begin{proof}
Let $\widehat{\mathcal{A}}$ be the $S$-ring dual to $\mathcal{A}$ over $\widehat{G}\cong G$. Since $D$ is not an $\mathcal{A}$-subgroup, the group $D^\bot\cong E_4$ is not an $\widehat{\mathcal{A}}$-subgroup by Lemma~\ref{dual}(1). So Theorem~\ref{hnot} holds for $\widehat{\mathcal{A}}$. If $\widehat{\mathcal{A}}$ is trivial, then  so is $\mathcal{A}$ and we are done. If $\widehat{\mathcal{A}}$ is a nontrivial tensor product, then so is $\mathcal{A}$ by Lemma~\ref{dual}(3) and we are done. 

Let $\widehat{\mathcal{A}}$ be a nontrivial $\widehat{U}/\widehat{L}$-wreath product for some $\widehat{\mathcal{A}}$-section $\widehat{U}/\widehat{L}$ satisfying one of Conditions~$(1)$-$(3)$ from Theorem~\ref{hnot}. Then $\mathcal{A}$ is a nontrivial $U/L$-wreath product, where $L=\widehat{U}^\bot$ and $U=\widehat{L}^\bot$, by Lemma~\ref{dual}(4). Observe that $|U/L|=|\widehat{U}/\widehat{L}|$ by Lemma~\ref{dual}(1). So if $|\widehat{U}/\widehat{L}|\leq 2$, then $|U/L|\leq 2$ as desired.

Lemma~\ref{dual}(2) implies that $\widehat{\mathcal{A}_{G/L}}=\widehat{\mathcal{A}}_{\widehat{U}}$ and $\widehat{\mathcal{A}_{U/L}}=\widehat{\mathcal{A}}_{\widehat{U}/\widehat{L}}$. So if $\widehat{L}$ is $\otimes$-complemented in $\widehat{\mathcal{A}}_{\widehat{U}}$, then $\mathcal{A}_{U/L}$ is $\otimes$-complemented in $\mathcal{A}_{G/L}$ by Lemma~\ref{dual}(3) as required. 

Suppose that $4 \nmid |\widehat{U}|$ and $|\rad(\widehat{\mathcal{A}}_{\widehat{U}})|=1$. The first part together with $L=\widehat{U}^\bot$ and Lemma~\ref{dual}(1) imply that $4 \nmid |G/L|$. Since $|\rad(\widehat{\mathcal{A}}_{\widehat{U}})|=1$, the $S$-ring $\widehat{\mathcal{A}}_{\widehat{U}}$ is not a nontrivial generalized wreath product by Lemma~\ref{leungman}(2). So the latter holds for $\mathcal{A}_{G/L}$ by Lemma~\ref{dual}(4) because $\widehat{\mathcal{A}_{G/L}}=\widehat{\mathcal{A}}_{\widehat{U}}$. Therefore $|\rad(\mathcal{A}_{G/L})|=1$ by Lemma~\ref{leungman}(2) and we are done.
\end{proof}

\subsection{Dense $S$-rings}

The lemma below is straightforward. 

\begin{lemm}\label{e4}
Every $S$-ring over $H\cong E_4$ is normal and equals $\cyc(K,H)$, where $K\leq \aut(H)$ is of order at most~$3$.
\end{lemm}

\begin{rem}\label{e4rem}
In Lemma~\ref{e4}, we have $\cyc(K_H,H)=\mathbb{Z}H$ if $|K_H|=1$, $\cyc(K_H,H)\cong\mathbb{Z}C_2\wr \mathbb{Z}C_2$ if $|K_H|=2$, and $\cyc(K_H,H)=\mathcal{T}_H$ if $|K_H|=3$.
\end{rem}

Further throughout this subsection, we assume that $H$ and $D$ are $\mathcal{A}$-subgroups. Denote two distinct elements from $H^\#$ by $a$ and $b$. Given $X\in \mathcal{S}(\mathcal{A})$, put $\lambda_X=|X\cap Hx|$, where $x\in X$. Due to Lemma~\ref{intersection}, $\lambda_X$ does not depend on $x$. Clearly, $1\leq \lambda_X\leq |X_H|\leq 3$.

\begin{lemm}\label{hpasubgroups1}
Let $\mathcal{A}$ be dense and $X\in \mathcal{S}(\mathcal{A})_{G\setminus (H\cup D)}$. Then $X=X_H\times X_D$ or $\lambda_X=1$.  
\end{lemm}

\begin{proof}
Assume that $X\neq X_H\times X_D$. Then $\lambda_X<|X_H|$. We are done if $\lambda_X=1$. So it remains to consider only the case when $1<\lambda_X<|X_H|\leq 3$. Since $H\cong E_4$, this is possible only if $X_H=H^\#=\{a,b,ab\}$ and $\lambda_X=2$. 

Assume that two latter equalities hold. Then $Y=(X_H\times X_D)\setminus X$ is an $\mathcal{A}$-set containing exactly one element from $Hx$ for every $x\in X_D$. Together with Lemma~\ref{tenspr}(1) and the equality $X_H=H^\#$, this implies that $Y$ is exactly a basic set of $\mathcal{A}$ with 
\begin{equation}\label{lambda1}
\lambda_Y=1. 
\end{equation}

Let $x,x^\prime\in X_D$ be of the same order. Then there is a positive integer $m$ coprime to~$2n$ such that $x^m=x^\prime$. Since $\lambda_X=2$, the sets $H_x=Xx^{-1}\cap H$ and $H_{x^\prime}=X(x^\prime)^{-1}\cap H$ are subsets of $H^\#$ of size~$2$. Therefore there is $h\in H_x\cap H_{x^\prime}$. One can see that $(hx)^m=hx^\prime\in X^{(m)}\cap X$. So $X^{(m)}=X$ by Lemma~\ref{burn}. This yields that 
$$(H_xx)^m=H_xx^m=H_xx^\prime=H_{x^\prime}x^\prime$$
and hence $H_x=H_{x^\prime}$. 

Let $l$ be the maximum of orders of the elements from~$X_D$. Due to the above discussion, there is $H_0\subseteq H$ such that $|H_0|=2$ and $H_x=H_0$ for every $x\in (X_D\cap D_l^*)$. So
\begin{equation}\label{d1} 
X_D\cap D_l^*\subseteq Y_{h_0},
\end{equation}
where $h_0$ is a unique nontrivial element from $H\setminus H_0$. Let $h\in H_0$. From Eq.~\eqref{lambda1} it follows that $Y_{h_0}\cap Y_{h}=\varnothing$. Together with $Y_{h_0}\cup Y_h \subseteq X_D$ (Lemma~\ref{tenspr}(1)) and Eq.~\eqref{d1}, this implies that
\begin{equation}\label{d2}
Y_h\subseteq X_D\setminus D_l^*.
\end{equation}

Observe that $|Y_{h_0}|=|Y_h|$ by Lemma~\ref{intersection}. In view of Eqs.~\eqref{d1} and~\eqref{d2}, to obtain a contradiction, it suffices to show that
\begin{equation}\label{d3}
|X_D\cap D_l^*|>|X_D\setminus D_l^*|.
\end{equation}
The set $X_D$ is a basic set of $\mathcal{A}_D$ (Lemma~\ref{tenspr}(1)). If $|X_D\setminus D_l^*|=0$, then Eq.~\eqref{d3} is obvious. Suppose that $|X_D\setminus D_l^*|>0$, i.e. $X_D$ is nonregular. If $n=pq$, then Lemma~\ref{pq} implies that $\mathcal{A}_D$ is trivial or a nontrivial generalized wreath product. So $X_D=D^\#$ or $X_D$ is a union of some $L$-cosets, where $L=\rad(X_D)$ is a nontrivial $\mathcal{A}_D$-subgroup. In the former case, $|X_D\cap D_l^*|=(p-1)(q-1)$ and $|X_D\setminus D_l^*|=p+q-2$, whereas in the latter one $|X_D\cap D_l^*|=r(|L|-1)$ and $|X_D\setminus D_l^*|=r$, where $r=|X_D/L|$. If $n=p^k$, then by Lemma~\ref{pnonreg}, we obtain $X_D=U\setminus L$ for some $\mathcal{A}$-subgroups $U\cong C_{p^l}>L$ and consequently $|X_D\cap D_l^*|=p^{l-1}(p-1)$ and $|X_D\setminus D_l^*|=p^{l-1}-|L|$. Since $p$ and $q$ are odd, Eq.~\eqref{d3} holds in all cases and we are done. 
\end{proof}

\begin{theo}\label{hpasubgroups2}
Let $\mathcal{A}$ be dense. Suppose that $|\rad(\mathcal{A}_D)|=1$ and $\mathcal{A}\neq \mathcal{A}_H\otimes \mathcal{A}_D$. Then:
\begin{enumerate}

\tm{1} $\mathcal{A}$ is normal, cyclotomic, and $2$-minimal;

\tm{2} $|\aut(\mathcal{A})^H|\leq 12$.

\end{enumerate} 
\end{theo}

\begin{proof}
By the condition $\mathcal{A}\neq \mathcal{A}_H\otimes \mathcal{A}_D$ and Lemma~\ref{tenspr}(2), we have $\mathcal{A}_H\neq \mathbb{Z}H$. Then Lemma~\ref{e4} and Remark~\ref{e4rem} imply that $\mathcal{A}_H=\cyc(K_H,H)$, where $K_H\leq\aut(H)$ is of order~$2$ or~$3$, and a unique nonsingleton basic set of $\mathcal{A}_H$ is of size~$|K_H|$. We divide the rest of the proof into two cases depending on whether $\mathcal{A}_D$ is cyclotomic or not. 

\hspace{5mm}

\noindent \textbf{Case~1: $\mathcal{A}_D$ is noncyclotomic.} The condition $|\rad(\mathcal{A}_D)|=1$ of the theorem and Lemma~\ref{leungman} imply that $\mathcal{A}_D$ is trivial whenever $n=p^k$. The same arguments and Lemma~\ref{pq} imply that $\mathcal{A}_D$ is trivial in case $n=pq$. Thus, $\mathcal{A}_D$ is trivial in any case and hence $D^\#\in \mathcal{S}(\mathcal{A}_D)$. Moreover, $n$ is not prime by Lemma~\ref{cyclprime}. Therefore $n=pq$ or $n=p^k$ for $k\geq 2$.

Let $n=pq$. Then $H$ is a maximal $\mathcal{A}$-subgroup whose order is not divisible by~$p$ and $D$ is the least $\mathcal{A}$-subgroup whose order is divisible by~$p$. By Lemma~\ref{nonpowernew1} applied to $\mathcal{A}$, $H$, and $D$, we have $\mathcal{A}=\mathcal{A}_H\star\mathcal{A}_D$. Since $H\cap D$ is trivial, we obtain $\mathcal{A}=\mathcal{A}_H\otimes\mathcal{A}_D$, a contradiction to the assumption of the theorem.

Now let $n=p^k$, $k\geq 2$. Since $\mathcal{A}\neq \mathcal{A}_H\otimes \mathcal{A}_D$ by the condition of the theorem, there is $X\in \mathcal{S}(\mathcal{A})_{G\setminus (H\cup D)}$ such that 
\begin{equation}\label{nondirect}
X\neq X_H\times X_D.
\end{equation}
Note that $X_H\in\mathcal{S}(\mathcal{A}_H)$ and $X_D=D^\#$ by Lemma~\ref{tenspr}(1). Eq.~\eqref{nondirect} implies that $|X_H|>1$. So $|X_H|\in\{2,3\}$ by Lemma~\ref{e4}. Eq.~\eqref{nondirect} and Lemma~\ref{hpasubgroups1} yield that $\lambda_X=1$. If $P\nleq \rad(X_{h,p^l})$ (see Notation for $X_{h,p^l}$) for some $h\in X_H$ and $l\in\{2,\ldots,k\}$, then the set $Y=X^{[p]}$, which is an $\mathcal{A}$-set by Lemma~\ref{sch}, lies outside $H\cup D$ and $Y_D\subsetneq D^\#$, a contradiction to Lemma~\ref{tenspr}(1). Therefore 
\begin{equation}\label{trivd1}
P\leq \rad(X_h\setminus X_{h,p}) 
\end{equation}
for every $h\in X_H$.

Lemma~\ref{orbit} applied to $X$, $h\in X_H$, and $l=p$ implies that the sets $X_{h,p}$, $h\in X_H$, are orbits of some $K\leq \aut(P)$. This yields that all of these sets are of the same size $|K|$. Since $X_D=D^\#$, we have 
$$\bigcup \limits_{h\in X_H} X_{h,p}=P^\#$$
and hence
$$\sum \limits_{h\in X_H} |X_{h,p}|=p-1.$$
Thus,
\begin{equation}\label{trivd2}
|X_{h,p}|=(p-1)/|X_H|.
\end{equation} 

Further, we are going to compute $\sum \limits_{h\in X_H}\underline{X_{h,p}}^2$. Recall that $|X_H|\in\{2,3\}$. If $|X_H|=2$ and $(p-1)/2$ is odd, then $\cyc(K,P)$ is an antisymmetric $S$-ring of rank~$3$. Using the formulas for the intersection numbers of an antisymmetric association scheme of rank~$3$ (see, e.g.~\cite[Exercise~2.7.57]{CP}), one can compute that 
\begin{equation}\label{trivd3}
\sum \limits_{h\in X_H}\underline{X_{h,p}}^2=\frac{p-1}{2}\underline{P}^\#.
\end{equation}
If $|X_H|=2$ and $(p-1)/2$ is even or $|X_H|=3$, then $X_{h,p}=X_{h,p}^{-1}$ for every $h\in X_H$ and the required sum is a special case of the sum computed in~\cite[Lemma~5.2]{Ry4}. So applying this lemma, we obtain
\begin{equation}\label{trivd4}
\sum \limits_{h\in X_H}\underline{X_{h,p}}^2=(p-1)e+(\frac{p-1}{|X_H|}-1)\underline{P}^\#.
\end{equation}

Put 
$$\xi=\sum \limits_{h\in X_H} \underline{X_h}^2\in \mathbb{Z}D.$$
One can compute directly that
$$\underline{X}^2=\xi+\sum \limits_{h,h^\prime\in X_H} hh^\prime\underline{X_h}\cdot\underline{X_{h^\prime}}\in \mathcal{A}.$$
Observe that only elements from $G\setminus D$ enter the second sum on the right-hand side of the above equality. Since $D$ is an $\mathcal{A}$-subgroup, we conclude that $\xi\in \mathcal{A}$. Further, one can compute $\xi$ as follows:
$$\xi=\sum \limits_{h\in X_H} (\underline{X_h\setminus X_{h,p}}+\underline{X_{h,p}})^2=\sum \limits_{h\in X_H} (\underline{X_h\setminus X_{h,p}})^2+
2\sum \limits_{h\in X_H} \underline{X_h\setminus X_{h,p}}\cdot \underline{X_{h,p}}+\sum \limits_{h\in X_H} \underline{X_{h,p}}^2.$$
Each element appearing in the first of three last sums enters with a coefficient divisible by~$p$ due to Eq.~\eqref{trivd1} and hence 
$$\sum \limits_{h\in X_H} (\underline{X_h\setminus X_{h,p}})^2=p\eta$$ 
for some $\eta\in \mathbb{Z}D$. Recall that $X_{h,p}\subseteq P\leq \rad(X_h\setminus X_{h,p})$ for every $h\in X_H$, where the latter inequality holds by Eq.~\eqref{trivd1}. Therefore  
$$\sum \limits_{h\in X_H} \underline{X_h\setminus X_{h,p}}\cdot \underline{X_{h,p}}=\sum \limits_{h\in X_H} |X_{h,p}|\underline{X_h\setminus X_{h,p}}=2(p-1)/|X_H|\sum \limits_{h\in X_H} (\underline{X_h}-\underline{X_{h,p}}),$$  
where the first equality follows from Eq.~\eqref{inrad0}, whereas the second one from Eq.~\eqref{trivd2}. Thus,
$$\xi=p\eta+2(p-1)/|X_H|\sum \limits_{h\in X_H} (\underline{X_h}-\underline{X_{h,p}})+\sum \limits_{h\in X_H} \underline{X_{h,p}}^2.$$

The above expression for $\xi$ implies that every element from $D\setminus P$ enters $\xi$ with the coefficient equal to~$2(p-1)/|X_H|$ modulo~$p$. On the other hand, this expression and Eqs.~\eqref{trivd3} and~\eqref{trivd4} imply that every element from $P^\#$ enters $\xi$ with coefficient equal to~$(p-1)/|X_H|$ or~$(p-1)/|X_H|-1$ modulo~$p$. It is easy to see that none of the above numbers is equal to~$2(p-1)/|X_H|$ modulo~$p$. However, each element from $D^\#$ must enter $\xi$ with the same coefficient because $D^\#\in \mathcal{S}(\mathcal{A})$, a contradiction.

\hspace{5mm}

\noindent \textbf{Case~2: $\mathcal{A}_D$ is cyclotomic.} In this case, all basic sets of $\mathcal{A}_D$ are regular. Let $K_D\leq \aut(D)$ be such that $\mathcal{A}_D=\cyc(K_D,D)$.

\begin{claim1}
Let $X\in \mathcal{S}(\mathcal{A})_{G\setminus (H\cup D)}$. Then $X^{(r)}\in \mathcal{S}(\mathcal{A})$ for every $r\in \mathbb{Z}$.
\end{claim1}

\begin{proof}
Without loss of generality, we may assume that $r$ is prime. If $r$ is coprime to~$2n$, then the claim follows from Lemma~\ref{burn}. Lemma~\ref{tenspr}(1) implies that $X_D\in \mathcal{S}(\mathcal{A}_D)$. If $r=2$, then $X^{(r)}=X^{(2)}=X_D^{(2)}\in \mathcal{S}(\mathcal{A}_D)$ by Lemma~\ref{burn}.

It remains to consider the case when $r$ is an odd prime divisor of~$n$. Put $Y=X^{(r)}$. Let $R$ be the (unique) subgroup of $D$ of order~$r$. Since $\mathcal{A}_D$ is cyclotomic, $R$ is an $\mathcal{A}$-subgroup. Observe that $R\nleq \rad(X_D)$. Indeed, if $n=p^k$, then this follows from Lemma~\ref{trivradorb}(2), whereas if $n=pq$, then this immediately follows from the the fact that $\mathcal{A}$ is cyclotomic with trivial radical. Therefore $R\nleq \rad(X)$ and hence $Y=X^{[r]}$. Lemma~\ref{sch} yields that $Y$ is an $\mathcal{A}$-set. 

If $Y_D$ is trivial, then $Y=X_H$ and we are done by Lemma~\ref{tenspr}(1). Further, we assume that $Y_D$ is nontrivial. Let us show that
\begin{equation}\label{bsyd}
Y_D\in \mathcal{S}(\mathcal{A}_D)
\end{equation}
One can see that $Y_D=X_D^{(r)}$. If $n=p^k$, then $|Y_D|=|X_D|$ because all the elements of $X_D$ lie in pairwise distinct $R$-cosets (Lemma~\ref{trivradorb}(1)). From Lemma~\ref{trivradorb}(2) it follows that all nontrivial basic set of $\mathcal{A}_D$ has the same size and hence Eq.~\eqref{bsyd} holds. If $n=pq$, then $R$ and the Hall $r^\prime$-subgroup $U$ of $D$ are $\mathcal{A}$-subgroups because $\mathcal{A}_D$ is cyclotomic. So $H$, $U$, and $R$ are $\mathcal{A}$-subgroups such that $G=H\times U \times R$. As $R\nleq \rad(X)$, one can see that $Y=X_{H\times U}$ and consequently $Y_D=(X_{H\times U})_D=X_U$. Therefore $Y_D\in \mathcal{S}(\mathcal{A}_U)$ by Lemma~\ref{tenspr}(1) applied to $G_1=H\times R$ and $G_2=U$. Thus, Eq.~\eqref{bsyd} holds.

To complete the proof, let us show that $Y\in \mathcal{S}(\mathcal{A})$. Assume the contrary that $Y\notin \mathcal{S}(\mathcal{A})$. Then there exists $Z\in \mathcal{S}(\mathcal{A})$ such that $Z\subsetneq Y$. Recall that $\lambda_X=|Hx\cap X|$ does not depend on $x\in X$. We have $\lambda_X\neq 1$. Indeed, if $\lambda_X=1$, then $\lambda_Y=1$ and hence $|Y|=|Y_D|$. Together with $Z\subsetneq Y$, this implies that $Z_D\subsetneq Y_D$. Due to Eq.~\eqref{bsyd}, we conclude that $Z_D\notin \mathcal{S}(\mathcal{A}_D)$, a contradiction to Lemma~\ref{tenspr}(1).

Since $\lambda_X\neq 1$, we have $X=X_H\times X_D$ by Lemma~\ref{hpasubgroups1} and $\lambda_X=|X_H|\geq 2$. So $Y=Y_H\times Y_D=X_H\times Y_D$. As $Z_H,Y_H\in \mathcal{S}(\mathcal{A}_H)$ and $Z_H\subseteq Y_H$, we have $Z_H=Y_H$. Besides, $Z_D=Y_D$ by Eq.~\eqref{bsyd}. Since $Z\subsetneq Y$, we have $\lambda_Z=1$ by Lemma~\ref{hpasubgroups1}. 

If $n=p^k$, then given $x\in X_D$, all the elements from $Y_Dx$ lie in the same $\langle Y_D \rangle$-coset. So by Lemma~\ref{trivradorb}(1), they lie in pairwise distinct basic sets of $\mathcal{A}_D$. Together with $|X_D|=|Y_D|$, this implies that there are at least $|X_D|$ pairwise distinct $T_D\in \mathcal{S}(\mathcal{A}_D)$ with $c_{X_D Y_D}^{T_D}\geq 1$. Therefore there are at least $|X_D|$ pairwise distinct $T\in \mathcal{S}(\mathcal{A})$ with $c_{X_D Z}^{T}\geq 1$. Every such $T$ is rationally conjugate to~$X$ (Lemma~\ref{burn}) and hence $|T|=|X|=|X_H||X_D|$. Thus, at least $|X_H||X_D|^2$ elements enter the element $\underline{X_D}\cdot \underline{Z}$. On the other hand, $|Z|=|Z_D|=|Y_D|=|X_D|$ because $\lambda_Z=1$ and consequently exactly $|X_D|^2$ elements enter the element $\underline{X_D}\cdot \underline{Z}$, a contradiction to $|X_H|=|K_H|\geq 2$.

If $n=pq$, then $X^\pi\cap Z^\pi \neq \varnothing$, where $\pi$ is the canonical epimorphism from $G$ to $G/R$. Since $X^\pi,Z^\pi\in \mathcal{S}(\mathcal{A}_{G/R})$, we conclude that $X^\pi=Z^\pi$. On the other hand, $|X^\pi \cap H^\pi x|=|X_H|\geq 2$ for every $x\in X^\pi$, whereas $|Z^\pi \cap H^\pi z|=1$ for every $z\in Z^\pi$ by Lemma~\ref{intersection}, a contradiction.    
\end{proof}

Since $\mathcal{A}_H=\cyc(K_H,H)$, where $K_H\leq\aut(H)$ is of order~$2$ or~$3$, and $\mathcal{A}_D=\cyc(K_D,D)$, one can choose a basic set $X\in \mathcal{S}(\mathcal{A})_{G\setminus (H\cup D)}$ such that 
$$|X_H|=|K_H|>1$$
and $X_D$ consists of generators of~$D$. Clearly, $X_H\in \orb(K_H,H)$ and $X_D\in\orb(K_D,D)$. 

\begin{claim2}
Let $Y\in \mathcal{S}(\mathcal{A})\setminus (\mathcal{S}(\mathcal{A}_H)\cup \mathcal{S}(\mathcal{A}_D))$. Then $Y=X^{(r)}$ for some $r\in \mathbb{Z}$ or $|Y_H|=1$ and $Y=Y_H \times Y_D$. 
\end{claim2}

\begin{proof}
Let $Y_H=X_H$. Since $Y\notin \mathcal{S}(\mathcal{A}_H)$ by the condition of the claim and $X_D$ consists of generators of $D$, we conclude that $Y\cap X^{(r)}\neq \varnothing$ for some $r\in \mathbb{Z}$. Claim~$1$ implies that $X^{(r)}\in \mathcal{S}(\mathcal{A})$ and hence $Y=X^{(r)}$ as required. Now let $Y_H\neq X_H$. Since $X_H,Y_H\subseteq H^\#$ and $|X_H|>1$, we obtain $|Y_H|=1$ and hence $Y=Y_H \times Y_D$ as required. 
\end{proof}

If $X=X_H\times X_D$, then $Y=Y_H\times Y_D$ for every $Y\in \mathcal{S}(\mathcal{A})_{G\setminus (H\cup D)}$ by Claim~$2$ and hence $\mathcal{A}=\mathcal{A}_H\otimes \mathcal{A}_D$ which contradicts to the supposition of the theorem. Therefore $X\neq X_H\times X_D$. Together with Lemma~\ref{hpasubgroups1}, this implies that
$$\lambda_X=1.$$

Further, we are going to construct $K\leq \aut(G)$ such that
\begin{equation}\label{xorbk} 
X\in \orb(K,G).
\end{equation}
Since $X_D\in \orb(K_D,D)$, the set $X_D$ is regular. So $X_h=X_{h,p^k}$ for every $h\in X_H$. Lemma~\ref{orbit} implies that $X_h$ is an orbit of some $K^0_D\leq \aut(D)$ for every $h\in H$. Clearly, $K^0_D\leq K_D$, $K_D$ acts transitively on the set $\{X_h:~h\in X_H\}=\orb(K^0_D,X_D)$, and $K^0_D$ is the kernel of this action. As $\lambda_X=1$, we have $|K_D:K^0_D|=|X_H|=|K_H|\leq 3$. Therefore $K_H\cong K_D/K^0_D\cong C_{|X_H|}$. Let $\sigma_0$ be a generator of $K_H$. Since $|K_D:K^0_D|=|K_H|\leq 3$, there is $\tau_0\in K_D$ such that
\begin{equation}\label{agree} 
X_h^{\tau_0}=X_{h^{\sigma_0}} 
\end{equation}
for every $h\in X_H$. Let $\psi$ be an isomorphism from $K_H$ to $K_D/K^0_D$ such that $\sigma_0^\psi=K^0_D\tau_0$ and  
$$K=\{(\sigma,\tau)\in K_H\times K_D:~(\sigma)^\psi=K^0_D\tau\}.$$
Eq.~\eqref{agree} yields that $X$ is $K$-invariant. As $K_H$ and $K_D$ are transitive on $X_H$ and $X_D$, respectively, so is $K$ on $X$. Thus, Eq.~\eqref{xorbk} holds.

To prove that $\mathcal{A}$ is cyclotomic, it is enough to show that $\mathcal{A}=\cyc(K,G)$. One can see that $\mathcal{A}_H=\cyc(K,G)_H$ ($\mathcal{A}_D=\cyc(K,G)_D$, respectively) because $K^H=K_H$ ($K^D=K_D$, respectively). Let $Y\in \mathcal{S}(\mathcal{A})\setminus (\mathcal{S}(\mathcal{A}_H)\cup \mathcal{S}(\mathcal{A}_D))$. Then $Y=X^{(r)}$ for some $r\in \mathbb{Z}$ or $|Y_H|=1$ and $Y=Y_H \times Y_D$ by Claim~$2$. In the former case, $Y\in \orb(K,G)$ by Eq.~\eqref{xorbk}, whereas in the latter one, by the definition of $K$. Thus, every basic set of $\mathcal{A}$ is an orbit of $K$, i.e. $\mathcal{A}=\cyc(K,G)$ as desired.

Now let us prove that $\mathcal{A}$ is $2$-minimal. Let $x\in X$. It is enough to show that the pointwise stabilizer $\aut(\mathcal{A})_{e,x}$ is trivial. Indeed, then $X$ is a faithful regular orbit of $\aut(\mathcal{A})_e$ and $\mathcal{A}$ is $2$-minimal by~\cite[Lemma~8.2]{MP}. It suffices to verify that every $f\in \aut(\mathcal{A})_{e,x}$ is trivial. 

\begin{claim3}
If $y^f=y$ for some $y\in G\setminus H$, then $(Hy)^f=Hy$ and $f^{Hy}=\id_{Hy}$.
\end{claim3}

\begin{proof}
As $H$ is an $\mathcal{A}$-subgroup, each $H$-coset is a block of $\aut(\mathcal{A})$. One can see that $y=y^f\in Hy\cap (Hy)^f$. So $(Hy)^f=Hy$. Since $\mathcal{A}=\cyc(K,G)$, we have $\lambda_Y=1$ for every $Y\in \mathcal{S}(\mathcal{A})_{G\setminus (H\cup D)}$. Therefore all the elements of $Hy$ lie in pairwise distinct basic sets. Together with $(Hy)^f=Hy$, this implies that $f^{Hy}=\id_{Hy}$.
\end{proof}

The element $x$ can be uniquely presented in the form $x=hx_0$, where $h\in H^\#$ and $x_0\in D$ is a generator of~$D$. From Claim~$3$ it follows that $x_0^f=x_0$. Observe that $D^f=D$ because $D$ is an $\mathcal{A}$-subgroup and hence $f^D\in \aut(\mathcal{A}_D)$. Lemma~\ref{2minnorm} yields that $\aut(\mathcal{A}_D)=G_r\rtimes K_D\leq \Hol(G)$. So $f^D\in K_D\leq \aut(D)$. Together with $x_0^f=x_0$, this implies that $f^D=\id_D$. Therefore $f$ fixes an element from every $H$-coset. Thus, $f$ is trivial by Claim~$3$. 

Since $\mathcal{A}=\cyc(K,G)$ and $\mathcal{A}$ is $2$-minimal, we have $\aut(\mathcal{A})=G_r\rtimes K$. This implies that $\mathcal{A}$ is normal and $|\aut(\mathcal{A})^H|=|H||K^H|\leq 12$ which completes the proof of the theorem. 
\end{proof}

\begin{theo}\label{hpasubgroups3}
Let $\mathcal{A}$ be dense. Suppose that $\mathcal{A}_D$ is the nontrivial $U/L$-wreath product, where $L$ is a minimal nontrivial $\mathcal{A}_D$-subgroup and $U$ is an $\mathcal{A}_D$-subgroup with $|\rad(\mathcal{A}_U)|=1$. Then one of the following statements holds:
\begin{enumerate}
\tm{1} $\mathcal{A}$ is the nontrivial $(H\times U)/L$-wreath product;

\tm{2} $n=3^k$, $\mathcal{A}$ is the $(H\times W)/L$-wreath product (possibly, trivial) for an $\mathcal{A}_D$-subgroup $W>U$ such that $|\rad(\mathcal{A}_W)|=3$ and $\mathcal{A}_{H\times W}$ is cyclotomic.
\end{enumerate}
\end{theo}

\begin{proof}
At first, suppose that $n=pq$. As $\mathcal{A}_D$ is the nontrivial $U/L$-wreath product, we may assume without loss of generality that $|U|=|L|=p$. Then $H$ is a maximal $\mathcal{A}$-subgroup whose order is not divisible by~$p$. So $\mathcal{A}$ is the nontrivial $(H\times U)/L$-wreath product by Lemma~\ref{nonpowernew1} applied to $H_1=H$ and $P_1=L$ and Statement~$(1)$ of the theorem holds as desired.

Now suppose that $n=p^k$. Since $\mathcal{A}_D$ is the nontrivial $U/L$-wreath product, we have $k\geq 2$.

\begin{claim4}
If there is $X\in \mathcal{S}(\mathcal{A})_{G\setminus (H\times U)}$ such that $L\nleq \rad(X)$, then $n=3^k$, $\mathcal{A}_H=\mathcal{T}_H$, $L=P$, and every such $X$ is equal to
\begin{equation}\label{forms}
h_1\{x\}\cup h_2\{xx_0\}\cup h_3\{xx_0^2\}~\text{or}~h_1\{x,x^{-1}\}\cup h_2\{xx_0,x^{-1}x_0^2\}\cup h_3\{xx_0^2,x^{-1}x_0\},
\end{equation}
where $h_1,h_2,h_3$ are pairwise distinct nonidentity elements of~$H$, $x_0$ is a generator of~$P$, and~$x\in D$. In particular, $|X|=|X_D|\leq 6$. 
\end{claim4}

\begin{proof}
We have $X\neq X_H\times X_D$ because otherwise $L\leq \rad(X_D)$ and hence $L\leq \rad(X)$ that contradicts to the assumption of the lemma. Hence $|X_H|\in \{2,3\}$ by Lemma~\ref{e4} and 
$$\lambda_X=|X\cap Hx|=1~\text{for every}~x\in X$$ 
by Lemma~\ref{hpasubgroups1}. The number $\mu_X=|X\cap Lx|$ does not depend on $x\in X$ by Lemma~\ref{intersection}. Since $\lambda_X=1$, we conclude that $\mu_X=|L|$ or $\mu_X=|L|/|X_H|$. In the former case, $L\leq \rad(X)$, a contradiction to the assumption of the claim. In the latter one, since $|X_H|\in \{2,3\}$ and $|L|$ is a $p$-power for an odd prime $p$, we obtain $|X_H|=p=3$. This implies that $n=3^k$, $\mathcal{A}_H=\mathcal{T}_H$, and $X_H=H^\#$.

Let $X_D$ be regular. Then all the $X_h$, $h\in H^\#$, are orbits of some $K\leq \aut(\langle X_D\rangle)$ by Lemma~\ref{orbit}. Since $L\nleq \rad(X)$, we obtain $P\nleq \rad(X_h)$ and hence $|\rad(X_h)|=1$ for every $h\in H^\#$. Lemma~\ref{trivradorb}(1) implies that $|K|=|X_h|\leq 2$ for every $h\in H^\#$. Therefore each $X_h$ is a singleton or each $X_h$ is of the form~$\{x,x^{-1}\}$, $x\in D$. Due to $\lambda_X=1$, all the $X_h$ are pairwise disjoint and consequently 
$$\sum \limits_{h\in H^\#} |X_h|=|X_D|\in \{3,6\}.$$
As $L\leq \rad(X_D)$, we have $L=\rad(X_D)=P$ and hence $X_D=Px$ or $X_D=Px\cup Px^{-1}$ for some $x\in D$. In the former case, each $X_h$ is a singleton and hence $X$ is of the first form from Eq.~\eqref{forms}, whereas in the latter one, each $X_h$ consists of two mutually inverse elements and hence $X$ is of the second form from Eq.~\eqref{forms} as desired.

Now let $X_D$ be nonregular. Then $X_D=V\setminus N$ for some $\mathcal{A}_D$-subgroups $V>N\geq U$ such that $|V:N|\geq 9$ by Lemma~\ref{pnonreg}. Let $m$ be the minimum of orders of the elements from~$X_D$. If $P\nleq \rad(X_{h,l})$ for some $h\in H^\#$ and $l>m$, then 
$$\varnothing\neq X^{[p]}_D\subsetneq X_D.$$ 
However, $X^{[p]}$ is an $\mathcal{A}$-set by Lemma~\ref{sch} and hence $X^{[p]}_D$ is a union of some basic sets of $\mathcal{A}_D$ by Lemma~\ref{tenspr}(1), a contradiction to the above inclusion. Therefore
\begin{equation}\label{lradp}
P\leq \rad(X_{h,l})     
\end{equation}
for all $h\in H^\#$ and $l>m$. Due to Lemma~\ref{orbit}, all $X_{h,m}$, $h\in H^\#$, are orbits of some $K\leq \aut(D_m)$. As $L\nleq \rad(X)$, Eq.~\eqref{lradp} implies that $P\nleq \rad(X_{h,m})$ and hence $|\rad(X_{h,m})|=1$ for every $h\in H^\#$. From Lemma~\ref{trivradorb}(1) it follows that $|X_{h,m}|\leq 2$ for every $h\in H^\#$. Together with $X_D=V\setminus N$, this yields that 
$$3^m-3^{m-1}=|D_m^*|=|\bigcup \limits_{h\in H^\#} X_{h,m}|\leq 6.$$
Therefore $m\leq 2$. If $m=1$, then $U$ is trivial. However, this contradicts to the assumption of the theorem that $\mathcal{A}$ is the nontrivial $U/L$-wreath product. If $m=2$, then $U=L=P$ is an $\mathcal{A}$-subgroup. By Lemma~\ref{intersection} and Eq.~\eqref{lradp}, we obtain $P=L\leq \rad(X)$, a contradiction to the assumption of the claim.
\end{proof}

If $L\leq \rad(X)$ for every $X\in \mathcal{S}(\mathcal{A})_{G\setminus (H\times U)}$, then $\mathcal{A}$ is the nontrivial $(H\times U)/L$-wreath product and Statement~$(1)$ of the theorem holds.

Suppose that $L\nleq \rad(X)$ for some $X\in \mathcal{S}(\mathcal{A})_{G\setminus (H\times U)}$. As $L$ is a minimal nontrivial $\mathcal{A}_D$-subgroup, we obtain $P\nleq \rad(X)$. So the set
$$\mathcal{W}=\{X\in \mathcal{S}(\mathcal{A})_{G\setminus (H\times U)}:~P\nleq \rad(X)\}$$
is nonempty. Put
$$W=\langle X_D:~X\in \mathcal{W} \rangle.$$ 
Observe that $X_D$ is an $\mathcal{A}_D$-set for every $X\in \mathcal{W}$ by Lemma~\ref{tenspr}(1) and hence $W$ is an $\mathcal{A}_D$-subgroup. By the definition, $W>U$. Claim~$4$ implies that $n=3^k$, $\mathcal{A}_H=\mathcal{T}_H$, $L=P$, and $|X_D|\leq 6$ for every $X\in \mathcal{W}$. Together with $\rad(X_D)\geq L$, the latter yields that $\rad(X_D)=P$ for every $X\in \mathcal{W}$ and consequently $\rad(\mathcal{A}_{W})=P$. By the definition of $W$, the radical of every basic set outside $H\times W$ contains $P$. Therefore $\mathcal{A}$ is the $(H\times W)/P$-wreath product (possibly, trivial). 

In view of the above paragraph, to prove Statement~$(2)$ of the theorem, it suffices to show that the $S$-ring $\mathcal{A}_{H\times W}$ is cyclotomic. Let $X\in \mathcal{W}$ (recall that $\mathcal{W}$ is nonempty) be such that $X_D$ contains a generator of~$W$. By Claim~$4$, the set $X$ is one of the forms from Eq.~\eqref{forms}, where $x\in W$ is a generator of~$W$. Let $\sigma\in \aut(H)$ be such that $\sigma=(h_1 h_2 h_3)$, and let $\tau_1,\tau_2\in \aut(W)$ be such that $x^{\tau_1}=xx_0$ and $x^{\tau_2}=x^{-1}$. Put 
$$K_1=\langle \varphi\rangle \leq \aut(H\times W)~\text{and}~K_2=K_1\times \langle \psi \rangle\leq \aut(H\times W),$$ 
where $\varphi,\psi\in \aut(H\times W)$ are such that $\varphi^H=\sigma$, $\varphi^D=\tau_1$, $\psi^H=\id_H$, $\psi^D=\tau_2$. Clearly, $K_1\cong C_3$ and $K_2\cong C_6$. By the definitions of $K_1$ and $K_2$, we have $X\in \orb(K,H\times W)$ and $|X|=|K|$, where $K=K_1$ if $X$ is of the first form from Eq.~\eqref{forms} and $K=K_2$ otherwise. It is easy to see that $H^\#,X_D\in \orb(K,H\times W)$. Put $X^0=X$ and $X^i=(X^{i-1})^{[p]}$ for $i\in\{1,\ldots,l-1\}$, where $l$ is such that $|W|=3^l$. Eq.~\eqref{forms} implies that $X^i=H^\#x^{p^i}$ or $X^i=H^\#x^{p^i}\cup H^\#x^{-p^i}$ for $i\in\{1,\ldots,l-1\}$. Each $X^i$ is an $\mathcal{A}$-set by Lemma~\ref{sch}. Moreover, using Lemma~\ref{intersection} and Lemma~\ref{tenspr}(1), it is easy to verify that each $X^i$ and $X^i_D$ are exactly basic sets of $\mathcal{A}_{H\times W}$ and $\mathcal{A}_W$, respectively. Due to the definition of $K$, each $X^i$ and $X^i_D$ are orbits of $K$. Every $Y\in \mathcal{S}(\mathcal{A}_{H\times W})_{G\setminus H}$ is rationally conjugate to $X^i$ or $X^i_D$ for some $i\in\{0,\ldots,l-1\}$ by Lemma~\ref{burn} and hence $Y\in \orb(K,H\times W)$. Thus, every basic set of $\mathcal{A}_{H\times W}$ is an orbit of~$K$, i.e. $\mathcal{A}_{H\times W}=\cyc(K,G)$ as required.  
\end{proof}

\begin{lemm}\label{hpasubgroups4}
Suppose that $\mathcal{A}$ is schurian, $\mathcal{A}_H=\mathcal{T}_H$, and $\mathcal{A}\neq \mathcal{A}_H\otimes \mathcal{A}_D$. Then $|\aut(\mathcal{A})^H|=12$. 
\end{lemm}

\begin{proof}
One can see that $\aut(\mathcal{A})^H\approx_2 \aut(\mathcal{A}_H)\cong \sym(4)$ because $\mathcal{A}$ is schurian. Moreover, $|\aut(\mathcal{A})^H|\in\{12,24\}$ by Remark~\ref{2minsmallrem}.

Since $\mathcal{A}\neq \mathcal{A}_H\otimes \mathcal{A}_D$, there is $X\in \mathcal{S}(\mathcal{A})_{G\setminus (U\cup D)}$ such that $X\neq X_H \times X_D$. Lemma~\ref{hpasubgroups1} implies that $\lambda_X=1$. From Lemma~\ref{tenspr}(1) it follows that $X_D\in \mathcal{S}(\mathcal{A}_D)$ and $X_H=H^\#$. Let $L=\rad(X)$, $U=\langle X \rangle$, and $W=\langle X_D \rangle$. Since $D$ is an $\mathcal{A}$-subgroup, we have $L\leq D$. The image of $T\subseteq G$ under the canonical epimorphism from $G$ to $G/L$ is denoted by $\overbar{T}$. The argument from the previous paragraph applied to $\mathcal{A}_{\overbar{U}}$ implies that $|\aut(\mathcal{A}_{\overbar{U}})^{\overbar{H}}|\in \{12,24\}$. To prove the lemma, it is enough to show that
$$|\aut(\mathcal{A}_{\overbar{U}})^{\overbar{H}}|=12.$$
Indeed, then 
$$|\aut(\mathcal{A})^H|=|\aut(\mathcal{A}_{\overbar{G}})^{\overbar{H}}|\leq |\aut(\mathcal{A}_{\overbar{U}})^{\overbar{H}}|=12$$
as required, where the first equality holds because $|L\cap H|=1$ and hence the mapping $f\mapsto \overbar{f}$ from $\aut(\mathcal{A})^H$ to $\aut(\mathcal{A}_{\overbar{G}})^{\overbar{H}}$, where $\overbar{f}$ is the permutation induced by $f$ on $\overbar{G}$, is an isomorphism.

By the definition of $L$, we have
\begin{equation}\label{trivradxpi}
|\rad(\overbar{X})|=1. 
\end{equation}
As $\lambda_X=1$ and $L\leq D$, we obtain $\lambda_{\overbar{X}}=|\overbar{X}\cap \overbar{H} \overbar{x}|=1$ for every $\overbar{x}\in \overbar{X}$. So $\mathcal{A}_{\overbar{U}}\neq \mathcal{A}_{\overbar{H}}\otimes \mathcal{A}_{\overbar{W}}$. The set $\overbar{X}_D$ is a basic set of $\mathcal{A}_{\overbar{D}}$. If $|\rad(\overbar{X}_D)|=1$, then $|\rad(\mathcal{A}_{\overbar{W}})|=1$ and Theorem~\ref{hpasubgroups2}(2) applied to $\mathcal{A}_{\overbar{U}}$ yields that $|\aut(\mathcal{A}_{\overbar{U}})^{\overbar{H}}|=12$ as desired. 

Suppose that $|\rad(\overbar{X}_D)|>1$. In view of Lemma~\ref{pq} if $|W|=pq$ and Lemma~\ref{cyclpwreath} if $|W|$ is a $p$-power, Theorem~\ref{hpasubgroups3} holds for $\mathcal{A}_{\overbar{U}}$. If Statement~$(1)$ of Theorem~\ref{hpasubgroups3} holds for $\mathcal{A}_{\overbar{U}}$, then  $|\rad(\overbar{X})|>1$, a contradiction to Eq.~\eqref{trivradxpi}. Therefore Statement~$(2)$ of Theorem~\ref{hpasubgroups3} holds for $\mathcal{A}_{\overbar{U}}$. Due to Claim~$4$ and Eq.~\eqref{trivradxpi}, the basic sets of $\mathcal{A}_{\overbar{U}}$ inside $\overbar{X}_H\times \overbar{X}_D$ are of the form 
$$\overbar{X}=\overbar{h}_1\overbar{X}_1\cup \overbar{h}_2\overbar{X}_2\cup \overbar{h}_3\overbar{X}_3,~\overbar{Y}=\overbar{h}_1\overbar{X}_2\cup \overbar{h}_2\overbar{X}_3\cup \overbar{h}_3\overbar{X}_1,~\overbar{Z}=\overbar{h}_1\overbar{X}_3\cup \overbar{h}_2\overbar{X}_1\cup \overbar{h}_3\overbar{X}_2,$$
where the $\overbar{h}_i$ are pairwise distinct elements of $\overbar{X}_H$ and the $\overbar{X}_i$ are pairwise disjoint subsets of $\overbar{X}_D$ of size at most~$2$. If $|\aut(\mathcal{A}_{\overbar{U}})^{\overbar{H}}|=24$, then there is $f\in \aut(\mathcal{A}_{\overbar{U}})$ such that $\overbar{h}_1^f=\overbar{h}_1$, $\overbar{h}_2^f=\overbar{h}_3$, and $\overbar{h}_3^f=\overbar{h}_2$. Eq.~\eqref{aut} implies that
$$\overbar{X}_1^f=(\overbar{X} \overbar{h}_1\cap \overbar{D})^f=\overbar{X} \overbar{h}_1^f\cap \overbar{D}=\overbar{X}_1$$
and 
$$\overbar{X}_1^f=(\overbar{Y} \overbar{h}_3\cap \overbar{D})^f=\overbar{Y} \overbar{h}_3^f \cap \overbar{D}=\overbar{Y} \overbar{h}_2 \cap \overbar{D}=\overbar{X}_3,$$
a contradiction. Thus, $|\aut(\mathcal{A}_{\overbar{U}})^{\overbar{H}}|=12$ and we are done.
\end{proof}

\subsection{Proof of Theorem~\ref{e4cn}}
Suppose that $H$ or $D$ is not an $\mathcal{A}$-subgroup. Then Theorem~\ref{e4cn} holds for $\mathcal{A}$ or $\mathcal{A}$ is a nontrivial $S$-wreath product for an $\mathcal{A}$-section $S=U/L$ such that $4 \nmid |W|$ and $|\rad(\mathcal{A}_W)|=1$, where $W=U$ if $H$ is not an $\mathcal{A}$-subgroup and $W=G/L$ if $D$ is not an $\mathcal{A}$-subgroup: indeed, this follows from Theorem~\ref{hnot} in the former case and from Theorem~\ref{dnot} in the latter one.

Note that $|W_H|\leq 2$, where $W_H$ is the Sylow $2$-subgroup of $W$, because $|W|$ is not divisible by~$4$ and consequently $\mathcal{A}_{W_H}=\mathbb{Z}W_H$. So Lemma~\ref{tenspr}(2) yields that 
\begin{equation}\label{tensw1w0}
\mathcal{A}_W=\mathcal{A}_{W_H}\otimes \mathcal{A}_{W_D},
\end{equation}
where $W_D$ is the Hall $2^\prime$-subgroup of $W$ (here $W_H$ and $W_D$ can be trivial). Since $|\rad(\mathcal{A}_W)|=1$, we conclude that $|\rad(\mathcal{A}_{W_D})|=1$. 

If $\mathcal{A}_{W_D}=\mathcal{T}_{W_D}$, then due to Eq.~\eqref{tensw1w0}, we obtain $L\in \{W_H,W_D,W\}$ whenever $W=U$ and $S\in \{W_H,W_D,\{L\}\}$ whenever $W=G/L$. This yields that Statement~$(1)$ or~$(2)$ from Theorem~\ref{e4cn} holds. If $\mathcal{A}_{W_D}\neq\mathcal{T}_{W_D}$, then $\mathcal{A}_{W_D}=\cyc(K_D,W_D)$ for some $K_D\leq \aut(W_D)$ by Lemma~\ref{leungman}. As $|W_H|\leq 2$, we have $\mathcal{A}_{W_H}=\cyc(K_H,W_H)$, where $K_H$ is trivial. Therefore $\mathcal{A}_W=\cyc(K_H\times K_D,W)$ by Eqs.~\eqref{cycltens} and~\eqref{tensw1w0} and Statement~$(3)$ from Theorem~\ref{e4cn} holds.

Now suppose that $H$ and $D$ are $\mathcal{A}$-subgroups. If $|\rad(\mathcal{A}_D)|=1$, then $\mathcal{A}=\mathcal{A}_H\otimes \mathcal{A}_D$ or $\mathcal{A}$ is cyclotomic by Theorem~\ref{hpasubgroups2}(1) and we are done. If $|\rad(\mathcal{A}_D)|>1$, then $\mathcal{A}_D$ is a nontrivial $U/L$-wreath product for some $\mathcal{A}_D$-section $U/L$ such that $L$ is a minimal nontrivial $\mathcal{A}_D$-subgroup and $|\rad(\mathcal{A}_U)|=1$. Indeed, this follows from Lemma~\ref{pq} if $n=pq$ and from Lemma~\ref{cyclpwreath} if $n=p^k$. So Theorem~\ref{hpasubgroups3} holds for $\mathcal{A}$. If Statement~$(2)$ of Theorem~\ref{hpasubgroups3} holds for $\mathcal{A}$, then $n=3^k$, $\mathcal{A}$ is the $S=(H\times W)/L$-wreath product (possibly, trivial) for an $\mathcal{A}_D$-subgroup $W>U$ such that $|\rad(\mathcal{A}_W)|=3$, and $\mathcal{A}_{H\times W}$ is cyclotomic. In this case, $H\times W=G$ and $\mathcal{A}$ is cyclotomic or $H\times W<G$, $\mathcal{A}$ is the nontrivial $S$-wreath product. Thus, Statement~$(4)$ from Theorem~\ref{e4cn} holds.

Suppose that Statement~$(1)$ of Theorem~\ref{hpasubgroups3} holds for $\mathcal{A}$, i.e. $\mathcal{A}$ is the nontrivial $S=V/L$-wreath product, where $V=H\times U$. If $\mathcal{A}_{V}\neq \mathcal{A}_H\otimes \mathcal{A}_U$, then $\mathcal{A}_{V}$ is cyclotomic by Theorem~\ref{hpasubgroups2}(1). So Statement~$(4)$ from Theorem~\ref{e4cn} holds for $\mathcal{A}$ and $S$. Further, we assume that
\begin{equation}\label{tenshv}
\mathcal{A}_{V}=\mathcal{A}_{H}\otimes \mathcal{A}_{U}.
\end{equation} 
Then $\mathcal{A}_H=\cyc(K_H,H)$ for some $K_H\leq \aut(H)$ by Lemma~\ref{e4}. If $\mathcal{A}_U=\cyc(K_U,U)$ for some $K_U\leq \aut(U)$, then $\mathcal{A}_{V}=\cyc(K_H\times K_U,V)$ by Eqs.~\eqref{cycltens} and~\eqref{tenshv}. Therefore Statement~$(4)$ from Theorem~\ref{e4cn} holds for $\mathcal{A}$ and $S$. Otherwise, $\mathcal{A}_U=\mathcal{T}_U$ by Lemma~\ref{leungman}. Then $L=U$. Together with Eq.~\eqref{tenshv}, this implies that $\mathcal{A}_L$ is $\otimes$-complemented in $\mathcal{A}_V$. Thus, Statement~$(2)$ from Theorem~\ref{e4cn} holds for $S$ and we are done.

\subsection{Auxiliary statements}

\begin{lemm}\label{trivrade4cn}
Let $\mathcal{B}$ be an $S$-ring over a subgroup $V$ of $E_4\times C_{p^k}$, where $p\geq 5$ and $k\geq 2$. Suppose that $p^2~|~|V|$, $\mathcal{B}$ is indecomposable, and there is a $\mathcal{B}$-subgroup of order or index~$p$. Then $\mathcal{B}$ is normal.
\end{lemm}

\begin{proof}
Clearly, $\mathcal{B}$ is nontrivial. Firstly, suppose that $V$ is cyclic. If $|V|$ is a $p$-power, then $\mathcal{B}$ is normal by the supposition of this lemma and Lemma~\ref{leungman}. If $|V|=2p^l$ for some $l\geq 2$ and $\mathcal{B}$ is not normal, then $\mathcal{B}=\mathcal{T}_{V_1}\otimes \mathbb{Z}V_2$, where $V_1,V_2\leq V$ with $|V_1|=p^l$ and $|V_2|=2$, by Lemma~\ref{2primepower}. However, the latter equality contradicts to the supposition of the lemma that  there is a $\mathcal{B}$-subgroup of order or index~$p$. 

Now suppose that $V$ is noncyclic, i.e. $4~|~|V|$. Let 
$$\mathcal{B}=\mathcal{B}_{V_1}\otimes \mathcal{B}_{V_2}$$ 
for some proper nontrivial $\mathcal{B}$-subgroups $V_1$ and $V_2$ such that $V=V_1\times V_2$. Without loss of generality, we may assume that $|V_1|=p^l$ and $|V_2|=4$, or $|V_1|=2p^l$ and $|V_2|=2$ for some $l\geq 2$. In particular, $V_1$ is cyclic. Note that $|\rad(\mathcal{B}_{V_1})|=1$ because otherwise $\mathcal{B}_{V_1}$ is decomposable by Lemma~\ref{leungman} and hence so is $\mathcal{B}$, which contradicts to the supposition of the lemma. If $|V_1|=p^l$ ($|V_2|=2p^l$, respectively), then the last condition of this lemma and Lemma~\ref{leungman} (Lemma~\ref{2primepower}, respectively) imply that $\mathcal{B}_{V_1}$ is normal. Then $\mathcal{B}_{V_2}$ is also normal. Indeed, this is clear if $|V_2|=2$ and follows from Lemma~\ref{e4} if $|V_2|=4$. Therefore $\mathcal{B}$ is normal by Eq.~\eqref{auttens}.  

In view of the above paragraph, we may assume that $\mathcal{B}$ is not a nontrivial tensor product. Since $\mathcal{B}$ is also nontrivial and indecomposable, Theorems~\ref{hnot} and~\ref{dnot} can not hold for $\mathcal{B}$. So $\mathcal{B}$ is dense. In particular, the Sylow $p$-subgroup $V_1$ of $S$ is a $\mathcal{B}$-subgroup. If $|\rad(\mathcal{B}_{V_1})|>1$, then $\mathcal{B}$ is decomposable or $p=3$ by Theorem~\ref{hpasubgroups3}, a contradiction to the supposition of the lemma. Therefore $|\rad(\mathcal{B}_{V_1})|=1$. Thus, $\mathcal{B}$ is normal by Theorem~\ref{hpasubgroups2}(1). 
\end{proof}

Recall that $\mathcal{A}$ is an $S$-ring over $G=H\times D$, where $H\cong E_4$ and $D\cong C_n$ for $n\in\{pq,p^k\}$.

\begin{lemm}\label{autrestprime}
Let $n=p^k$. The following statements hold.

\begin{enumerate}
\tm{1} If $P\cong C_p\leq G$ is an $\mathcal{A}$-subgroup, then $\aut(\mathcal{A})^P\leq \Hol(P)$ or $\mathcal{A}$ is the $U/P$-wreath product (possibly, trivial) for an $\mathcal{A}$-subgroup $U$ such that $|U/P|\leq 4$.

\tm{2} If $U\cong E_4\times C_{p^{k-1}}\leq G$ is an $\mathcal{A}$-subgroup, then $\aut(\mathcal{A})^{G/U}\leq \Hol(G/U)$ or $\mathcal{A}$ is the $U/L$-wreath product (possibly, trivial) for an $\mathcal{A}$-subgroup $L$ such that $|U/L|\leq 4$.
\end{enumerate}
\end{lemm}

\begin{proof}
If $p=3$, then $\Hol(C_3)\cong \sym(3)$ and the lemma is clear. Let $p\geq 5$. Let us prove Statement~$(1)$. If $P\leq \rad(X)$ for every $X\in \mathcal{S}(\mathcal{A})_{G\setminus (H\times P)}$, then $\mathcal{A}$ is the $U/P$-wreath product for an $\mathcal{A}$-subgroup $U\leq H\times P$ and we are done. Further, we assume that there is $X\in \mathcal{S}(\mathcal{A})_{G\setminus (H\times P)}$ such that
$$\rad(X)\ngeq P.$$

Let $R=\rad(X)$ and $V=\langle X \rangle$. The image of $Y\subseteq G$ under the canonical epimorphism from $G$ to $G/R$ is denoted by $\overbar{Y}$. As $R\ngeq P$, we have $R\leq H$ and $|R\cap P|=1$ (note that $R$ can be trivial). The latter equality implies that the kernel of the action of $\aut(\mathcal{A})^P$ on $\overbar{P}$ is trivial and hence it suffices to show that
$$\aut(\mathcal{A})^{\overbar{P}}\leq \Hol(\overbar{P}).$$

Since $\overbar{X}$ generates $\overbar{V}$ and $|\rad(\overbar{X})|=1$, the $S$-ring $\mathcal{A}_{\overbar{V}}$ is indecomposable. Clearly, $\overbar{P}$ is an $\mathcal{A}_{\overbar{V}}$-subgroup of order~$p$. Moreover, $p^2~|~|\overbar{V}|$ because $X\nsubseteq H\times P$ and $R\leq H$. So $\overbar{V}$ satisfies all the conditions of Lemma~\ref{trivrade4cn} and consequently $\mathcal{A}_{\overbar{V}}$ is normal. Therefore
$$\aut(\mathcal{A})^{\overbar{P}}\leq \aut(\mathcal{A}_{\overbar{V}})^{\overbar{P}}\leq (\Hol(\overbar{V}))^{\overbar{P}}\leq \Hol(\overbar{P}),$$
where the second inequality follows from the normality of $\mathcal{A}_{\overbar{V}}$, and we are done.

Now let us prove Statement~$(2)$. If $U\cap D\leq \rad(X)$ for every $X\in \mathcal{S}(\mathcal{A})_{G\setminus U}$, then  $\mathcal{A}$ is the $U/L$-wreath product for an $\mathcal{A}$-subgroup $L\geq (U\cap D)$ and we are done. Further, we assume that there is $X\in \mathcal{S}(\mathcal{A})_{G\setminus U}$ such that
$$\rad(X)\ngeq (U\cap D).$$

Let $V=\langle X \rangle$. Clearly, $|G:V|\in\{1,2,4\}$ and $|V:(V\cap U)|=|G:U|=p$. Observe that $\aut(\mathcal{A})^{G/U}$ and $\aut(\mathcal{A})^{V/(V\cap U)}$ are permutation isomorphic in the sense of~\cite[p.~17]{DM}. Indeed, the kernel of the action of $\aut(\mathcal{A})$ on $G/U$ contains $H_r$ and hence for every $f\in \aut(\mathcal{A})$, there is $f^\prime\in \aut(\mathcal{A})$ such that $f^{G/U}=(f^\prime)^{G/U}$ and $V^{f^\prime}=V$. As $V\cap U\leq U$, for every $\Delta\in G/U$, there is $\Delta^\prime \in V/(V\cap U)$ such that $\Delta^\prime\subseteq\Delta$. Moreover, such $\Delta^\prime$ is unique because $|V:(V\cap U)|=|G:U|=p$. Therefore the mappings $f^{G/U}\mapsto (f^\prime)^{V/(V\cap U)}$ from $\aut(\mathcal{A})^{G/U}$ to $\aut(\mathcal{A})^{V/(V\cap U)}$ and $\Delta\mapsto \Delta^\prime$ from $G/U$ to $V/(V\cap U)$ define a permutation isomorphism from $\aut(\mathcal{A})^{G/U}$ to $\aut(\mathcal{A})^{V/(V\cap U)}$. Thus, it is enough to show that
$$\aut(\mathcal{A})^{V/(V\cap U)}\leq \Hol(V/(V\cap U)).$$

Let $R=\rad(X)$. Again, the image of $Y\subseteq G$ under the canonical epimorphism from $G$ to $G/R$ is denoted by $\overbar{Y}$. Clearly, the $S$-ring $\mathcal{A}_{\overbar{V}}$ is indecomposable and $\overbar{V\cap U}$ is an $\mathcal{A}_{\overbar{V}}$-subgroup of index~$p$. Since $R\ngeq (U\cap D)$, we conclude that $p^2~|~|\overbar{V}|$. So $\overbar{V}$ satisfies all the conditions of Lemma~\ref{trivrade4cn} and hence $\mathcal{A}_{\overbar{V}}$ is normal. Therefore 
$$\aut(\mathcal{A})^{V/(V\cap U)}\leq \aut(\mathcal{A}^{\overbar{V}})^{V/(V\cap U)}\leq \aut(\mathcal{A}_{\overbar{V}})^{V/(V\cap U)}\leq$$ 
$$\leq (\Hol(\overbar{V}))^{V/(V\cap U)}\leq \Hol(V/(V\cap U)),$$
where the third inequality follows from the normality of $\mathcal{A}_{\overbar{V}}$, and we are done.
\end{proof}

It should be mentioned that if $\mathcal{A}_P$ ($\mathcal{A}_{G/U}$, respectively) in Statement~$(1)$ (Statement~$(2)$, respectively) of Lemma~\ref{autrestprime} is nontrivial, then $\mathcal{A}_P$ ($\mathcal{A}_{G/U}$, respectively) is normal by Lemma~\ref{leungman} because $P$ ($G/U$, respectively) is of prime order~$p$. Therefore the claim of Lemma~\ref{autrestprime} easily follows from the definition of a normal $S$-ring. We also note that if the generalized wreath product in Statement~$(1)$ or~$(2)$ of Lemma~\ref{autrestprime} is trivial, then $|G|=4p$.

\section{Proof of Theorem~\ref{main1}}

Throughout this section, we keep the notation from Section~7. We start with a lemma concerned with a schurity of proper subgroups of $G$.

\begin{lemm}\label{sections2}
Every proper subgroup of $G$ is Schur. 
\end{lemm}

\begin{proof}
Let $N<G$. If $N$ is cyclic, then $N$ is Schur by~\cite[Theorem~1.1]{EKP1}, whereas if $N\cong E_4$, then $N$ is Schur by~\cite[Theorem~1.2]{EKP2}. Otherwise, $N\cong E_4 \times C_{r^m}$, where $r\in \{p,q\}$ and $k>m\geq 1$, is Schur by induction on $m$ whose base follows from~\cite[Theorem~1.5]{EKP2}.
\end{proof}

Let $\mathcal{A}$ be a nontrivial $S$-ring over $G$. Let us prove that $\mathcal{A}$ is schurian. Theorem~\ref{e4cn} holds for $\mathcal{A}$. We are done if $\mathcal{A}$ is cyclotomic. If $\mathcal{A}$ is a nontrivial tensor product, then $\mathcal{A}$ is schurian by Lemma~\ref{schurtens} and Lemma~\ref{sections2}. So we may assume that $\mathcal{A}$ is a nontrivial $S$-wreath product for some $\mathcal{A}$-section $S=U/L$ and one of Statements~$(1)$-$(4)$ from Theorem~\ref{e4cn} holds. The $S$-rings $\mathcal{A}_U$ and $\mathcal{A}_{G/L}$ are schurian by Lemma~\ref{sections2}. If Statement~$(1)$ holds, i.e. $|S|\leq 2$, then $\mathcal{A}_S$ is $2$-minimal by Lemma~\ref{2minsmall} and hence $\mathcal{A}$ is schurian by Lemma~\ref{2min}. If Statement~$(2)$ holds, i.e. $\mathcal{A}_{L}$ is $\otimes$-complemented in $\mathcal{A}_{U}$ or $\mathcal{A}_{S}$ is $\otimes$-complemented in $\mathcal{A}_{G/L}$, then $\mathcal{A}$ is schurian by Lemma~\ref{otimescomplement}.

Now suppose that Statement~$(3)$ or~$(4)$ holds. Then one can choose $W\in \{U,G/L\}$ such that $\mathcal{A}_W$ is cyclotomic. This yields that $\mathcal{A}_S$ is also cyclotomic. Let $S_H$ and $S_D$ be the Sylow $2$-subgroup and the Hall $2^\prime$-subgroup of $S$, respectively. Note that $S_H$ and $S_D$ are $\mathcal{A}_S$-subgroups because $\mathcal{A}_S$ is cyclotomic. If $n=pq$ and $|S_D|=pq$, then $S=S_D$, $|U|=|G/L|=2pq$, and $|L|=|G/U|=2$. So $|W_H|=2$, where $W_H$ is the Sylow $2$-subgroup of $W$. Since $W$ is cyclotomic, $W_H$ is an $\mathcal{A}_W$-subgroup. Clearly, $\mathcal{A}_{W_H}=\mathbb{Z}W_H$. Together with Lemma~\ref{tenspr}(2), this implies that $\mathcal{A}_W\cong \mathbb{Z}W_H\otimes \mathcal{A}_{S}$. Therefore $\mathcal{A}_{L}$ is $\otimes$-complemented in $\mathcal{A}_{U}$ or $\mathcal{A}_{S}$ is $\otimes$-complemented in $\mathcal{A}_{G/L}$ and hence $\mathcal{A}$ is schurian by Lemma~\ref{otimescomplement}. 

In view of the above paragraph, we may assume further that $|S_D|$ is a $p$-power. The Hall $2^\prime$-subgroup of $W$ is denoted by~$W_D$. Note that $W_D$ is an $\mathcal{A}_W$-subgroup because $\mathcal{A}_W$ is cyclotomic.

\begin{lemm}\label{s0structure}
With the above notation, $\mathcal{A}_{S_D}$ is a cyclotomic $S$-ring with trivial radical.
\end{lemm}

\begin{proof}
One can see that $\mathcal{A}_{S_D}$ is cyclotomic because so is $\mathcal{A}_W$. Suppose that Statement~$(3)$ of Theorem~\ref{e4cn} holds. Then $|\rad(\mathcal{A}_W)|=1$ and $|W_H|\leq 2$. Observe that $|\rad(\mathcal{A}_{W_D})|=1$. Indeed, Lemma~\ref{tenspr}(2) and Lemma~\ref{leungman} imply that $\mathcal{A}_W=\mathbb{Z}W_H\otimes \mathcal{A}_{W_D}$ and hence $|\rad(\mathcal{A}_{W_D})|=|\rad(\mathcal{A}_W)|=1$. If $S_D=W_D$, then obviously $|\rad(\mathcal{A}_{S_D})|=1$. Otherwise, $|\rad(\mathcal{A}_{S_D})|=1$ by Lemma~\ref{cyclradp}.

Now suppose that Statement~$(4)$ of Theorem~\ref{e4cn} holds. Then $W=U$, $\{e\}<L\leq D$, and $|\rad(\mathcal{A}_{W_D})|\in\{1,3\}$. As $\{e\}<L\leq D$, we conclude that $S_D\neq W_D$. Therefore $|\rad(\mathcal{A}_{S_D})|=1$ by Lemma~\ref{cyclradp}.
\end{proof}

\begin{lemm}\label{s0schur}
With the above notation, $\mathcal{A}_S$ is $2$-minimal unless $\mathcal{A}_S=\mathcal{A}_{S_H}\otimes \mathcal{A}_{S_D}$ and at least one of the following statements holds:
\begin{enumerate}

\tm{1} $|S_H|=4$ and $\mathcal{A}_{S_H}=\mathcal{T}_{S_H}$;

\tm{2} $|S_D|=p$ and $\mathcal{A}_{S_D}=\mathcal{T}_{S_D}$.

\end{enumerate}
\end{lemm}

\begin{proof}
Assume that $\mathcal{A}_S$ is not $2$-minimal. Clearly, $|S_H|\in\{1,2,4\}$. Let us prove that
\begin{equation}\label{tens111}
\mathcal{A}_S=\mathcal{A}_{S_H}\otimes \mathcal{A}_{S_D}
\end{equation}
If $|S_H|\leq 2$, then $\mathcal{A}_{S_H}=\mathbb{Z}S_H$ and Eq.~\eqref{tens111} follows from Lemma~\ref{tenspr}(2). If $|S_H|=4$ and Eq.~\eqref{tens111} does not hold, then $\mathcal{A}_S$ is $2$-minimal by Theorem~\ref{hpasubgroups2}(1) and Lemma~\ref{s0structure}, a contradiction to the assumption.

As $\mathcal{A}_S$ is not $2$-minimal, Corollary~\ref{2mintens} yields that at least one of the $S$-rings $\mathcal{A}_{S_H}$, $\mathcal{A}_{S_D}$ is not $2$-minimal. If $\mathcal{A}_{S_H}$ is not $2$-minimal, then Statement~$(1)$ of the lemma holds by Lemma~\ref{2minsmall}. If $\mathcal{A}_{S_D}$ is not $2$-minimal, then $\mathcal{A}_{S_D}$ is not normal by Lemma~\ref{2minnorm}. Note that $|\rad(\mathcal{A}_{S_D})|=1$ by Lemma~\ref{s0structure} and hence $\mathcal{A}_{S_D}=\mathcal{T}_{S_D}$ by Lemma~\ref{leungman}. Since $\mathcal{A}_{S_D}$ is cyclotomic, we conclude that $|S_D|=p$. Thus, Statement~$(2)$ of the lemma holds.
\end{proof}

It is worth to mention that the tensor product from Lemma~\ref{s0schur} can be trivial, i.e. $S_H$ or $S_D$ can be trivial. In view of Lemma~\ref{2min} and Lemma~\ref{s0schur}, we may assume further that 
$$\mathcal{A}_S=\mathcal{A}_{S_H}\otimes \mathcal{A}_{S_D}$$ 
and Statement~$(1)$ or~$(2)$ of Lemma~\ref{s0schur} holds. The canonical epimorphism from $G$ to $G/L$ is denoted by $\pi$ and the image of $X\subseteq G$ under~$\pi$ is denoted by $\overbar{X}$.

\begin{lemm}\label{state3}
If Statement~$(3)$ of Theorem~\ref{e4cn} holds, then $\mathcal{A}$ is schurian.
\end{lemm}

\begin{proof}
By the condition of the lemma, $\mathcal{A}_W$ is a circulant cyclotomic $S$-ring with trivial radical. Clearly, $|W_H|\leq 2$ in this case. So $|S_H|\leq 2$. This implies that Statement~$(2)$ of Lemma~\ref{s0schur} holds for $\mathcal{A}_S$, i.e. $|S_D|=p$ and $\mathcal{A}_{S_D}=\mathcal{T}_{S_D}$. Let $V\in \{U,G/L\}\setminus \{W\}$. If $4 \nmid |V|$, then $V$ is cyclic and $2~|~|L|$. The latter yields that $2 \nmid |S|$ and hence $|S|=|S_D|=p$. Therefore $\mathcal{A}$ is schurian by Lemma~\ref{prime}.

Now suppose that $4~|~|V|$. Note that 
$$|V|=4p^l$$ 
for some $l\geq 1$. This is obvious if $n=p^k$ and follows from $|S_D|=p$ and $V<G$ if $n=pq$. One can see that 
$$|S_H|=2$$
because $4\nmid |W|$ and $4~|~|V|$. Clearly, $\mathcal{A}_{S_H}=\cyc(\aut(S_H),S_H)$ and $\mathcal{A}_{S_D}=\mathcal{T}_{S_D}=\cyc(\aut(S_D),S_D)$. So 
$$\mathcal{A}_S=\mathcal{A}_{S_H}\otimes \mathcal{A}_{S_D}=\cyc(\aut(S_H)\times\aut(S_D),S)=\cyc(\aut(S),S)$$
by Eq.~\eqref{cycltens}. The latter is equivalent to 
$$\mathcal{A}_S=V(\Hol(S),S).$$

As $\mathcal{A}_W$ is cyclotomic, there exists $K_1\leq \Hol(W)$ such that $\mathcal{A}_W=V(K_1,W)$. It is easy to see that $K_1^S\leq \Hol(S)$ and $\mathcal{A}_S=V(K_1^S,S)$. Since $\mathcal{A}_S$ is Cayley minimal (Lemma~\ref{cyclcayleymin}) and $\mathcal{A}_S=V(\Hol(S),S)$, we conclude that 
\begin{equation}\label{k11}
K_1^S=\Hol(S).
\end{equation}

Let $\aut(\mathcal{A}_{V})^{S_D}\leq \Hol(S_D)$. Then $\aut(\mathcal{A}_{V})^{S_D}=\Hol(S_D)$ because $\Hol(S_D)\in \mathcal{K}^{\min}(\mathcal{A}_{S_D})$ and $\mathcal{A}_{S_D}=V(\aut(\mathcal{A}_{V})^{S_D},S_D)$. Note that $\aut(\mathcal{A}_{V})^{S_H}=\Hol(S_H)$ because $|S_H|=2$. Therefore
$$\aut(\mathcal{A}_{V})^S=\aut(\mathcal{A}_{V})^{S_H}\times \aut(\mathcal{A}_{V})^{S_D}=\Hol(S_H)\times \Hol(S_D)=\Hol(S).$$
Together with Eq.~\eqref{k11}, this yields that $K_1^S=\aut(\mathcal{A}_{V})^S$. Thus, $\mathcal{A}$ is schurian by Lemma~\ref{schurwr}.

In view of the above paragraph, we may assume that 
$$\aut(\mathcal{A}_{V})^{S_D}\nleq \Hol(S_D).$$
If $V=G/L$, then $\mathcal{A}_{V}$ is the $\overbar{U_1}/S_D$-wreath product for an $\mathcal{A}_{G/L}$-subgroup $\overbar{U_1}\geq S$ with $|\overbar{U_1}/S_D|\leq 4$ by Lemma~\ref{autrestprime}(1) applied to $\mathcal{A}_{G/L}$, whereas if $V=U$, then $\mathcal{A}_{V}$ is the $U_2/L_2$-wreath product, where $U_2=S_H^{\pi^{-1}}$ and $L_2\leq L$ is an $\mathcal{A}_{U}$-subgroup with $|S_H^{\pi^{-1}}/L_2|\leq 4$, by Lemma~\ref{autrestprime}(2) applied to $\mathcal{A}_{U}$.

Suppose that $|V|\neq 4p$. Then $\mathcal{A}$ is the nontrivial $S_1=U_1/L_1$-wreath product, where $U_1=\overbar{U_1}^{\pi^{-1}}$ and $L_1=(S_D)^{\pi^{-1}}$, whenever $V=G/L$, and $\mathcal{A}$ is the nontrivial $S_2=U_2/L_2$-wreath product whenever $V=U$. As $|\overbar{U_1}/S_D|\leq 4$ ($|S_H^{\pi^{-1}}/L_2|\leq 4$, respectively), we obtain $|S_1|\leq 4$ ($|S_2|\leq 4$, respectively). If $|S_i|=4$, $i\in\{1,2\}$, then $\mathcal{A}_{S_i}\neq \mathcal{T}_{S_i}$ because $S_D$ is an $\mathcal{A}_{G/L}$-subgroup of order~$2$. Therefore $\mathcal{A}_{S_i}$, $i\in\{1,2\}$, is $2$-minimal by Lemma~\ref{2minsmall}. Thus, $\mathcal{A}$ is schurian by Lemma~\ref{2min}.

Now suppose that $|V|=4p$. If $\mathcal{A}_V$ is dense, then $\mathcal{A}$ is cyclotomic by~\cite[Case~$3$,~p.~115]{EKP2}. Therefore $\mathcal{A}_{V}=V(K_0,V)$ for some $K_0\leq \Hol(V)$. Clearly, $K_0^S\leq \Hol(S)$ and $\mathcal{A}_S=V(K_0^S,S)$. Since $\mathcal{A}_S$ is Cayley minimal (Lemma~\ref{cyclcayleymin}) and $\mathcal{A}_S=V(\Hol(S),S)$, we conclude that $K_0^S=\Hol(S)$. Together with Eq.~\eqref{k11} and Lemma~\ref{schurwr}, this implies that $\mathcal{A}$ is schurian.

Let $\mathcal{A}$ be nondense. If $V=G/L$, then $\mathcal{A}_{V}=\mathcal{A}_{G/L}$ is the $S/S_D$-wreath product by~\cite[Case~$1$,~p.~114]{EKP2} and hence $\mathcal{A}$ is the $S_1=U/L_1$-wreath product, where $L_1=(S_D)^{\pi^{-1}}$. If $V=U$, then $\mathcal{A}_{V}=\mathcal{A}_{U}$ is the $S_2=H/L$-wreath product by~\cite[Case~$2$,~p.~114]{EKP2} (in this case, $|L|=2$) and hence $\mathcal{A}$ is the $S_2$-wreath product. One can see that $|S_1|=|U/L_1|=|S/S_D|=|S_H|=2$ and $|S_2|=|H/L|=2$. Thus, $\mathcal{A}$ is schurian by Lemma~\ref{2min}. 
\end{proof}

In view of Lemma~\ref{state3}, we may assume that Statement~$(4)$ of Theorem~\ref{e4cn} holds, i.e. $\mathcal{A}$ is dense, $U\geq H$, $\mathcal{A}_U$ is cyclotomic, $|\rad(\mathcal{A}_{U\cap D})|=1$ unless $n=3^k$ and $|\rad(\mathcal{A}_{U\cap D})|=3$, and $L\leq D$. In this case, $|S_H|=4$.

\begin{lemm}\label{hrest}
There exist $K_1\leq \aut(\mathcal{A}_U)$ and $K_0\leq \aut(\mathcal{A}_{G/L})$ such that $K_1\approx_2 \aut(\mathcal{A}_U)$, $K_0\approx_2 \aut(\mathcal{A}_{G/L})$, and $K_1^{S_H}=K_0^{S_H}\in \mathcal{K}^{\min}(\mathcal{A}_{S_H})$.
\end{lemm}

\begin{proof}
If $\mathcal{A}_{S_H}\neq \mathcal{T}_{S_H}$, then $\mathcal{A}_{S_H}$ is $2$-minimal by Lemma~\ref{2minsmall}. So $\aut(\mathcal{A}_U)^{S_H}=\aut(\mathcal{A}_{G/L})^{S_H}=\aut(\mathcal{A}_{S_H})$ and the lemma holds for $K_1=\aut(\mathcal{A}_U)$ and $K_0=\aut(\mathcal{A}_{G/L})$.

Further, we assume that $|S_H|=4$ and $\mathcal{A}_{S_H}=\mathcal{T}_{S_H}$. By Remark~\ref{2minsmallrem}, there is a unique proper $2$-equivalent subgroup $\alt(S_H)\cong \alt(4)$ of $\aut(\mathcal{A}_{S_H})\cong \sym(4)$ of order~$12$ such that $\alt(S_H)\in \mathcal{K}^{\min}(\mathcal{A}_{S_H})$. If $\mathcal{A}_U\neq \mathcal{A}_H\otimes \mathcal{A}_{U\cap D}$, then put $K_1=\aut(\mathcal{A}_U)$. In this case, $K_1^{S_H}=\alt(S_H)$ by Lemma~\ref{hpasubgroups4}. Otherwise, $\aut(\mathcal{A}_U)=\aut(\mathcal{A}_H)\times \aut(\mathcal{A}_{U\cap D})$. Let $K_H\leq \sym(H)$ be such that $K_H\cong \alt(S_H)$. The group $K_H$ is $2$-equivalent to $\aut(\mathcal{T}_H)$ (see Remark~\ref{2minsmallrem}) and hence $K_1=K_H\times \aut(\mathcal{A}_{U\cap D})$ is $2$-equivalent to $\aut(\mathcal{A}_U)$. By the definition of $K_1$, 
$$K_1^{S_H}=(K_H\times \aut(\mathcal{A}_{U\cap D}))^{S_H}=(K_H)^{S_H}=\alt(S_H).$$

The groups $S_H$ and $D/L$ are $\mathcal{A}_{G/L}$-subgroups. If $\mathcal{A}_{G/L}\neq \mathcal{A}_{S_H}\otimes \mathcal{A}_{D/L}$, then put $K_0=\aut(\mathcal{A}_{G/L})$; otherwise put $K_0=\alt(S_H)\times \aut(\mathcal{A}_{D/L})$. The argument which is the similar to the argument from the previous paragraph implies that $K_0$ is $2$-equivalent to $\aut(\mathcal{A}_{G/L})$ and $K_0^{S_H}=\alt(S_H)=K_1^{S_H}$ as required. 
\end{proof}

Let $K_1$ and $K_0$ are defined as in Lemma~\ref{hrest}.

\begin{lemm}\label{drest}
With the above notation, $K_1^{S_D}=K_0^{S_D}\in \mathcal{K}^{\min}(\mathcal{A}_{S_D})$ unless $\mathcal{A}$ is the $(H\times L)/L$- or $U/(U\cap D)$-wreath product. 
\end{lemm}

\begin{proof}
By Lemma~\ref{leungman} and Lemma~\ref{s0structure}, the $S$-ring $\mathcal{A}_{S_D}$ is normal or trivial over a group of order~$p$. If $\mathcal{A}_{S_D}$ is normal, then $\mathcal{A}_{S_D}$ is $2$-minimal by Lemma~\ref{2minnorm}. If $|S_D|=3$, then $\mathcal{A}_{S_D}$ is obviously $2$-minimal. So $K_1^{S_D}=K_0^{S_D}=\aut(\mathcal{A}_{S_D})\in \mathcal{K}^{\min}(\mathcal{A}_{S_D})$ in both these cases as required. Further, we assume that 
$$|S_D|=p\geq 5~\text{and}~\mathcal{A}_{S_D}=\mathcal{T}_{S_D}.$$
This assumption implies that $|(U\cap D):L|=p$.

If $\aut(\mathcal{A}_U)^{S_D}\leq \Hol(S_D)$ and $\aut(\mathcal{A}_{G/L})^{S_D}\leq \Hol(S_D)$, then $K_1^{S_D}\leq \Hol(S_D)$ and $K_0^{S_D}\leq \Hol(S_D)$. Since $\Hol(S_D)\in \mathcal{K}^{\min}(\mathcal{A}_{S_D})$, we obtain $K_1^{S_D}=K_0^{S_D}=\Hol(S_D)$ as desired. 

If $\aut(\mathcal{A}_U)^{S_D}\nleq \Hol(S_D)$ or $\aut(\mathcal{A}_{G/L})^{S_D}\nleq \Hol(S_D)$, then $\mathcal{A}_U$ is the $(H\times L)/L$-wreath product or $\mathcal{A}_{G/L}$ is the $S/S_D$-wreath product, respectively, by Lemma~\ref{autrestprime}. In the former case, $\mathcal{A}$ is the $(H\times L)/L$-wreath product, whereas in the latter one, $\mathcal{A}$ is the $U/(U\cap D)$-wreath product and we are done. 
\end{proof}

Return to the proof of Theorem~\ref{main1}. By Lemma~\ref{hrest}, we have $K_1^{S_H}=K_0^{S_H}\in \mathcal{K}^{\min}(\mathcal{A}_{S_H})$. If $\mathcal{A}$ is not the $T$-wreath product, where $T=U_1/L_1\in\{(H\times L)/L,U/(U\cap D)\}$, then Lemma~\ref{drest} implies that $K_1^{S_D}=K_0^{S_D}\in \mathcal{K}^{\min}(\mathcal{A}_{S_D})$. Therefore $K_1^S=K_0^S$ by Lemma~\ref{tensgroup}. Thus, $\mathcal{A}$ is schurian by Lemma~\ref{schurwr}. 

Suppose that $\mathcal{A}$ is the $T$-wreath product. Clearly, $|T|=|H|=4$ and Statement~$(4)$ of Theorem~\ref{e4cn} holds for~$\mathcal{A}$ and~$T$. By Lemma~\ref{hrest}, there are $M_1\leq \aut(\mathcal{A}_{U_1})$ and $M_0\leq \aut(\mathcal{A}_{G/L_1})$ such that $M_1\approx_2 \aut(\mathcal{A}_{U_1})$, $M_0\approx_2 \aut(\mathcal{A}_{G/L_1})$, and $M_1^T=M_0^T$. Again, $\mathcal{A}$ is schurian by Lemma~\ref{schurwr}.

\section{Proof of Theorem~\ref{main2}}

Since a section of a Schur group is a Schur group, to prove Theorem~\ref{main2}, it suffices to prove that the group 
$$G\cong C_8\times C_2\times C_p,$$ 
where $p$ is an odd prime, is not a Schur group. Let $A$, $B$, and $C$ be subgroups of $G$ such that
$$A\cong C_8,~B\cong C_2,~P\cong C_p,~\text{and}~G=A\times B\times P.$$
Denote generators of $A$ and $B$ by $a$ and $b$, respectively. Put 
$$A_1=\langle a^2 \rangle\cong C_4,~A_2=\langle a^4 \rangle\cong C_2,~U_1=A_1\times P\cong C_{4p},~\text{and}~U=A_1\times B\times P=B\times U_1\cong C_2\times C_{4p}.$$ 

Let us construct a nonschurian $S$-ring over $G$. The group $\aut(P)$ is a cyclic group of order~$p-1$. Let $M_0$ be the subgroup of $\aut(P)$ of index~$2$. The canonical epimorphism from $\aut(P)$ to $\aut(P)/M_0$ is denoted by $\pi$. Since $\aut(A_1)$ is a group of order~$2$, there exists a unique isomorphism $\varphi$ from $\aut(A_1)\cong C_2$ to $\aut(P)/M_0\cong C_2$. Put
$$\mathcal{A}_1=\cyc(M_1,U_1),$$
where $M_1=\{(\delta,\tau)\in \aut(A_1)\times \aut(P):~\delta^\varphi=\tau^\pi\}\leq \aut(U_1)$, and
$$\mathcal{A}_2=\mathcal{T}_{B}\otimes \mathcal{T}_P.$$
One can see that $P$ is an $\mathcal{A}_1$- and $\mathcal{A}_2$-section and $(\mathcal{A}_1)_P=(\mathcal{A}_2)_P=\mathcal{T}_P$. Therefore one can form the $S$-ring 
$$\mathcal{A}_{12}=\mathcal{A}_1\wr_{U_1/A_1} \mathcal{A}_2$$
over $U$. 

Put
$$\mathcal{A}_3=\cyc(\langle \delta_0 \rangle,A),$$
where $\delta_0\in \aut(A)$ is such that $a^{\delta_0}=a^{-1}$. It is a straightforward to check that the partition of $(A/A_2)\times B\cong C_4\times C_2$ into the sets
$$\{A_2\},~\{A_2a^2\},~\{A_2ab\},~\{A_2a^3b\},~\{A_2a,A_2a^3\},~\{A_2b,A_2a^2b\}$$
defines the $S$-ring $\mathcal{A}_4$ over $(A/A_2)\times B$ such that
$$\mathcal{A}_4\cong \mathbb{Z}C_4\wr_{C_2}\mathbb{Z}E_4.$$
One can see that $A/A_2$ is an $\mathcal{A}_3$- and $\mathcal{A}_4$-section and $(\mathcal{A}_3)_{A/A_2}=(\mathcal{A}_4)_{A/A_2}\cong \mathbb{Z}C_2\wr\mathbb{Z}C_2$. Therefore one can form the $S$-ring 
$$\mathcal{A}_{34}=\mathcal{A}_3\wr_{A/A_2} \mathcal{A}_4$$
over $A\times B$. 

The section $U/P=A_1\times B$ is an $\mathcal{A}_{12}$- and $\mathcal{A}_{34}$-section and 
$$(\mathcal{A}_{12})_{U/P}=(\mathcal{A}_{34})_{U/P}=(\mathbb{Z}A_2\wr \mathbb{Z}(A_1/A_2))\wr \mathbb{Z}B\cong (\mathbb{Z}C_2\wr \mathbb{Z}C_2)\wr \mathbb{Z}C_2.$$ 
Therefore one can form the $S$-ring
$$\mathcal{A}=\mathcal{A}_{12}\wr_{U/P}\mathcal{A}_{34}$$
over $G$. The basic sets of $\mathcal{A}$ are the following:
$$X_0=\{e\},~X_1=\{a^4\},~X_2=\{a^2,a^6\},~X_3=P^\#,~X_4=a^4P^\#,$$
$$Y_1=a^2P_1\cup a^6P_2,~Y_2=a^2P_2\cup a^6P_1,$$
$$Z_1=A_1b,~Z_2=A_1bP^\#,$$
$$T_1=\{a,a^7\}P,~T_2=\{a^3,a^5\}P,~T_3=\{ab,a^5b\}P,~T_4=\{a^3b,a^7b\}P,$$
where $P_1$ and $P_2$ are the nontrivial orbits of $M_0$ on $P$.

\begin{prop}\label{nonschur}
The $S$-ring $\mathcal{A}$ is nonschurian. 
\end{prop}

\begin{proof}
Assume the contrary. Then 
\begin{equation}\label{schurian}
\mathcal{S}(\mathcal{A})=\orb(K,G),
\end{equation}
where $K=\aut(\mathcal{A})_e$. Let $x_1\in P_1$ and $x_2\in P_2$. Since the elements $bx_1$ and $bx_2$ lie in the same basic set $Z_2$ of $\mathcal{A}$, there exists $f\in K$ such that $(bx_1)^f=bx_2$. From Eqs.~\eqref{aut} and~\eqref{schurian} it follows that
$$(P^\#\setminus \{x_1\})^f=(A_1P^\#x_1\cap P^\#)^f=(Z_2bx_1\cap X_3)^f=Z_2bx_2\cap X_3=A_1P^\#x_2\cap P^\#=P^\#\setminus \{x_2\}.$$
Since $X_3=P^\#$ is an orbit of $K$, we have 
$$x_1^f=x_2.$$
The above equality together with Eqs.~\eqref{aut} and~\eqref{schurian} imply that 
$$((a^2P_1x_1\cup a^6P_2x_1)\cap \{a^2,a^6\})^f=(Y_1x_1\cap X_2)^f=Y_1x_2\cap X_2=(a^2P_1x_2\cup a^6P_2x_2)\cap \{a^2,a^6\}.$$
If $|M_0|$ is even, then $P_1=P_1^{-1}$ and $P_2=P_2^{-1}$ and hence  the left-hand side and the right-hand side of the latter equality are equal to $\{a^2\}$ and $\{a^6\}$, respectively. If $|M_0|$ is odd, then $P_2=P_1^{-1}$ and hence  the left-hand side and the right-hand side of the latter equality are equal to $\{a^6\}$ and $\{a^2\}$, respectively. Since $X_2=\{a^2,a^6\}\in \orb(K,G)$, we have
\begin{equation}\label{a2a6}
(a^2)^f=a^6
\end{equation}
in both cases.

Now by Eqs.~\eqref{aut},~\eqref{schurian}, and~\eqref{a2a6}, we have
\begin{equation}\label{contr}
(Pa)^f=(T_1a^2\cap T_1)^f=T_1a^6\cap T_1=Pa^7.
\end{equation}
On the other hand, since $(bx_1)^f=bx_2$, we have
$$(Pa)^f=(T_3bx_1\cap T_1)^f=T_3bx_2\cap T_1=Pa,$$
a contradiction to Eq.~\eqref{contr}.
\end{proof}

Theorem~\ref{main2} follows from Proposition~\ref{nonschur}.

\section{Proof of Theorem~\ref{main3}}

Let $p$ be an odd prime, $P\cong C_p$, $H\cong E_{16}$, and $G=H\times P$. If $p=3$, then $G$ is a Schur group by the computational results~\cite{Ziv}. Further, we assume that $p\geq 5$. In this case, we are going to construct a nonschurian $S$-ring over $G$ and thereby to prove Theorem~\ref{main3}.

Let $a,b,c,d\in H$ be such that $\langle a,b,c,d \rangle=H$, $A=\langle a \rangle$, $B=\langle b \rangle$, $C=\langle c \rangle$, and let $\delta_0\in \aut(H)$ be such that
$$\delta_0:(a,b,c,d)\mapsto (a,ab,bc,cd),$$
$M_0=\langle \delta_0 \rangle$, and 
$$\mathcal{A}_0=\cyc(M_0,H).$$ 
It is straightforward to verify that $|\delta_0|=4$ and the basic sets of $\mathcal{A}_0$ are the following:
$$\{e\},~\{a\},~Ab,~(A\times B)c,~\{e,b,c,abc\}d,~\{a,ab,ac,bc\}d.$$
One can see that each of two latter basic sets has trivial radical and generates $H$. So $\mathcal{A}_0$ is indecomposable. Every basic set of $\mathcal{A}_0$ has the $2$-power size, i.e. $\mathcal{A}_0$ is a $2$-$S$-ring in terms of~\cite{KR}. Clearly, $\mathcal{A}_0$ is cyclotomic and hence schurian. Therefore $\mathcal{A}_0$ is $2$-minimal by~\cite[Lemma~5.6]{KR} stating that every schurian indecomposable $p$-$S$-ring over an elementary abelian group of order~$p^4$ is $2$-minimal for every prime~$p$. Thus,
\begin{equation}\label{2mink0}
K_0=H_r\rtimes M_0
\end{equation}
for every $K_0\leq \sym(H)$ such that $K_0\geq H_r$ and $K_0\approx_2 \aut(\mathcal{A}_0)$ (recall that $H_r$ is the group of $\sym(H)$ induced by all right multiplications of $H$). In particular, $\aut(\mathcal{A}_0)=H_r\rtimes M_0$.

Let $V=A\times B \times C$ and $U=V\times P$. As $p\geq 5$, the group $M_P=\aut(P)\cong C_{p-1}$ has a nontrivial subgroup~$M_P^0$ of index~$2$. Let $\delta_1,\delta_2\in \aut(V)$ be such that
$$\delta_1:(a,b,c)\mapsto (a,ab,bc),~\delta_2:(a,b,c)\mapsto (a,b,bc).$$
Put $M_V=\langle \delta_1,\delta_2 \rangle$. It can be verified directly that $|\delta_1|=4$, $|\delta_2|=2$, and $\delta_2\delta_1\delta_2=\delta_1^{-1}$. So $M_V\cong D_8$.

Let $M^0_V=\langle \delta_1^2,\delta_2\delta_1 \rangle$. Clearly, $M_V^0\cong E_4$ and hence $M_V/M_V^0\cong C_2$.
Put 
$$M_1=\{(\delta,\tau)\in M_V\times M_P:~(\delta^{\pi_1})^\psi=\tau^{\pi_2}\}\leq \aut(U),$$
where $\pi_1$ and $\pi_2$ are the canonical epimorphisms from $M_V$ to $M_V/M_V^0$ and from $M_P$ to $M_P/M_P^0$, respectively, and $\psi$ is a unique isomorphism from $M_V/M_V^0\cong C_2$ to $M_P/M_P^0\cong C_2$, and 
$$\mathcal{A}_1=\cyc(M_1,U).$$ 
The definition of $M_1$ implies that $V$ and $P$ are $\mathcal{A}_1$-subgroups and the basic sets of $\mathcal{A}_1$ are the following:
$$\{e\},~\{a\},~Ab,~(A\times B)c,~P^\#,~aP^\#,~AbP^\#,$$
$$AcP_1\cup AbcP_2,~AcP_2\cup AbcP_1,$$
where $P_1$ and $P_2$ are the nontrivial orbits of $M_P^0$ on $P$.

One can see that $V$ is an $\mathcal{A}_0$- and $\mathcal{A}_1$-subgroup and $(\mathcal{A}_0)_V=(\mathcal{A}_1)_V\cong (\mathbb{Z}C_2\wr \mathbb{Z}C_2)\wr \mathbb{Z}C_2$. Therefore one can form the $S$-ring
$$\mathcal{A}=\mathcal{A}_1\wr_S \mathcal{A}_0$$
over $G$, where $S=U/P$ and $\mathcal{A}_0$ is considered as an $S$-ring over $G/P\cong H$. 

To complete the proof of Theorem~\ref{main3}, it suffices to prove the proposition below.

\begin{prop}\label{nonschure16}
The $S$-ring $\mathcal{A}$ is nonschurian. 
\end{prop}

\begin{proof}
Assume the contrary that $\mathcal{A}$ is schurian. Then there exist two groups $K_1\leq \sym(U)$ and $K_0\leq \sym(G/P)$ such that 
$$K_1\geq U_r,~K_0\geq (G/P)_r,~K_1\approx_2 \aut(\mathcal{A}_1),~K_0\approx_2 \aut(\mathcal{A}_0),~\text{and}~K_0^S=K_1^S$$
by Lemma~\ref{schurwr}. From Eq.~\eqref{2mink0} it follows that
\begin{equation}\label{2mink0s} 
K_0^S=S_r\rtimes M_0^S.
\end{equation}

Let $X\in \mathcal{S}(\mathcal{A}_U)=\mathcal{S}(\mathcal{A}_1)$ be such that 
$$X=AcP_1\cup AbcP_2$$
and $x\in P_1$. Suppose that $f$ belongs to the one-point stabilizer $M_x$ of $x$ in $M$, where $M=(K_1)_e$ is the one-point stabilizer of $e$ in $K_1$. Since $K_0^S=K_1^S$, Eq.~\eqref{2mink0s} implies that $f^S\in M_0^S$. In view of Eq.~\eqref{aut} and $(A\times B)c,X^{-1}\in \mathcal{S}(\mathcal{A})$, we obtain
$$(Ac)^f=(X^{-1}x\cap (A\times B)c)^f=X^{-1}x\cap (A\times B)c=Ac.$$
Together with $f^S\in M_0^S$ and the definition of $M_0$, this shows that $f^S\in \langle \delta_0^2 \rangle$ and hence $b^f=b$. Since $H$ and $P$ are $\mathcal{A}$-subgroups, $Hx$ and $Pb$ are blocks of $\aut(\mathcal{A})$. As $x^f=x$ and $b^f=b$, we conclude that $(Hx)^f=Hx$ and $(Pb)^f=Pb$. Therefore
$$(bx)^f=(Hx\cap Pb)^f=Hx\cap Pb=bx.$$
Thus, $f$ belongs to the one-point stabilizer $M_{bx}$ of $bx$ in $M$ and consequently
\begin{equation}\label{stabinc} 
M_x\leq M_{bx}.
\end{equation}

The definition of $M_1$ implies that the basic set of $\mathcal{A}$ containing~$x$ is $P^\#$ and the basic set of $\mathcal{A}$ containing~$bx$ is $Y=AbP^\#$. Clearly, $|Y|=2(p-1)$. By the assumption, we have $P^\#,Y\in \orb(M,U)$. Therefore
$$|M_x|=\frac{|M|}{p-1}>\frac{|M|}{2(p-1)}=\frac{|M|}{|Y|}=|M_{bx}|,$$
in contrast to Eq.~\eqref{stabinc}.
\end{proof}

\vspace{5mm}

\noindent \textbf{Acknowledgment:} The author is very grateful to Prof. Ilia Ponomarenko for a careful reading and valuable comments which help to improve the text significantly. The author would like to thank Dr. Matan Ziv-Av for the help with computer calculations.

\end{document}